\documentclass[notitlepage,reqno,11pt]{amsart}
\usepackage{latexsym,amssymb, epsfig, amsmath,amsfonts, subfigure,amsthm}

\usepackage{rotating}
\usepackage[toc,page]{appendix}
\usepackage{color}
\usepackage{multirow}
\usepackage{relsize}
\usepackage{microtype}
\usepackage[normalem]{ulem}
\usepackage[foot]{amsaddr}

\usepackage[colorlinks, allcolors=blue]{hyperref}

\usepackage[margin=1in]{geometry}

\numberwithin{equation}{section}
 \newtheorem{assumption}{Assumption}[section]
\newtheorem{lemma}{Lemma}[section]
\newtheorem{theorem}{Theorem}[section]
\newtheorem{definition}{Definition}[section]

\newtheorem{prop}{Proposition}[section]
\newtheorem{remark}{Remark}[section]

\newtheorem{example}{Example}[section]

\newlength{\defbaselineskip}
\newcommand{\setlinespacing}[1]%
           {\setlength{\baselineskip}{#1 \defbaselineskip}}

\newcommand{\RR}{{\mathbb R}}
\newcommand{\ZZ}{{\mathbb Z}}
\newcommand{\NN}{{\mathbb N}}

\def\E{\mathbb{E}}
\def\P{\mathbb{P}}

\newcommand{\sF}{{\mathcal{F}}}

\newcommand{\beql}[1]{\begin{equation}\label{#1}}
\newcommand{\eeq}{\end{equation}}

\newcommand{\beqal}[1]{\begin{eqnarray}\label{#1}}
\newcommand{\eeqa}{\end{eqnarray}}
\newcommand{\beq}{\begin{displaymath}}
\newcommand{\eeqno}{\end{displaymath}}
\newcommand{\bali}[1]{\begin{align}\label{#1}}
\newcommand{\eali}{\begin{align}}
\newcommand{\balino}{\begin{align*}}
\newcommand{\ealino}{\begin{align*}}

\newcommand{\ep}{\epsilon}

\newcommand{\Var}{\text{\rm Var}}
\newcommand{\Cov}{\text{\rm Cov}}
\newcommand{\wt}{\widetilde}

\newcommand{\R}  {\mathbb{R}}

\newcommand{\bD}{{\mathbf D}}

\newcommand{\bC}{{\mathbf C}}

\newcommand{\bone}{{\mathbf 1}}

\newcommand{\qandq}{\quad\mbox{and}\quad}

\newcommand{\qforq}{\quad\mbox{for}\quad}

\newcommand{\qasq}{\quad\mbox{as}\quad}
\newcommand{\qinq}{\quad\mbox{in}\quad}

\newcommand{\non}{\nonumber}
\newcommand{\RA}{\Rightarrow}

\newcommand{\baa}{\begin{eqnarray*}}
\newcommand{\eaa}{\end{eqnarray*}}

\newcommand{\ttl}{\Large Functional central limit theorems for epidemic models \\[5pt]
with varying infectivity}

\begin{document}

\title[]{\ttl}

\author[Guodong \ Pang]{Guodong Pang}
\address{Department of Computational Applied Mathematics and Operations Research,
George R. Brown College of Engineering,
Rice University,
Houston, TX 77005}
\email{gdpang@rice.edu}

\author[{\'E}tienne \ Pardoux]{{\'E}tienne Pardoux}
\address{Aix Marseille Univ, CNRS, I2M, Marseille, France}
\email{etienne.pardoux@univ.amu.fr}

\begin{abstract} 
In this paper, we prove a functional central limit theorem (FCLT) for a stochastic epidemic model with varying infectivity and general infectious periods recently introduced in \cite{FPP2020b}. 
The infectivity process  (total force of infection at each time) is composed of the independent infectivity random functions of each infectious individual, which starts at the time of infection. 
These infectivity random functions induce the infectious periods (as well as exposed, recovered or immune periods in full generality), whose probability distributions can be very general. 
The epidemic model includes the generalized non--Markovian SIR, SEIR, SIS, SIRS models with infection-age dependent infectivity. 
In the FCLTs for the generalized SIR and SEIR models,
 the limits of the diffusion-scaled fluctuations of the infectivity and susceptible processes are a unique solution to a two-dimensional Gaussian-driven stochastic Volterra integral equations, and then given these solutions, the limits for the infected (exposed/infectious) and recovered processes are Gaussian processes expressed in terms of the solutions to those stochastic Volterra integral equations. 
We also present the FCLTs for the generalized SIS and SIRS models. 
\end{abstract}


\maketitle

\allowdisplaybreaks

\section{Introduction}

It has been observed in recent studies of the Covid-19 pandemic (see e.g. \cite{he2020temporal}) that the infectivity of infectious 
individuals decreases from the epoch of symptom first appearing to full recovery. The varying infectivity 
characteristics also appears in many other epidemic diseases \cite{KMK,BCF-2019}. We have presented a 
stochastic epidemic model with varying infectivity in \cite{FPP2020b}, where the various individuals have i.i.d. 
infectivity random functions, and the total force of infection at each time is the aggregate infectivity of all the 
individuals that are currently infectious.  The i.i.d. assumption on the infectivity functions can be justified since we model the spread of the same disease and the infectivity is largely determined by viral load progression.
We have proved a functional law of large numbers  (FLLN) for the epidemic dynamics which results in a 
deterministic epidemic model, which is the model described as an ``age-of-infection epidemic model" in 
\cite{KMK,brauer2008age,BCF-2019}.
For the early stage, we have studied the stochastic model directly starting with a small number of infectious individuals, and proved that the epidemic grows at an exponential rate on the event of non-extinction 
 using an approximation by a non Markovian branching process.  
In addition, we have deduced the initial basic reproduction number $R_0$ from the limit process and computed its value for the case of the early phase of the Covid-19 epidemic in France. We have concluded 
a decreased value of $R_0$ induced by the decrease of the infectivity when the age of infection increases.

In this paper, we study the stochastic fluctuations of the evolution dynamics around the deterministic limits for the stochastic epidemic models with varying infectivity. 
The infectivity random functions can take a very general form (see Assumption \ref{AS-lambda}) and start with a value zero for a period of time, which then in turn determines the durations of the exposed and infectious periods. The joint distribution of these two periods is determined by the law of the random function, and can be very general. Clearly, this also includes the special case with only infectious periods, where the random functions do not start from zero. 
We start with the generalized SIR model, where ``I" represents the ``infected", including either the case of only infectious periods, or the case of both exposed and infectious periods. 
We then study the generalized SEIR model, where ``E" and ``I" represent the ``exposed" and ``infectious" periods, which can be regarded as a more detailed model in the second above scenario by considering the separate compartments. 
Although it may appear this is a special case of the generalized SIR model, in many settings, one would be interested in the dynamics of the detailed ``E" and ``I" compartments, for example, the infectious population in isolation or hospitalization.
Moreover, mathematically, for the FCLTs, 
one can deduce the driving Gaussian processes for the limits of the SIR model from those of the SEIR model, but not the other way around
(see also Remarks \ref{rem-I-expression} and \ref{rem-comparison}).   
Therefore, for the sake of readability, we first describe the generalized SIR model and state the limit theorems, 
and then do the same for the generalized SEIR model, while we only provide the detailed proofs for the generalized SEIR model. Would we provide the proofs for the SIR case only, that would not simplify significantly the arguments. 

For both models, in the FLLNs, we study the mean  total infectivity process jointly with the proportions of the compartment counting processes. 
 In particular, the infectivity and susceptible functions in the limit are uniquely determined by a two-dimensional Volterra integral equation, and given these two functions, the other compartment limits are given by Volterra integral formulas. 
In the FCLTs,  we prove the joint convergence of the diffusion-scaled infectivity process with the compartment counting processes. 
In the limit, the diffusion-scaled infectivity and susceptible processes is determined by a two-dimensional Gaussian-driven linear stochastic Volterra integral equation. 
In particular, the diffusion-scaled instantaneous infectivity rate process has a limit that is a linear functional of the susceptible and infectivity limit processes. 
Given these,  the limits of the other compartment counting processes are expressed in terms of the solution of the above stochastic Volterra integral equation.   These results extend the FCLT for the classical non-Markovian SIR and SEIR models in \cite{PP-2020}.

The main challenge in the proof of the FCLT lies in  the convergence of the aggregate infectivity process. We allow
the individual infectivity random functions to be piecewise continuous with a finite number of discontinuities as stated in Assumption \ref{AS-lambda}, which covers many practical settings, see Examples \ref{ex-1}--\ref{ex-2}. 
We use Poisson random measures (PRMs) induced by the laws of these random functions in the functional space $\bD$, and take advantage of some useful properties of stochastic integrals with respect to 
the corresponding compensated PRMs.
 We first give a useful decomposition of this process, and construct two auxiliary processes by replacing the random instantaneous infectivity rate process by its deterministic limit function in the FLLN. 
For these auxiliary processes, we employ the moment method to prove their tightness, using the criterion for tightness in  \cite[Theorem 13.5]{billingsley1999convergence}. This condition is more precisely (13.14) from \cite{billingsley1999convergence}, and is \eqref{moment3points} in the Appendix. This, together with the convergence of finite dimensional distributions,   proves their weak convergence. 
The martingale approach employed in \cite{PP-2020} (see the proof of Lemma~\ref{lem-3.12}) could not be used for the aggregate infectivity process. It seems hard to construct appropriate martingales for the model considered in the present paper. 

A major difficulty in the proof of tightness lies in the estimation of the expectation of the supremum of the square of the diffusion-scaled processes (see Lemma \ref{lem-hatPhi-2-boundedness} and equation \eqref{eqn-hatPhi-2-boundedness-sup-inside}), since we cannot use a semimartingale decomposition and exploit Doob's maximal inequality. In order to establish this, we have developed a new criterion to interchange $\sup_t$ and the expectation of the square of a stochastic process, provided that the process satisfies the well--known tightness criterion using the increments at three time points: see equation \eqref{3pts} as in Theorem \ref{13.5} (which is a variant of Theorem 13.5 of \cite{billingsley1999convergence}), as well as the sufficient moment criterion in \eqref{3pts-moments}).   This result is stated in Theorem \ref{th:momentsup}. It should be of independent interest and we expect that it will be useful for other purposes. For our model, we employ the moment formulas (specifically the mixed second moment formula in \eqref{mixed2moment}) for the stochastic integrals with respect to the PRMs to establish the tightness moment bound for the increments of the infectivity process (see Propositions \ref{prop-wt-mfI-1-inc-moment} and \ref{prop-wt-mfI-2-inc-moment}).

 In addition, another challenge is to prove the joint convergence of the aggregate infectivity processes and the compartment counting processes. 
 We have developed an approach using a common PRM to represent the subprocesses associated with the newly infected individuals in the expressions of the infectivity and compartment processes, and then carried out calculations with characteristic functions of these processes by making use of the properties of PRMs (See the proof of Lemma \ref{lem-3.12}).

\smallskip

We also state the FCLTs without proofs for the generalized SIS and SIRS models with varying infectivity (the proofs  follow with slight modifications). 
For the SIS model, the epidemic dynamics is determined by the aggregate infectivity process and the infectious process, whose limits are given by a two-dimensional Gaussian-driven linear stochastic Volterra integral equation (Theorem \ref{thm-FCLT-SIS}). For the SIRS model, the epidemic dynamics is determined by the aggregate infectivity process and the infectious and recovered/immune processes, whose limits are given by a three-dimensional Gaussian-driven linear stochastic Volterra integral equation (Theorem \ref{thm-FCLT-SIRS}).

\smallskip 

This work contributes to the literature of stochastic epidemic models in the aspects of infection-age dependent infectivity, and general infectious periods. It establishes a useful result describing the fluctuations of the stochastic individual based model around its law of large numbers limit. 
The existing work on epidemic models with infection-age dependent (varying) infectivity has all been about the deterministic models, pioneered by Kermack and McKendrick  in 1927 \cite{KMK}
and more recently proposed by  Brauer \cite{brauer2008age} (see also  \cite[Chapter 4.5]{BCF-2019}), and the PDE models (see, e.g., \cite{hoppensteadt1974age,thieme1993may,inaba2004mathematical,magal2013two}). 
These were not considered as the FLLN limits of a well specified stochastic model either, except our work in \cite{FPP2020b}. 
For non-Markovian epidemic models with general infectious periods, although some deterministic models (including Volterra integral equations) appeared in the literature (see, e.g., \cite[Chapter 4.5]{BCF-2019} and references therein), the works that rigorously establish them as a FLLN from a stochastic model are limited. 
Both FLLNs and FCLTs are proved in \cite{wang1975limit, wang1977central,wang1977gaussian}  for population models, including an epidemic SIR model which has an infection rate dependent on the number of infectious individuals and general infectious periods for newly infected individuals, and which also has the remaining infectious periods of the initially infected individuals dependent on the elapsed infectious durations.  
A FLLN for the associated measure-valued processes was established  for the SIR model with the infection rate dependent on time and  the number of infected individuals using Stein's method in \cite{reinert1995asymptotic}.   
 For the general SIS, SIR, SEIR and SIRS models with a constant infection rate and the associated multi-patch models,  both FLLNs and FCLTs are recently established in \cite{PP-2020,PP-2020b}. 
 The results in  \cite{PP-2020}  are similar to those in  \cite{wang1975limit, wang1977central, wang1977gaussian} for the SIR  model with general infectious periods, since the limits are given as Volterra integral equations, deterministic in the FLLNs and stochastic driven by Gaussian processes in the FCLTs. 
 However, the proofs are rather different, and the FCLTs in \cite{PP-2020} do not require any condition on the c.d.f. of the infectious periods, while those in  \cite{wang1975limit, wang1977central, wang1977gaussian} assume a $\bC^1$ condition. 
 As mentioned at the beginning, the model with varying infectivity studied in this paper is much more general than that in \cite{PP-2020}. It is also quite different from the models in \cite{wang1975limit, wang1977central, wang1977gaussian, reinert1995asymptotic} since an infection rate dependent on the number of infectious individuals  captures neither the randomness of infectivity as the i.i.d. random infectivity function of each individual, nor the age of infection dependent infectivity of each individual.  
  Distinct from \cite{PP-2020}, this work focuses on the convergence of the aggregate infectivity process and its joint convergence with the compartment counting processes, which is much more involved as discussed above. 
For other aspects of epidemic models with general infectious periods, including  the Sellke construction and the final size of an epidemic,   we refer the readers to \cite{sellke1983asymptotic, ball1986unified,ball1993final,barbour1975duration} and  the recent survey \cite{britton2018stochastic}.

\smallskip

The paper is organized as follows. 
In Section~\ref{sec-SIRmodel}, we first describe the generalized SIR model with varying infectivity and state the limit theorems.  Here the infected periods possibly include the combined exposed and infectious durations.
We then describe a generalized SEIR model with separate exposed and infectious 
compartments in Section \ref{sec-model}, and state the limit theorems for this model.
We provide a detailed proof for the FCLT of the generalized SEIR model in Sections \ref{sec-Proofs} and \ref{sec-EIR-proof}. 
We first state several technical preliminaries in Subsection \ref{sec-proofs-tech}, including the criterion for moment estimate of the supremum of stochastic processes in $\bD$, the formulas of the Laplace functional and moments of stochastic integrals with respect to PRMs, and an estimate of the increments of the infectivity function. We then proceed to the proofs of the convergence of the number of susceptible individuals and the aggregate infectivity process in Subsections \ref{sec-rep-roadmap}-\ref{sec-thm21-proof}, with a proof roadmap given in Subsection \ref{sec-rep-roadmap}. The joint convergence of the exposed, infectious and recovered processes together with the two processes mentioned above are given in Section \ref{sec-EIR-proof}. 
 Finally, the FCLTs for the generalized SIR and SIRS models are stated in Sections \ref{sec-SIS} and \ref{sec-SIRS}, respectively.

 \subsection{Notation} 
 Throughout the paper, $\NN$ denotes the set of natural numbers, and $\R^k (\R^k_+)$ denotes the space of $k$-dimensional vectors
with  real (nonnegative) coordinates, with $\R (\R_+)$ for $k=1$.  For $x,y \in\R$, denote $x\wedge y = \min\{x,y\}$ and $x\vee y = \max\{x,y\}$. 
Let  $\bD=\bD([0,+\infty), \R)$ denote the space of $\R$--valued c{\`a}dl{\`a}g functions defined on $[0,+\infty)$. Throughout the paper, convergence in $\bD$ means convergence in the  Skorohod $J_1$ topology, see chapter 3 of \cite{billingsley1999convergence}. 
 Also, $\bD^k$ stands for the $k$-fold product equipped with the product topology.   We use $\bC$ to denote be the subset of $\bD$ consisting of continuous functions and  $\bC^1$ the subset of differentiable functions whose derivative is continuous.

 All random variables and processes are defined on a common complete probability space $(\Omega, \mathcal{F}, \P)$. The notation $\RA$ means convergence in distribution.  The notation $x \lesssim y$ means that there exists some constant $c$ such that $x \le cy$ for $x,y \in \RR$.  We use $\bone(\cdot)$ or $\bone_{\{\cdot\}}$ for indicator function.

We write $F(t) = \int_0^t F(ds)$ for a cumulative distribution function (c.d.f.) $F$ on $\R_+$.  
 For any measure $\mu$ on $\R$ and $f$ a measurable and $\mu$--integrable function, the integral 
 $\int_a^bf(t)\mu(dt)$ will mean $\int_{(a,b]}f(t)\mu(dt)$.

We use the following conventional notation in the paper: for the scaled processes and quantities (all indexed by the population size $N$) and their limits in the FLLN, we use an upper bar, and for those in the FCLT, we use a hat. 
 This means that for any process $X^N$, we let $\bar{X}^N=N^{-1} X^N$, and $\hat{X}^N:=\sqrt{N} (\bar{X}^N- \bar{X})$.
For the processes and quantities associated with the initial quantities, we use a super index `0'.  The letters $S, E, I, R$ with different indices (e.g., index `0' for the initial quantities) represent the Susceptible, Exposed, Infectious and Recovered compartments.  In the case of the $S I R$ model, $I$ stands for Infected.

\section{The models and Main Results}\label{sec-main}

\subsection{Generalized SIR model with varying infectivity}\label{sec-SIRmodel}

In this subsection, $S$ stands  for the compartment of susceptible individuals, $I$ for infected (either exposed, and not yet infectious, or infectious) and $R$ for Recovered. Those who are recovered have a permanent immunity, 
and do not play a role in the epidemic anymore. Those who died from the disease are placed in that $R$ compartment. 

Let the population size be $N$, and $S^N(t), I^N(t), R^N(t)$  be the numbers of susceptible, infected and recovered individuals at each time $t$, respectively.  We have the balance equation $N= S^N(t) +
 I^N(t) + R^N(t)$ for $t\ge 0$. We assume that $R^N(0)=0$, $S^N(0)>0$ and $I^N(0) >0$.  If there are recovered individuals at time $0$, we can suppress them from the population, since they play no role in the epidemic. 
 
 An individual who gets infected (i.e., jumps from the $S$ to the $I$ compartment) at time $t>0$ will have at time $s>t$ an infection age infectivity $\lambda(s-t)$.
 The function $\lambda$ is random, in the sense that each individual experiences a specific realization of that 
 random function (due to the specificity of each individual's immune system reaction to the illness). We assume that the realizations associated to the various individuals are i.i.d. 
 Let 
\begin{align} \label{eqn-chi}
\chi = \inf\{s>0,\ \lambda(r)=0, \text{ for all }r>s\}.
\end{align}
 The time interval $[t,t+\chi)$ is the time interval during which the individual is infected (including either exposed or infectious). 
At time $t+\chi$ the individual jumps from the $I$ into the $R$ compartment. 
 Let $\{\lambda_i(t),\ t\ge0\}$ denote the infectivity function of the $i$--th infected individual after time $0$. We assume that the collection of random functions $\{\lambda_i\}_{i\ge1}$ is i.i.d., and hence, the corresponding infected periods $\{\chi_i\}_{i\ge 1}$ are also i.i.d.

 The $I^N(0)$ individuals who are infected at time $0$ are thought of as having been infected in the past (at some time $t<0$).  Therefore their random infectivity functions $\{\lambda^0_j(t),\ t\ge0\}_{1\le j\le I^N(0)}$, starting at time $t=0$, are i.i.d., and their common law differs for that of $\lambda_i$. We also assume that
 the random functions $\lambda^0_1,\ldots,\lambda^0_{I^N(0)},\lambda_1,\lambda_2,\ldots$ are independent.

 Also, let us define, for $j=1,\dots,I^N(0)$, 
\[\chi^0_j=\inf\{t>0,\ \lambda^0_j(s)=0,\ \forall s>t\}\,, \]
which is the remaining infected period for the $j$-th initially infected individual. By assumption, the collections of variables $\{\chi^0_j\}$ and $\{\chi_i\}$ are also independent.  
Let $\Phi_0$ and $\Phi$ be the c.d.f.'s for the variables $\{\chi^0_j\}$ and $\{\chi_i\}$, respectively. Also let $\Phi_0^c =1- \Phi_0$ and $\Phi^c=1-\Phi$. 

 \begin{remark}
 One possible choice for $\{\lambda^0_j,\ 1\le j\le I^N(0)\}$ is as follows. Suppose we are given a sequence of i.i.d.
 random functions $\{\lambda_i\}_{i\in\ZZ}$. As above, $\lambda_i(t)$, for $i\ge1$, describes the infectivity function of the $i$--th infected individual after time $0$. To each $1\le j\le I^N(0)$, we associate the random function
 $\lambda_{-j}(t)$, and a positive random variable $\varkappa_j$, which represents the infection age at time $0$ of the initially infected individual $j$, which we assume to satisfy  $\varkappa_j<\chi_{-j}$. For $1\le j\le I^N(0)$, we let
 $\lambda^0_j(t)=\lambda_{-j}(\varkappa_j+t)$.
 \end{remark}

 For $t>0$, let $A^N(t)$ denote the number of individuals who have been infected on the time interval $(0,t]$. This process can be expressed as
 \begin{align} \label{eqn-An-rep-1}
A^N(t)=\int_0^t\int_0^\infty{\bf1}_{u\le \Upsilon^N(s^-)}Q(ds,du)\,, \quad t \ge 0, 
\end{align}
where 
\begin{align} \label{eqn-Phi}
\Upsilon^N(t) := \frac{S^N(t)}{N} \mathfrak{I}^N(t)\,, 
\end{align}
is the instantaneous infection rate function at time $t$, and 
$Q$ is a standard Poisson random measure (PRM)\footnote{$Q$ is a sum of Dirac measures
at countably many points in $\R^2_+$, in such a way that if $A_1,\ldots,A_k$ are disjoint Borel subsets of  $\R^2_+$,
the number of those points are mutually independent, while the number of points in a Borel set $A\subset\R^2_+$
follows the Poisson distribution with parameter the Lebesgue measure of $A$. } on $\R^2_+$.
In the formula for $\Upsilon^N(t)$, $\mathfrak{I}^N(t)$ is the total force of infection which is exerted on the susceptibles at time $t$.
If $\{\tau^N_i\}_{i\ge1}$ denote the successive times of infection (i.e., the jump times of the counting process 
 $A^N(t)$),  
 \begin{align} \label{eq:FiN}
\mathfrak{I}^N(t) = \sum_{j=1}^{I^N(0)}\lambda^{0}_j(t) 
+\sum_{i=1}^{A^N(t)}\lambda_i (t-\tau^N_i) \,, \quad t \ge 0.
\end{align}
 
Clearly, 
\begin{equation}\label{eq:SN}
S^N(t)=S^N(0)-A^N(t)\,.
\end{equation}

We can now precise the evolution of the numbers of infected and recovered individuals. 
Now the numbers of infected and recovered individuals at time $t$,  $I^N(t)$ and $R^N(t)$, are given by
\begin{align}\label{eq:I+RN}
I^N(t)&=\sum_{j=1}^{I^N(0)}{\bf1}_{t<\chi^0_j}+\sum_{i=1}^{A^N(t)}{\bf1}_{t<\tau^N_i+\chi_i},\\
R^N(t)&=\sum_{j=1}^{I^N(0)}{\bf1}_{t\ge\chi^0_j}+\sum_{i=1}^{A^N(t)}{\bf1}_{t\ge\tau^N_i+\chi_i}\,.
\end{align}

We also make the following assumptions on the random infectivity functions. See the discussions on the assumptions in Section \ref{sec-on-lambda}.

\begin{assumption} \label{AS-lambda-LLN}
The random functions $\lambda(t)$ (resp. $\lambda^0(t)$), of which $\lambda_1(t), \lambda_2(t),\ldots$ (resp. $\lambda^0_1(t), \lambda^0_2(t),\\ \ldots$) are i.i.d. copies, satisfying the following properties.

\begin{itemize}
\item[(i)] There exists a constant $\lambda^*<\infty$ such that $\sup_{t \in [0,T]} \max\{\lambda^0(t), \lambda(t)\} \le \lambda^*$ almost surely. 

\item[(ii)] 
There exist a given number $k\ge1$, a random sequence 
 $0=\xi^0<\xi^1<\cdots<\xi^k=+\infty$
and random functions $\lambda^j\in \bC$, $1\le j\le k$ such that 
\begin{equation} \label{eqn-lambda}
 \lambda(t)=\sum_{j=1}^k\lambda^j(t){\bf1}_{[\xi^{j-1},\xi^j)}(t)\,.
\end{equation}
In addition, for any $T>0$, there exists a deterministic nondecreasing function $\varphi_T \in \bC$ with $\varphi_T(0)=0$ such that $|\lambda^j(t)-\lambda^j(s)|\le \varphi_T(t-s)$ almost surely, for all $0 \le t, s \le T$, $1 \le j \le k$. 
\end{itemize}
 \end{assumption}

In \cite{FPP2020b}, the following functional law of large numbers (FLLN) is established for the renormalized quantities 
\[ \big(\bar{S}^N,\bar{\mathfrak{I}}^N,\bar{I}^N,\bar{R}^N\big):=N^{-1}\big(S^N,\mathfrak{I}^N,I^N,R^N\big)\,.\]
Let $\bar{\lambda}^0(t) =\E[\lambda^0(t)]$ and $\bar{\lambda}(t) =\E[\lambda(t)]$ for $t\ge 0$.

\begin{theorem}\label{th:LLN-SIR}
If $\bar{I}^N(0) \to \bar{I}(0) \in (0,1)$,  under
 Assumption \ref{AS-lambda-LLN}, we have 
\begin{align}
\big(\bar{S}^N,\bar{\mathfrak{I}}^N,\bar{I}^N,\bar{R}^N\big)\to\big(\bar{S},\bar{\mathfrak{I}},\bar{I},\bar{R}\big)
\quad \text{in} \quad \bD^4 \quad \text{as} \quad N \to \infty,
\end{align}
in probability, where 
\begin{align*}
\bar{S}(t)&=\bar{S}(0)-\int_0^t\bar{S}(s)\bar{I}(s)ds,\\
\bar{\mathfrak{I}}(t)&=\bar{I}(0)\bar{\lambda}^0(t)+\int_0^t\bar{\lambda}(t-s)\bar{S}(s)\bar{\mathfrak{I}}(s)ds,\\
\bar{I}(t)&=\bar{I}(0)\Phi_0^c(t)+\int_0^t\Phi^c(t-s)\bar{S}(s)\bar{\mathfrak{I}}(s)ds,\\
\bar{R}(t)&=\bar{I}(0)\Phi_0(t)+\int_0^t\Phi(t-s)\bar{S}(s)\bar{\mathfrak{I}}(s)ds\,.
\end{align*}
\end{theorem}

The purpose of the present paper is to establish a functional central limit theorem (FCLT) for the fluctuations of the stochastic sequence around its deterministic limit. We define the renormalized differences
\begin{align}
\big(\hat{S}^N,\hat{\mathfrak{I}}^N,\hat{I}^N,\hat{R}^N\big):=
\sqrt{N}\big(\bar{S}-\bar{S}^N,\bar{\mathfrak{I}}-\bar{\mathfrak{I}}^N,\bar{I}-\bar{I}^N,\bar{R}-\bar{R}^N\big) \,.
\end{align}
We make the following assumption on the initial conditions.
\begin{assumption} \label{AS-FCLT-SIR}
There exist a deterministic constant $\bar{I}(0) \in (0,1)$ satisfying  and a random variable  $\hat{I}(0)$ such that   $ \hat{I}^N(0) := \sqrt{N} ( \bar{I}^N(0) - \bar{I}(0)) \Rightarrow \hat{I}(0) $ as $N \to \infty$  and  $\sup_N\E[\hat{I}^N(0)^2]<\infty$, and thus   $\E[\hat{I}(0)^2]<\infty$. 
\end{assumption}

Before we state the FCLT for the generalized SIR model, let us state the following definition.
\begin{definition}\label{def-W-SIR}
Let $(\hat{W}_S,\hat{W}_{\mathfrak{I}},\hat{W}_I, \hat{W}_R)$ is a four-dimensional centered Gaussian process, independent of $\hat{I}(0)$, with covariance functions: for $t, t' \ge 0$, 
\begin{align*}
\Cov(\hat{W}_S(t),\hat{W}_S(t'))&=\int_0^{t\wedge t'}\bar{S}(s) \bar{\mathfrak{I}}(s)ds,\\
\Cov( \hat{W}_{\mathfrak{I}}(t), \hat{W}_{\mathfrak{I}}(t')) & =  \bar{I}(0) \Cov(\lambda^{0}(t), \lambda^{0}(t'))  \\
& \quad +  \int_0^{t\wedge t'} \Big( \Cov(\lambda(t-s), \lambda(t'-s))+  \bar{\lambda}(t-s) \bar{\lambda}(t'-s)\Big) \bar{S}(s)\bar{\mathfrak{I}}(s)  ds\,, \\
\Cov(\hat{W}_I(t), \hat{W}_I(t')) &= \bar{I}(0) (\Phi_{0}^c(t\vee t') - \Phi_{0}^c(t) \Phi_{0}^c(t')) +   \int_0^{t\wedge t'} \Phi^c(t\vee t'-s) \bar{S}(s) \bar{\mathfrak{I}}(s) ds, \\ 
\Cov(\hat{W}_R(t), \hat{W}_R(t')) &=  \bar{I}(0) (\Phi_{0}(t\wedge t') - \Phi_{0}(t) \Phi_{0}(t'))  +  \int_0^{t\wedge t'} \Phi(t\wedge t'-s) \bar{S}(s) \bar{\mathfrak{I}}(s) ds , \\
\Cov(\hat{W}_S(t),\hat{W}_{\mathfrak{I}}(t')) &= -\int_0^{t\wedge t'} \bar{\lambda}(t'-s)  \bar{S}(s)\bar{\mathfrak{I}}(s)  ds\,,  \\
\Cov( \hat{W}_S(t),  \hat{W}_I(t')) &= - \int_0^{t \wedge t'} \Phi^c(t'-s)  \bar{S}(s) \bar{\mathfrak{I}}(s) ds\,, \\
\Cov( \hat{W}_S(t),  \hat{W}_R(t')) &=  - \int_0^{t \wedge t'}   \Phi(t'-s)  \bar{S}(s) \bar{\mathfrak{I}}(s) ds\,, \\
\Cov(\hat{W}_{\mathfrak{I}}(t),\hat{W}_I(t')) & = \bar{I}(0) \big(  \E \big[ \lambda^{0}(t) \bone_{\chi^{0}>t'} \big]  -  \bar\lambda^{0}(t) \Phi^c_{0}(t') \big)   +   \int_0^{t \wedge t'}  \Big( \E \big[\lambda(t-s) \bone_{ t'-s<\chi}\big]  \Big) \bar{S}(s) \bar{\mathfrak{I}}(s) ds,\\
\Cov(\hat{W}_{\mathfrak{I}}(t),\hat{W}_R(t')) & =  \bar{I}(0) \big(  \E \big[ \lambda^{0,I}(t) \bone_{\chi^{0}\le t'} \big]  -  \bar\lambda^{0}(t) \Phi_{0}(t') \big)   +  \int_0^{t \wedge t'}  \Big( \E \big[\lambda(t-s) \bone_{\chi\le t' -s}\big]  \Big) \bar{S}(s) \bar{\mathfrak{I}}(s) ds,
  \\
\Cov(\hat{W}_I(t), \hat{W}_R(t')) &= \bar{I}(0) ( (\Phi_{0}(t') - \Phi_{0}(t)) \bone(t'\ge t) - \Phi^c_{0}(t) \Phi_{0}(t') ) \\
& \quad + \bone(t'\ge t)  \int_0^{t}  (\Phi(t'-s) - \Phi(t-s)) \bar{S}(s) \bar{\mathfrak{I}}(s) ds . 
\end{align*}
\end{definition}

We shall establish below the following FCLT result, for which we impose the following further conditions on the random function $\lambda(t)$. 

\begin{assumption} \label{AS-lambda}
In addition to the conditions in Assumption \ref{AS-lambda-LLN}, the following conditions hold. 

\begin{itemize}
\item[(i)] There exist nondecreasing functions $\phi$ and $\psi$ in $\bC$ and $\alpha>1/2$ and $\beta>1$ such that for all $0 \le r\le s \le t $, denoting $\breve{\lambda}^0(t) = \lambda^0(t) - \bar{\lambda}^0(t)$, 
\begin{align*}
(a) \quad &  \E\big[ \big(\breve{\lambda}^0(t)  -\breve{\lambda}^0(s)\big)^2 \big] \le (\phi(t) - \phi(s) )^\alpha\,,\\
(b) \quad & \E\big[ \big(\breve{\lambda}^0(t)  -\breve{\lambda}^0(s)\big)^2 \big(\breve{\lambda}^0(s)  -\breve{\lambda}^0(r)\big)^2 \big] \le (\psi(t) - \psi(r) )^\beta. 
\end{align*}

\item[(ii)] 
 The function $\varphi_T$ from Assumption \ref{AS-lambda-LLN} (ii) satisfies
\begin{equation}\label{eqn-lambda-inc}
\varphi_T(t)\le C t^\alpha\,.
\end{equation} 
Also, if $F_j$ denotes the c.d.f. of the r.v. $\xi^j$, there exist $C'$ and  
 $\rho>0$ such that for any $1\le j\le k-1$, $0\le s<t$,
 \begin{equation}\label{hypF}
 F_j(t)-F_j(s)\le C'(t-s)^\rho\,,
 \end{equation}
 and in addition, for any  $1\le j\le k-1$, $r>0$,
 \begin{equation}\label{increments}
 \P(\xi^{j}-\xi^{j-1}\le r|\xi^{j-1})\le C' r^\rho\,.
 \end{equation}
\end{itemize}
 \end{assumption}

\begin{theorem} \label{th:CLT-SIR}
Under Assumptions  \ref{AS-FCLT-SIR} and  \ref{AS-lambda}, 
\[ \big(\hat{S}^N,\hat{\mathfrak{I}}^N,\hat{I}^N,\hat{R}^N\big)\Rightarrow\big(\hat{S},\hat{\mathfrak{I}},\hat{I},\hat{R}\big)\, \qinq \bD^4 \qasq N \to \infty,\]
where the process $\big(\hat{S},\hat{\mathfrak{I}}\big)$ is the unique solution of the SDE
\begin{align*}
\hat{S}(t)&=-\hat{I}(0)+\hat{W}_S(t)+\int_0^t\hat{\Upsilon}(s)ds,\\
\hat{\mathfrak{I}}(t)&=\hat{I}(0)\bar{\lambda}^0(t)+\hat{W}_{\mathfrak{I}}(t)+\int_0^t\bar{\lambda}(t-s)\hat{\Upsilon}(s)ds,
\end{align*}
where
\begin{equation}
\hat{\Upsilon}(t)=\hat{S}(t)\bar{\mathfrak{I}}(t)+\bar{S}(t)\hat{\mathfrak{I}}(t)\,.
\end{equation}
Moreover, the pair $\big(\hat{I}^N,\hat{R}^N\big)$ is given as
\begin{align*}
\hat{I}^N(t)&=\hat{I}^N(0)\Phi^c_0(t)+\hat{W}_I(t)+\int_0^t \Phi^c(t-s)\hat{\Upsilon}(s)ds,\\
\hat{R}^N(t)&=\hat{I}^N(0)\Phi_0(t)+\hat{W}_R(t)+\int_0^t \Phi(t-s)\hat{\Upsilon}(s)ds\,.
\end{align*}
\end{theorem}

\subsection{Generalized SEIR model with varying infectivity } \label{sec-model}

In this section, we consider a more detailed model by studying the separate compartments for the exposed and infectious individuals, which were combined into the infected compartment in the previous section. Changing our notation, we use now $I^N(t)$ to denote the number of infectious individuals at time $t$ in this model. And we  add $E^N(t)$ to denote the number of exposed individuals at time $t$.  The processes  $S^N(t)$ and $ R^N(t)$ 
again represent the numbers of susceptible and recovered individuals at time $t$, respectively. We have the balance equation $N= S^N(t) +
E^N(t)+ I^N(t) + R^N(t)$ for $t\ge 0$. Assume that $R^N(0)=0$, $S^N(0)>0$ and $I^N(0) >0$.  

In this model, the infected period $\chi$ in \eqref{eqn-chi} is split into the exposed and infectious periods: for each newly infected individual $i$, 
\begin{equation} \label{eqn-lambda-eta}
\zeta_i=\inf\{t>0,\ \lambda_i(t)>0\},\quad\text{and } \quad \eta_i=\inf\{t>0,\ \lambda_i(\zeta_i+r)=0,\ \forall r\ge t\}\,.\end{equation}
That is, $\chi_i = \zeta_i+\eta_i$ for each $i \ge 1$. 

However, for the initially infected individuals, we encounter two groups: initially exposed individuals (going through the remaining exposed period and then the infectious period) and 
initially infectious individuals (only going through the remaining infectious period). Unlike the newly infected individuals as described above in section \ref{sec-SIRmodel}, we will use two separate collections of variables and quantities to describe the dynamics of these two groups, which is necessary. 

Let $\{\lambda^{0}_j(\cdot)\}$ and $\{\lambda^{0,I}_k(\cdot)\}$ be the infectivity processes associated with each initially exposed and initially infectious individual, respectively.  Assume that the sequence $\{\lambda^{0}_j(\cdot)\}$ is i.i.d., and so are $\{\lambda^{0,I}_k(\cdot)\}$. 
The remaining exposed period and the infectious period $(\zeta^0_j,\eta^0_j)$ of an initially exposed individual
are given as: 
\begin{equation} \label{eqn-lambda-eta-0}
\zeta^0_j=\inf\{t>0,\ \lambda^0_j(t)>0\}>0, \quad\text{and } \quad \eta^0_j=\inf\{t>0,\ \lambda^0_j(\zeta^0_j+r)=0,\ \forall r\ge t\}, 
\end{equation}
and the remaining infectious period $\eta^{0,I}_k$ of an initially infectious individual (note that $ \inf\{t>0,\ \lambda^{0,I}_k(t)>0\}=0$) is: 
 \begin{equation} \label{eqn-lambda-eta-0I}
\quad \eta^{0,I}_k=\inf\{t>0,\ \lambda^{0,I}_k(r)=0,\ \forall r\ge t\}.
 \end{equation}
 
Under the i.i.d. assumptions of the corresponding infectivity processes, the random vectors $\{(\zeta_i,\eta_i):i \in \NN\}$ and $\{(\zeta^0_j,\eta^0_j):j \in \NN\}$ are i.i.d., and so is the sequence $\{\eta^{0,I}_k: k\in \NN\}$.  In addition, these three sequences are mutually independent.

Let $H(du,dv)$ denote the law of $(\zeta,\eta)$, $H_0(du,dv)$ that of $(\zeta^0,\eta^0)$ and $F_{0,I}$ the c.d.f. of $\eta^{0,I}$. Define 
\begin{align*}
\Phi(t)&:=\int_0^t\int_0^{t-u}H(du,dv)=\P(\zeta+\eta\le t),\quad \\
 \Psi(t) &:=\int_0^t\int_{t-u}^{\infty}H(du,dv)=\P(\zeta\le t<\zeta+\eta),\\
\Phi_0(t)&:=\int_0^t\int_0^{t-u}H_0(du,dv)=\P(\zeta^0+\eta^0\le t),\quad \\
 \Psi_0(t)& :=\int_0^t\int_{t-u}^{\infty}H_0(du,dv)=\P(\zeta^0\le t<\zeta^0+\eta^0),\\
 F_{0,I}(t) &:=\P(\eta^{0,I}\le t).
\end{align*}
(Note that $\Phi(t)$ and $\Phi_0(t)$ were used for the infected periods in the previous section for the generalized SIR model; if the infected periods have both exposed and infectious periods, then they have the same meanings as defined above in the generalized SEIR model. See Remark \ref{rem-comparison} for further discussions.)  
We write $ H(du,dv)=G(du)F(dv|u)$ and $H_0(du,dv)=G_0(du)F_0(dv|u)$,
i.e., $G$ is the c.d.f. of $\zeta$ and $F(\cdot|u)$ is the conditional law of $\eta$, given that $\zeta=u$, 
$G_0$ is the c.d.f. of $\zeta^0$ and $F_0(\cdot|u)$ is the conditional law of $\eta^0$, given that $\zeta^0=u$.
In the case of independent exposed and infectious periods, it is reasonable that the infectious periods of the initially exposed individuals have the same distribution as the newly exposed ones, that is, $F_{0}=F$. Note that in the independent case, $\Psi(t) = G(t) - \Phi(t)$ and $\Psi_0(t) = G_0(t) - \Phi_0(t)$. 
Also, let $G^c_0=1-G_0$, $G^c=1-G$, $F^c_{0,I}=1-F_{0,I}$, and $F^c=1-F$. 
 
We again use $A^N(t)$ to denote the number of individuals that are exposed in $(0,t]$, and $\tau^N_i$ to denote the time at which the $i^{\rm th}$ individual gets exposed.  Then $A^N(t)$ has the same expression as in \eqref{eqn-An-rep-1}. 
However, by taking into account the detailed initial conditions with initially exposed and infectious individuals, the total force of infection $\mathfrak{I}^N(t)$ at time $t$ is expressed as 
\begin{align} \label{eqn-mf-I}
\mathfrak{I}^N(t) =  \sum_{j=1}^{E^N(0)}\lambda^0_j(t) +\sum_{k=1}^{I^N(0)}\lambda^{0,I}_k(t) 
+\sum_{i=1}^{A^N(t)}\lambda_i (t-\tau^N_i) \,, \quad t \ge 0.
\end{align}
The epidemic dynamics of the model can be described by 
\begin{align*}
S^N(t) &\,=\,S^N(0)-A^N(t)\,, \\
E^N(t)&\,=\, \sum_{j=1}^{E^N(0)}{\bf1}_{\zeta^0_j>t}+\sum_{i=1}^{A^N(t)}{\bf1}_{\tau^N_i+\zeta_i>t}\,,\\
 I^N(t) &\,=\, \sum_{j=1}^{E^N(0)}{\bf1}_{\zeta^0_j\le t<\zeta^0_j+\eta^0_j}+\sum_{k=1}^{I^N(0)}{\bf1}_{\eta^{0,I}_k>t}+\sum_{i=1}^{A^N(t)}{\bf1}_{\tau^N_i+\zeta_i\le t<\tau^N_i+\zeta_i+\eta_i}\,, \\
  R^N(t) &\,=\, \sum_{j=1}^{E^N(0)}{\bf1}_{\zeta^0_j+\eta^0_j\le t}+\sum_{k=1}^{I^N(0)}{\bf1}_{\eta^{0,I}_k\le t}+\sum_{i=1}^{A^N(t)}{\bf1}_{\tau^N_i+\zeta_i+\eta_i\le t}\,. 
\end{align*}

\begin{remark} \label{rem-I-expression}
We remark that in addition to the extra care in the initial conditions, 
since we only count the infectious individuals in $I^N(t)$, the third term in the expression of $I^N(t)$ differs from the second term in the expression of $I^N(t)$ in \eqref{eq:I+RN} in the generalized SIR model. 
\end{remark}

\begin{remark} \label{rem-comparison}
We remark that the model in the previous section includes the special case without any exposed periods. Specifically,  the random infectivity function $\lambda_i(t)$ does not equal to zero immediately after time $0$, that is, an infected individual is immediately infectious, so $\zeta=\zeta^0=0$ a.s., and there are no exposed individuals, $E^N(t) =0$ for all $t\ge 0$. 
In that case, the FLLN and FCLT results in Theorems \ref{th:LLN-SIR} and \ref{th:CLT-SIR} hold with $\Phi$ being equal to $F$, the c.d.f. of the infectious period in this section. 

On the other hand, the results for the SIR model without exposed periods can be also obtained by setting $G$ to have mass 1 at 0 so that $\Phi(t)=F(t)$ and $\bar{E}(0)=0$ in the following two theorems.

\end{remark}

In \cite[Theorem 2.1]{FPP2020b}, the following FLLN is proved. 
\begin{theorem} 
Suppose that the conditions in Assumption \ref{AS-lambda-LLN} hold for $\{\lambda^{0}_j(\cdot)\}$, $\{\lambda^{0,I}_k(\cdot)\}$ and $\{\lambda_i(\cdot)\}$. 
If there exist deterministic constants $\bar{E}(0),\bar{I}(0) \in [0,1)$ such that $0< \bar{E}(0)+\bar{I}(0)<1$, and
 $(\bar{E}^N(0),\bar{I}^N(0))\to (\bar{E}(0),\bar{I}(0)) \in \RR^2_+$ in probability as $N\to\infty$, 
  \begin{equation} \label{eqn-FLLN-conv-SEIR}
\big(\bar{S}^N, \bar{\mathfrak{I}}^N, \bar{I}^N, \bar{E}^N, \bar{R}^N\big) 
\to \big(\bar{S}, \bar{\mathfrak{I}}, \bar{I}, \bar{E}, \bar{R}\big) \qinq \bD^5 \qasq N \to \infty\,,
\end{equation}
in probability.  
The limit $(\bar{S},\bar{\mathfrak{I}})$ is the unique solution of the following system of integral equations: 
\begin{align}
\bar{S}(t)&=1- \bar{E}(0)-\bar{I}(0)-\int_0^t\bar{S}(s)\bar{\mathfrak{I}}(s)ds\,, \label{eqn-barS}\\
\bar{\mathfrak{I}}(t)&=\bar{E}(0)\bar{\lambda}^0(t) + \bar{I}(0)\bar{\lambda}^{0,I}(t)+\int_0^t\bar{\lambda}(t-s)\bar{S}(s)\bar{\mathfrak{I}}(s)ds\,,\label{eqn-barfrakI}
\end{align}
and the limits $(\bar{E},\bar{I},\bar{R})$ are given by the following formulas:
\begin{align*}
 \bar{E}(t) &=\bar{E}(0)G_0^c(t)+\int_0^tG^c(t-s)\bar{S}(s)\bar{\mathfrak{I}}(s)ds\,, \\
 \bar{I}(t) &= \bar{I}(0) F_{0,I}^c(t) + \bar{E}(0)\Psi_0(t)  +\int_0^t \Psi(t-s) \bar{S}(s)\bar{\mathfrak{I}}(s)ds\,, \\
\bar{R}(t)&=\ \bar{I}(0) F_{0,I}(t) + \bar{E}(0) \Phi_0(t)  +\int_0^t \Phi(t-s) \bar{S}(s)\bar{\mathfrak{I}}(s)ds\, . 
\end{align*}
We also have $\bar{\Upsilon}^N \to \bar{\Upsilon}$ in $\bD$ in probability as $N \to \infty$, where 
\begin{equation} \label{eqn-barUpsilon}
\bar{\Upsilon}(t) := \bar{S}(t)\bar{\mathfrak{I}}(t)\,, \quad t \ge 0\,. 
\end{equation}
\end{theorem}

We make the following assumption on the initial quantities for the FCLT.  
\begin{assumption} \label{AS-FCLT}
There exist deterministic constants  $\bar{E}(0), \bar{I}(0) \in [0,1)$ satisfying $\bar{E}(0)+\bar{I}(0)>0$ and random variables  $\hat{E}(0), \hat{I}(0)$ such that   $(\hat{E}^N(0), \hat{I}^N(0)) := \sqrt{N} (\bar{E}^N(0) - \bar{E}(0), \bar{I}^N(0) - \bar{I}(0)) \Rightarrow (\hat{E}(0), \hat{I}(0)) $ as $N \to \infty$. Moreover  $\sup_N\E[\hat{E}^N(0)^2]<\infty$,  $\sup_N\E[\hat{I}^N(0)^2]<\infty$, and thus  $\E[\hat{E}(0)^2]<\infty$, $\E[\hat{I}(0)^2]<\infty$. 
\end{assumption}

\begin{definition}\label{def-W}
Let $(\hat{W}_S,\hat{W}_{\mathfrak{I}},\hat{W}_E, \hat{W}_I, \hat{W}_R)$ is a five-dimensional centered Gaussian process, independent of $(\hat{E}(0),\hat{I}(0))$, with covariance functions: for $t, t' \ge 0$, 
\begin{align*}
\Cov( \hat{W}_{\mathfrak{I}}(t), \hat{W}_{\mathfrak{I}}(t')) & = \bar{E}(0) \Cov(\lambda^0(t), \lambda^0(t'))+ \bar{I}(0) \Cov(\lambda^{0,I}(t), \lambda^{0,I}(t'))  \\
& \quad +  \int_0^{t\wedge t'} \Big( \Cov(\lambda(t-s), \lambda(t'-s))+  \bar{\lambda}(t-s) \bar{\lambda}(t'-s)\Big) \bar{S}(s)\bar{\mathfrak{I}}(s)  ds\,, \\
 \Cov(\hat{W}_E(t), \hat{W}_E(t')) &=  \bar{E}(0) (G_0^c(t\vee t') - G^c_0(t) G^c_0(t')) +  \int_0^{t\wedge t'} G^c(t\vee t'-s) \bar{S}(s) \bar{\mathfrak{I}}(s) ds,  \\
\Cov(\hat{W}_I(t), \hat{W}_I(t')) &= \bar{I}(0) (F_{0,I}^c(t\vee t') - F_{0,I}^c(t) F_{0,I}^c(t')) + \bar{E}(0) \bigg(  \int_0^{t\wedge t'} F_0^c(t\vee t'-s|s) G_0(ds) - \Psi_0(t) \Psi_0(t') \bigg)   \\
& \quad +   \int_0^{t\wedge t'}\int_0^{t\wedge t'-s} F^c(t\vee t'-s-u|u)G(du)  \bar{S}(s) \bar{\mathfrak{I}}(s) ds, \\ 
\Cov(\hat{W}_R(t), \hat{W}_R(t')) &=  \bar{I}(0) (F_{0,I}(t\wedge t') - F_{0,I}(t) F_{0,I}(t')) +  \bar{E}(0) \left( \Phi_0(t\wedge t') -   \Phi_0(t) \Phi_0(t')  \right) \\
& \quad +  \int_0^{t\wedge t'} \Phi(t\wedge t'-s) \bar{S}(s) \bar{\mathfrak{I}}(s) ds , \\
\Cov( \hat{W}_S(t),  \hat{W}_E(t')) &= -\int_0^{t \wedge t'}   G^c(t'-s) \bar{S}(s) \bar{\mathfrak{I}}(s) ds\,, \,
\Cov( \hat{W}_S(t),  \hat{W}_I(t')) = - \int_0^{t \wedge t'} \Psi(t'-s)  \bar{S}(s) \bar{\mathfrak{I}}(s) ds\,, \\
\Cov(\hat{W}_{\mathfrak{I}}(t),\hat{W}_E(t')) & = \bar{E}(0) \big(\E \big[ \lambda^{0}(t) \bone_{\zeta^{0}> t'} \big] -\bar\lambda^{0}(t) G^c_{0}(t')   \big)  +  \int_0^{t \wedge t'}  \Big( \E \big[\lambda(t-s) \bone_{\zeta>t'-s}\big]  \Big) \bar{S}(s) \bar{\mathfrak{I}}(s) ds, 
 \\
\Cov(\hat{W}_{\mathfrak{I}}(t),\hat{W}_I(t')) & = \bar{I}(0) \big(  \E \big[ \lambda^{0,I}(t) \bone_{\eta^{0,I}>t'} \big]  -  \bar\lambda^{0,I}(t) F^c_{0,I}(t') \big) + \bar{E}(0) \big( \E \big[   \lambda^{0}(t) \bone_{\zeta^{0}\le t'<\zeta^{0} + \eta^{0}} \big]  - \bar\lambda^{0}(t)  \Psi_{0}(t') \big) \\
& \quad +  \int_0^{t \wedge t'}  \Big( \E \big[\lambda(t-s) \bone_{\zeta\le t'-s<\zeta+\eta}\big]  \Big) \bar{S}(s) \bar{\mathfrak{I}}(s) ds,\\
\Cov(\hat{W}_{\mathfrak{I}}(t),\hat{W}_R(t')) & =  \bar{I}(0) \big(  \E \big[ \lambda^{0,I}(t) \bone_{\eta^{0,I}\le t'} \big]  -  \bar\lambda^{0,I}(t) F_{0,I}(t') \big) + \bar{E}(0) \big( \E \big[  \lambda^{0}(t) \bone_{\zeta^{0} + \eta^{0}\le t'}\big]  -\bar\lambda^{0}(t)\Phi_{0}(t') \big)  \\
& \quad  +  \int_0^{t \wedge t'}  \Big( \E \big[\lambda(t-s) \bone_{\zeta+\eta\le t' -s}\big]  \Big) \bar{S}(s) \bar{\mathfrak{I}}(s) ds,
\end{align*} 

\begin{align*}
\Cov(\hat{W}_E(t), \hat{W}_I(t')) &= \bar{E}(0)\bone(t'\ge t) \bigg( \int_t^{t'}  F^c_0(t'-s|s)G_0(ds) - G_0^c(t) \Psi_0(t') \bigg)  \\
& \quad + \bone(t'\ge t)  \int_0^{t\wedge t'}  \int_{t-s}^{t'-s} F^c(t'-s-u|u) G(du)  \bar{S}(s) \bar{\mathfrak{I}}(s) ds,  \\
\Cov(\hat{W}_I(t), \hat{W}_R(t')) &= \bar{I}(0) ( (F_{0,I}(t') - F_{0,I}(t)) \bone(t'\ge t) - F^c_{0,I}(t) F_{0,I}(t') ) \\
& \quad +  \bar{E}(0) \bone(t'\ge t) \bigg( \int_0^{t} (F_0(t'-s|s) - F_0(t-s|s))G_0(ds) - \Psi_0(t) \Phi_0(t') \bigg) \\
& \quad + \bone(t'\ge t)  \int_0^{t\wedge t'} \int_0^{t-s}  (F(t'-s-u|u) - F(t-s-u|u)) G(du) \bar{S}(s) \bar{\mathfrak{I}}(s) ds, \\
\Cov(\hat{W}_E(t), \hat{W}_R(t')) &=  \bar{E}(0) \bone(t'\ge t) \bigg(\int_t^{t'}F_{0}(t'-s|s) G_0(ds) - G_0^c(t)  \Phi_0(t')\bigg) \\
& \quad +  \bone(t'\ge t)  \int_0^{t\wedge t'}  \int_{t-s}^{t'-s} F(t'-s-u|u) G(du) \bar{S}(s) \bar{\mathfrak{I}}(s)) ds,
\end{align*}
and $\Cov(\hat{W}_S(t),\hat{W}_S(t'))$, $\Cov(\hat{W}_S(t),\hat{W}_{\mathfrak{I}}(t'))$, and $\Cov( \hat{W}_S(t),  \hat{W}_R(t'))$  are the same as those in Definition \ref{def-W-SIR}. 
\end{definition}

\begin{theorem} \label{thm-FCLT}
Under Assumption \ref{AS-lambda} with $\lambda^{0,I}(t)$ also being bounded and satisfying the conditions in (ii), and  under Assumption \ref{AS-FCLT},  
\begin{equation} \label{eqn-hatSfrakI-conv}
\big(\hat{S}^N, \hat{\mathfrak{I}}^N,\hat{E}^N, \hat{I}^N, \hat{R}^N)\RA \big(\hat{S},\hat{\mathfrak{J}},\hat{E}, \hat{I}, \hat{R}\big) \qinq \bD^5 \qasq N \to \infty\,. 
\end{equation}
The limit process $(\hat{S},\hat{\mathfrak{I}})$ is the unique solution to the following system of stochastic integral equations: 
\begin{align}
\hat{S}(t) &= - \hat{E}(0)  -\hat{I}(0) + \hat{W}_S(t) + \int_0^t\hat{\Upsilon}(s)ds,  \label{eqn-hatS}\\
\hat{\mathfrak{I}}(t) &= \hat{I}(0) \bar{\lambda}^{0,I}(t) + \hat{E}(0) \bar{\lambda}^0(t)+ \hat{W}_{\mathfrak{I}}(t)  + \int_0^t \bar{\lambda}(t-s) \hat{\Upsilon}(s) ds,   \label{eqn-frakI}
\end{align}
where
\begin{align} \label{def-hatUpsilon}
\hat{\Upsilon}(t)  =  \hat{S}(t)  \bar{\mathfrak{I}}(t) + \bar{S}(t)  \hat{\mathfrak{I}}(t),
\end{align}
and $\bar{S}(t)$ and $ \bar{\mathfrak{I}}(t)$ are given by the unique solutions to the integral equations \eqref{eqn-barS} and \eqref{eqn-barfrakI}.
The limit processes  $\hat{E}$, $\hat{I}$ and $\hat{R}$ are given by the expressions: 
\begin{align*}
\hat{E}(t) &= \hat{E}(0) G_0^c(t) + \hat{W}_E(t)   + \int_0^t G^c(t-s)  \hat{\Upsilon}(s) ds,\\
\hat{I}(t) &= \hat{I}(0) F_{0,I}^c(t) + \hat{E}(0) \Psi_0(t) +  \hat{W}_{I}(t)  + \int_0^t \Psi(t-s)  \hat{\Upsilon}(s) ds, \\
\hat{R}(t) &=  \hat{I}(0) F_{0,I}(t) + \hat{E}(0) \Phi_0(t) +  \hat{W}_{R}(t)  + \int_0^t \Phi(t-s)  \hat{\Upsilon}(s) ds, 
\end{align*}

$\hat{S}$ has continuous paths, and 
if $\bar{\lambda}^0$ and $\bar{\lambda}^{0,I}$ are in $\bC$, which implies that $G_0$ and $F_{0,I}$ are continuous,  then $\hat{\mathfrak{I}}$, $\hat{E}$, $\hat{I}$ and $\hat{R}$ are continuous. 
\end{theorem}

\subsection{On the assumptions on $\lambda(t)$}
\label{sec-on-lambda}

We remark that the conditions in Assumption \ref{AS-lambda}
are not required to establish the FLLN \cite{FPP2020b}.  
It is not surprising that the FCLT requires additional assumptions, compared with the FLLN. 

These additional conditions are required to establish the moment criterion for tightness of the aggregate infectivity process (see Propositions \ref{prop-wt-mfI-1-inc-moment} and \ref{prop-wt-mfI-2-inc-moment} for the moment bounds for the increments of the associated processes). In fact, condition \eqref{increments} on $\{\xi^j\}$ is used only once in the end of the proof of  Proposition \ref{prop-wt-mfI-1-inc-moment}, while \eqref{hypF} is used twice in the proofs of both Proposition \ref{prop-wt-mfI-1-inc-moment} and \ref{prop-wt-mfI-2-inc-moment}. 
The conditions in  (iii) allow fairly general random infectivity functions $\lambda(t)$, which can have c\`adl\`ag paths. 
They evidently impose a restriction on the distribution functions of the exposed and infectious periods in the case where $\lambda$ jumps at the end of those periods, which in particular implies that they are continuous. 
Recall that no condition is required on the distribution functions of the exposed and infectious periods in establishing the FLLNs and FCLTs of the standard SIR and SEIR models in \cite{PP-2020}. 
The restriction on these distributions as a result of  \eqref{hypF} and \eqref{increments} is a compromise of allowing the random infectivity functions to have discontinuities. 

On the other hand, as a special case with $k=1$, $\lambda(\cdot)$ is continuous satisfying \eqref{eqn-lambda-inc}, no conditions of the type \eqref{hypF} and \eqref{increments} has to be imposed. 
In that case, one can allow the durations of the exposed and infectious periods to take discrete values, and as a consequence, 
 the same results of the FLLNs and FCLTs hold without requiring any condition on the distribution functions of those exposed and infectious periods.

 We now describe examples of $\lambda$'s which satisfy the conditions in Assumption \ref{AS-lambda}. 
 We also refer the reader  to
 \cite[Section 2.5]{FPP2020b} for an example of $\lambda$ adapted to the Covid--19 epidemic. 
 
 \begin{example} \label{ex-1}
Suppose we are given a random element $(\zeta,\eta)$ of $(0,+\infty)^2$, and let $\lambda$ be a random function with trajectories in $\bC$ which is such that
\[ \lambda(t) 
\begin{cases} >0, &\text{for $\zeta<t<\zeta+\eta$},\\
=0, & \text{for $t\not\in(\zeta,\zeta+\eta)$},
\end{cases}
\]
and which is bounded and satisfies \eqref{eqn-lambda-inc}. Then such a $\lambda$
satisfies all the conditions in Assumption  \ref{AS-lambda},  and the law of the pair $(\zeta,\eta)$ is completely arbitrary.  This corresponds to the case $k=1$ in Assumption  \ref{AS-lambda} (ii).
 \end{example}

\begin{example} \label{ex-2}
Suppose we are given again a random element $(\zeta,\eta)$ of $(0,+\infty)^2$, and $\lambda$ a random function with now trajectories in $\bD$ of the form
\[ \lambda(t)=\lambda^1(t){\bf1}_{[\zeta,\zeta+\eta)}(t),\]
where $\lambda^1\in\bC$ is bounded and satisfies \eqref{eqn-lambda-inc}. Note that unlike in the above example, 
we allow $\lambda^1(\zeta)>0$ and $\lambda^1(\zeta+\eta)>0$, hence $\lambda(\zeta)>\lambda(\zeta^-)=0$ and $\lambda((\zeta+\eta)^-)>\lambda(\zeta+\eta)=0$. In this case, Assumption \ref{AS-lambda} (ii) is satisfied with 
$k=3$, $\xi^1=\zeta$ and $\xi^2=\zeta+\eta$, provided the distribution functions of both $\zeta$ and $\zeta+\eta$ satisfy  
\eqref{hypF}, and \eqref{increments} holds. We can of course add more jumps on the time interval
$(\xi^1,\xi^2)=(\zeta,\zeta+\eta)$. In the particular case $\lambda^1(\zeta)=0$ but $\lambda^1(\zeta+\eta)>0$,
(resp. $\lambda^1(\zeta)>0$ but $\lambda^1(\zeta+\eta)=0$), we are in the case of Assumption \ref{AS-lambda} (ii) with $k=2$ and $\xi^1=\zeta+\eta$ (resp. $\xi^1=\zeta$), and condition \eqref{hypF} needs to be imposed to that r.v. only, while condition \eqref{increments} is irrelevant in this case. Finally if $\lambda^1(\zeta)=\lambda^1(\zeta+\eta)=0$,
we are back in the previous example.
\end{example}

\section{Proofs for the Convergence of $(\hat{S}^N, \hat{\mathfrak{I}}^N)$ } \label{sec-Proofs}

In this section we prove the convergence of $(\hat{S}^N, \hat{\mathfrak{I}}^N) \to (\hat{S}, \hat{\mathfrak{I}})$ in $\bD^2$ in Theorem \ref{thm-FCLT}. The joint convergence with  $(\hat{E}^N, \hat{I}^N, \hat{R}^N)$ will be studied in the next section.  
We first provide the following technical preliminaries used in the proofs.

\subsection{Technical Preliminaries}\label{sec-proofs-tech}

In this subsection, we gather a few technical results which will be used later in our proofs.
Recall the moment criterion \eqref{moment3points} (which is (13.14) of \cite{billingsley1999convergence}. We will first deduce from that condition a uniform moment estimate.

Suppose now that a sequence $\{X^N\}$ satisfies the condition \eqref{3points} (which is (13.13) in \cite{billingsley1999convergence}). Under an additional mild condition, it implies that  $\{X^N\}$ is tight in $\bD$, hence in particular that for any $T>0$, $\sup_{0\le t\le T}|X^N(t)|$ is a tight sequence of $\R_+$--valued r.v.'s. We now show how we can deduce a bound on the second moment of that supremum.

\begin{theorem}\label{th:momentsup}
Let $X^N$ be a sequence in $\bD([0,T])$. Suppose that the two following conditions are satisfied 
\begin{align}\label{supmoment}
\sup_{N\ge1}\E\left[(X^N(0))^2+(X^N(T))^2\right]<\infty,
\end{align}
and that for some $\alpha>1/2$, $\beta>1/2$, and
for all $0\le r\le s\le t\le T$, $N\ge1$, $\lambda>0$,
\begin{align}\label{3pts}
\P\left(|X^N(s)-X^N(r)|\wedge|X^N(t)-X^N(s)|\ge\lambda\right)\le\frac{[G(t)-G(r)]^{2\alpha}}{\lambda^{4\beta}}\,,
\end{align}
where $G$ is a nondecreasing continuous function satisfying $G(0)=0$.  
Then
\begin{align}\label{momentsup}
\sup_{N\ge1}\E\bigg[\sup_{0\le t\le T}|X^N(t)|^2\bigg]<\infty\,.
\end{align}
\end{theorem}
As noted in \cite{billingsley1999convergence}, a sufficient condition for \eqref{3pts} is that 
\begin{align} \label{3pts-moments}
\E\left[|X^N(s)-X^N(r)|^{2\beta} |X^N(t)-X^N(s)|^{2\beta}\right]\le  [G(t)-G(r)]^{2\alpha}\,.
\end{align}
We shall use this result with $\beta=1$. Condition \eqref{3pts} with any $\beta\ge0$ plus a minor condition, much weaker than the following reinforced version of \eqref{supmoment}:
\begin{equation}\label{suptE}
\sup_{N\ge1}\sup_{0\le t\le T}\E\bigg[|X^N(t)|^2\bigg]<\infty
\end{equation}
 implies tightness of the sequence $X^N$. Note that going from \eqref{suptE} to \eqref{momentsup} is usually easy to achieve if $X^N$ is a semimartingale, using Doob's inequality for the martingale part. In our non Markovian setup, we do not have enough martingales at our disposal, and the above theorem will be useful to us.
\begin{proof}
Define for each $N\ge1$ and $0\le r\le s\le t\le T$,  the r.v.'s 
\[ M^N_{r,s,t}:=|X^N(s)-X^N(r)|\wedge|X^N(t)-X^N(s)|,\quad L^N=\sup_{0\le r\le s\le t\le T}M^N_{r,s,t}\,.\]
 Theorem 10.3 in \cite{billingsley1999convergence} says that, provided that the process $X^N(t)$ is right continuous, \eqref{3pts} implies that there exists a constant $K_{\alpha,\beta}$ such that for all $N\ge1$, $x>0$,
 \begin{equation}\label{LN}
  \P(L^N\ge x)\le K_{\alpha,\beta}\frac{G(T)^{2\alpha}}{x^{4\beta}}\,.
 \end{equation}
 Since $X^N(s)\le X^N(0)\vee X^N(T)+M^N_{0,s,T}$ and $\sup_{0\le s\le T}M^N_{0,s,T}\le L^N$,
 \begin{align*}
 \E\left(\sup_{0\le t\le T}|X^N(t)|^2\right)&\le3\E\left(|X^N(0)|^2\right)
 +3 \E\left(|X^N(T)|^2\right)
 +3\int_0^\infty\P\left(L^N>\sqrt{x}\right)dx\\
 &\le 3\E\left(|X^N(0)|^2\right)+3 \E\left(|X^N(T)|^2\right)
 +3\int_0^\infty1\wedge\frac{K_{\alpha,\beta}G(T)^{2\alpha}}{x^{2\beta}}dx\\
 &\le3\E\left(|X^N(0)|^2+|X^N(T)|^2\right)+\frac{6\beta}{2\beta-1}K_{\alpha,\beta}^{1/2\beta}G(T)^{\alpha/\beta}\,,
 \end{align*}
 where we have used \eqref{LN} for the second inequality. The result now follows from \eqref{supmoment} and the fact that $2\beta>1$.
 \end{proof}

We shall also use the following technical lemma.

\begin{lemma}\label{Lem-20}
Let $\{X^N\}_{N\ge1}$ be a sequence of random elements in $\bD$ such that $X^N(0)=0$.
If for all $T>0$, $\ep>0$, as $\delta\to0$,
\begin{align*}
\limsup_{N}\sup_{0\le t\le T}\frac{1}{\delta}\P\bigg(\sup_{0\le u\le \delta}|X^N(t+u)-X^N(t)|>\ep\bigg)\to0,
\end{align*}
 then the sequence $X^N$ is tight in $\bD$.
 \end{lemma}

\begin{proof}
 The result follows from a combination of several results from \cite{billingsley1999convergence}. First the Corollary of Theorem 7.4 says that our assumption implies that for all $T>0$,
\begin{align*}\label{tightC}
\limsup_N\P\left(W^T_{X^N}(\delta)\ge\varepsilon\right)\to0,\ \text{ as }\delta\to0,
\end{align*}
where $W^T_x(\delta):=\sup_{0\le s,t\le T,\ |t-s|\le\delta}|x(t)-x(s)|$.
Combined with the inequality (12.7), that last fact implies that the second condition of Theorem 13.2 applies.
Moreover since our assumption implies that the jumps of $X^N$ tend to $0$ as $N\to\infty$, and $X^N(0)=0$,
condition (i'') from the Corollary of Theorem 13.2 applies, which replaces the first condition of the Theorem, hence 
that Theorem applies, and the sequence $X^N$ is tight in $\bD([0,T])$ for all $T>0$, hence in $\bD$.
\end{proof}

Let us recall well--known formulas for the exponential moment and moments of an integral with respect to  a compensated PRM. These follow, e.g., rather easily from Theorem VI.2.9 and Exercise VI.2.21 in \cite{ccinlar2011probability}.  
\begin{lemma}\label{lem:expmoment}
Let $Q$ be a PRM on some measurable space $(E,\mathcal{E})$, with mean measure $\nu$, and $\bar{Q}$ the associated compensated measure.
Let $f:E\mapsto\mathbb{C}$ be measurable and such that $e^f-1-f$ is $\nu$ integrable. Then
\[
\E\left[\exp\left(\int_E f(x)\bar{Q}(dx)\right)\right]
= \exp\left(\int_E \left[e^{f(x)}-1-f(x)\right]\nu(dx)\right)\,.
\]
If $\nu(f^2)<\infty$,
$$
\E\left[\left(\int_E f(x)\bar{Q}(dx)\right)^2\right] = \int_E f(x)^2\nu(dx) \,. 
$$
If $\nu(f^2+g^2+f^2g^2)<\infty$, 
\begin{align}\label{mixed2moment}
\E\left[\left(\int_E f(x)\bar{Q}(dx)\right)^2\left(\int_E g(x)\bar{Q}(dx)\right)^2\right] &=
\int_Ef^2(x)g^2(x)\nu(dx)+2\left(\int_E f(x)g(x)\nu(dx)\right)^2\non \\
&\quad+\int_E f^2(x)\nu(dx)\times\int_E g^2(x)\nu(dx)\,.
\end{align}
\end{lemma}

We will need the following bounds on the increments of the infectivity function $\lambda(\cdot)$,  which are used  below  in the proofs of Propositions \ref{prop-wt-mfI-1-inc-moment}--\ref{prop-wt-mfI-2-inc-moment} and Lemmas \ref{lem-frakI1-diff-conv}--\ref{lem-frakI2-diff-conv},   
 \begin{lemma} \label{lem-barlambda-inc-bound}
For $t\ge s\ge 0$, with $\alpha>1/2$ as in Assumption \ref{AS-lambda} (ii),
\begin{align*}
  \big|\lambda(t) - \lambda(s) \big| \le (t-s)^\alpha + \lambda^*  \sum_{j=1}^{k-1} \bone_{s < \xi^j \le t}\,,
 \end{align*}
 and
 \begin{align*}
 |\bar\lambda(t) -\bar \lambda(s)| \le (t-s)^\alpha + \lambda^* \sum_{j=1}^{k-1} (F_j(t) - F_j(s))\,. 
 \end{align*}
 Moreover, 
 \begin{align*}
\E[|\breve{\lambda}(t)-\breve{\lambda}(s)|] &\le 2(t-s)^\alpha+2\lambda^* \sum_{j=1}^{k-1} (F_j(t)-F_j(s))\,,
\end{align*}
\begin{align*}
\E[|\breve{\lambda}(t)-\breve{\lambda}(s)|^2] \le  8 (t-s)^{2\alpha}  +  4(\lambda^*)^2 k \sum_{j=1}^{k-1} \big(F_j(t)-F_j(s) \big)   + 4(\lambda^*)^2 \bigg(\sum_{j=1}^{k-1} (F_j(t) - F_j(s)) \bigg)^2 \,. 
\end{align*}
 \end{lemma}
 \begin{proof}
 We have 
 \begin{align*}
 \lambda(t) - \lambda(s) =\sum_{j=1}^k \big( \lambda^j(t) - \lambda^j(s) \big) \bone_{\xi^{j-1} \le s,\, t < \xi^j}
  +  \big(  \lambda(t) - \lambda(s)\big) \sum_{j=1}^{k-1} \bone_{s < \xi^j \le t}\,. 
 \end{align*}
 Thus the first statement follows from Assumption \ref{AS-lambda} (ii). The other statements follow readily from the first one.  
  \end{proof}
 
 We shall also use several times the following  straightforward inequality: for any $1\le j\le k-1$, $0\le r\le t$,
 \begin{align}\label{estim-int-F}
 \int_0^r(F_j((t-s)-F_j(r-s))ds&=\int_{t-r}^tF_j(s)ds-\int_0^rF_j(s)ds\non\\
 &\le\int_r^tF_j(s)ds\le t-r\,.
 \end{align}

\subsection{Representation of $(\hat{S}^N, \hat{\mathfrak{I}}^N)$ and proof roadmap} \label{sec-rep-roadmap}

By the expressions of $\mathfrak{I}^N$ in \eqref{eqn-mf-I} and $\bar{\mathfrak{I}}$ in \eqref{eqn-barfrakI}, we obtain 
\begin{align} \label{eqn-hatfrakI-rep}
\hat{\mathfrak{I}}^N(t) = \hat{I}^N(0) \bar{\lambda}^{0,I}(t)+  \hat{E}^N(0) \bar{\lambda}^{0}(t) +  \hat{\mathfrak{I}}^N_{0,1}(t) +  \hat{\mathfrak{I}}^N_{0,2}(t)  + \hat{\mathfrak{I}}^N_1(t) + \hat{\mathfrak{I}}^N_2(t)   
 + \int_0^t \bar{\lambda}(t-s) \hat{\Upsilon}^N(s) ds, 
\end{align}
where
\begin{align*} 
 \hat{\mathfrak{I}}^N_{0,1}(t)  := \frac{1}{\sqrt{N}}   \sum_{k=1}^{I^N(0)} \big( \lambda^{0,I}_k(t)  - \bar{\lambda}^{0,I}(t) \big),
\end{align*}
\begin{align*}
 \hat{\mathfrak{I}}^N_{0,2}(t)  := \frac{1}{\sqrt{N}}   \sum_{j=1}^{E^N(0)} \big( \lambda^0_j(t)  - \bar{\lambda}^0(t) \big),
\end{align*}
\begin{align} \label{eqn-hatfrakI1-def}
 \hat{\mathfrak{I}}^N_1(t)  := \frac{1}{\sqrt{N}}   \sum_{i=1}^{A^N(t)}\big( \lambda_i (t-\tau^N_i)   -
\bar\lambda (t-\tau^N_i)  \big), 
\end{align}
and 
\begin{align} \label{eqn-hatfrakI2-def}
 \hat{\mathfrak{I}}^N_2(t)  := \frac{1}{\sqrt{N}} \left(  \sum_{i=1}^{A^N(t)}\bar{\lambda} (t-\tau^N_i)   -  \int_0^t \bar{\lambda}(t-s) \Upsilon^N(s) ds  \right). 
\end{align}

For the process $A^N(t)$, we have the decomposition
\begin{align} \label{eqn-An-decomp}
A^N(t)= M_A^N(t) + \int_0^t\Upsilon^N(s)ds\,,
\end{align}
where
$$
M_A^N(t) = \int_0^t\int_0^\infty{\bf1}_{u\le \Upsilon^N(s^-)}\overline{Q}(ds,du)\,, 
$$
with $\overline{Q}(ds,du) := Q(ds,du)- dsdu$ being the compensated PRM.
The process $\{\hat{M}^N_A(t): t\ge 0\}$ is a square-integrable martingale with respect to the filtration 
$\{\sF^N_t: t\ge 0\}$ defined by 
$$
\sF^N_t :=\sigma\Big\{I^N(0), E^N(0), \lambda^0_j(\cdot)_{j\ge1}, \lambda^{0,I}_k(\cdot)_{k\ge1}, \lambda_i(\cdot)_{i\ge1}, Q\big|_{[0,t]\times\R_+} \Big\}\,.
$$
It has the quadratic variation  (see e.g. \cite[Chapter VI]{ccinlar2011probability})
\begin{align*}
\langle\hat{M}^N_A \rangle(t) = N^{-1} \int_0^t \Upsilon^N(s)ds, \quad t \ge 0. 
\end{align*}

Under Assumption \ref{AS-lambda}(i), we have
\begin{equation} \label{eqn-int-Phi-bound}
0 \le N^{-1} \int_s^{t} \Upsilon^N(u) du \le \lambda^*(t-s),  \quad \text{w.p.\,1} \qforq 0 \le s \le t.
\end{equation}
It is shown in Section 4.1 of \cite{FPP2020b} that 
\begin{align} \label{eqn-barPhi-conv}
\int_0^{\cdot} \bar{\Upsilon}^N(s) ds = \int_0^{\cdot} \bar{S}^N(s)\bar{\mathfrak{I}}^N(s)  ds  \RA  \int_0^{\cdot} \bar{S}(s)\bar{\mathfrak{I}}(s)  ds \qinq \bD \qasq N \to \infty. 
\end{align}
and
\begin{align*}
\bar{A}^N \RA \bar{A} =  \int_0^{\cdot} \bar{S}(s)\bar{\mathfrak{I}}(s)  ds \qinq \bD \qasq N \to \infty.
\end{align*}

By \eqref{eqn-An-decomp}, we have
\begin{align*}
\hat{A}^N(t) =\sqrt{N} (\bar{A}^N(t) - \bar{A}(t)) = \hat{M}_A^N(t) + \int_0^t\hat{\Upsilon}^N(s)ds,
\end{align*}
where 
$$
\hat{M}_A^N(t) = \frac{1}{\sqrt{N}}\int_0^t\int_0^\infty{\bf1}_{u\le \Upsilon^N(s^-)}\overline{Q}(ds,du), 
$$
and
\begin{align} \label{eqn-hatPhi-N-rep}
\hat{\Upsilon}^N(t) & = \sqrt{N} (\bar{S}^N(t)\bar{\mathfrak{I}}^N(t)  -\bar{S}(t)\bar{\mathfrak{I}}(t)  ) =  \hat{S}^N(t)  \bar{\mathfrak{I}}^N(t) + \bar{S}(t)  \hat{\mathfrak{I}}^N(t). 
\end{align}

It then follows that
\begin{align} \label{eqn-hatSn-rep}
\hat{S}^N(t) &= \hat{S}^N(0) - \hat{A}^N(t) = - \hat{E}^N(0) - \hat{I}^N(0) - \hat{M}_A^N(t) - \int_0^t\hat{\Upsilon}^N(s)ds. 
\end{align}

{\bf Roadmap of the proofs ahead.}
In what follows, we will first prove the joint convergence of $\big( \hat{\mathfrak{I}}^N_{0,1},  \hat{\mathfrak{I}}^N_{0,2} \big) $ in Lemma \ref{lem:tightJ0}, which is analogous to the proof of CLT for i.i.d. random processes, but with a variation of a random summation term $I^N(0)$ or $E^N(0)$. 
We then prove the joint convergences of $( \hat{\mathfrak{I}}^N_1, \hat{\mathfrak{I}}^N_2)$. 
We observe that both processes  $ \hat{\mathfrak{I}}^N_1$ and  $ \hat{\mathfrak{I}}^N_2$ can be written as stochastic integrals with respect to the associated PRMs: $ \hat{\mathfrak{I}}^N_1$ uses a PRM on $\RR_+^2 \times \bD$, whose mean measure is the product of the Lebesgue mesure on $\R_+^2$ with the law of $\lambda$,  while $ \hat{\mathfrak{I}}^N_2$ uses a standard PRM on $\RR_+^2$.  We employ the moment criterion 
\eqref{moment3points} to establish the tightness of these two processes which requires calculations of moments of the process increments. 
In order to apply the formulas in Lemma \ref{lem:expmoment}, we introduce two associated processes $ \wt{\mathfrak{I}}^N_1$ and  $ \wt{\mathfrak{I}}^N_2$ by replacing the stochastic intensity by a deterministic function, and provide  the required moment estimates in Section \ref{sec-moment}. 
The next crucial step is to prove the finiteness of the second moment of  the supremum of  the processes $\hat{S}^N$ and $\hat{\mathfrak{I}}^N$, which results in the same property of $\hat\Upsilon^N$ (Lemma \ref{lem-hatPhi-2-boundedness}).
  For that, we employ the new criterion in Theorem \ref{th:momentsup} to first establish the bounds for the supremum of the processes  $ \wt{\mathfrak{I}}^N_1$ and  $ \wt{\mathfrak{I}}^N_2$  using  the moment estimates of their increments in Propositions~\ref{prop-wt-mfI-1-inc-moment} and \ref{prop-wt-mfI-2-inc-moment}, and then establish the bounds for the supremum of the differences $ \hat{\mathfrak{I}}^N_1-\wt{\mathfrak{I}}^N_1$ and  $ \hat{\mathfrak{I}}^N_2- \wt{\mathfrak{I}}^N_2$. 
  This estimate of the supremum of processes in Lemma \ref{lem-hatPhi-2-boundedness} will be used in the proof of the asymptotic equivalence of $ (\hat{\mathfrak{I}}^N_1, \hat{\mathfrak{I}}^N_2)$ and $(\wt{\mathfrak{I}}^N_1, \wt{\mathfrak{I}}^N_2)$ in Lemmas \ref{lem-frakI1-diff-conv} and  \ref{lem-frakI2-diff-conv}, as well as the convergence of $\big(\hat{E}^N_1,\hat{I}^N_1, \hat{R}^N_1\big)$ in Lemma \ref{lem-3.12}.

\subsection{Convergence of  $\big( \hat{\mathfrak{I}}^N_{0,1},  \hat{\mathfrak{I}}^N_{0,2} \big)$} \label{sec-mfI-0-proof}

\begin{lemma}\label{lem:tightJ0}
Under Assumptions \ref{AS-lambda}(i) and \ref{AS-FCLT}, 
\begin{equation} \label{eqn-hatfrakI-0-conv}
\big( \hat{\mathfrak{I}}^N_{0,1},  \hat{\mathfrak{I}}^N_{0,2} \big) \RA  \big( \hat{\mathfrak{I}}_{0,1},  \hat{\mathfrak{I}}_{0,2}\big)  \qinq \bD^2 \qasq N \to \infty\,,
\end{equation}
where 
$(\hat{\mathfrak{I}}_{0,1},\hat{\mathfrak{I}}_{0,2})$ is a centered 2-dimensional Gaussian process  with the following covariance functions: for $t, t'\ge 0$, $\Cov( \hat{\mathfrak{I}}_{0,1}(t),  \hat{\mathfrak{I}}_{0,2}(t'))=0$, and 
 \begin{align*}
\Cov( \hat{\mathfrak{I}}_{0,1}(t),  \hat{\mathfrak{I}}_{0,1}(t'))
&= \bar{I}(0) \Cov(\lambda^{0,I}(t), \lambda^{0,I}(t')), \\
\Cov( \hat{\mathfrak{I}}_{0,2}(t),  \hat{\mathfrak{I}}_{0,2}(t'))
&= \bar{E}(0) \Cov(\lambda^0(t), \lambda^0(t')). 
\end{align*}
 
\end{lemma}

\begin{proof}
Define 
\begin{align} \label{eqn-hatfrakI01-def}
 \wt{\mathfrak{I}}^N_{0,1}(t)  := \frac{1}{\sqrt{N}}   \sum_{k=1}^{N \bar{I}(0)} \big( \lambda^{0,I}_k(t)  - \bar{\lambda}^{0,I}(t) \big),
\end{align}
\begin{align} \label{eqn-hatfrakI02-def}
 \wt{\mathfrak{I}}^N_{0,2}(t)  := \frac{1}{\sqrt{N}}   \sum_{j=1}^{N \bar{E}(0)} \big( \lambda^0_j(t)  - \bar{\lambda}^0(t) \big). 
\end{align}
By the CLT for the random elements in $\bD$ (see Theorem 2 in \cite{hahn1978central},  whose conditions (i) and (ii) are satisfied thanks to Assumption \ref{AS-lambda} (i) (a) and (b), respectively) and by the independence of the sequences $\{\lambda^0_j\}_{j\ge 1}$ and $\{\lambda^{0,I}_k\}_{k\ge 1}$,
we obtain 
\begin{equation} \label{eqn-tildefrakI-0-conv}
\big( \wt{\mathfrak{I}}^N_{0,1},  \wt{\mathfrak{I}}^N_{0,2} \big) \RA  \big( \hat{\mathfrak{I}}_{0,1},  \hat{\mathfrak{I}}_{0,2}\big)  \qinq \bD^2 \qasq N \to \infty. 
\end{equation}
It then suffices to show that
\begin{equation} \label{eqn-hatfrakI-0-diff-conv}
\big( \wt{\mathfrak{I}}^N_{0,1}-\hat{\mathfrak{I}}^N_{0,1},  \wt{\mathfrak{I}}^N_{0,2}-\hat{\mathfrak{I}}^N_{0,2} \big) \to  0  \qinq \bD^2 \qasq N \to \infty, 
\end{equation}
in probability. 
We focus on $ \wt{\mathfrak{I}}^N_{0,2}-\hat{\mathfrak{I}}^N_{0,2} \RA  0$. 
It is clear from the definition in \eqref{eqn-hatfrakI02-def} and the i.i.d. property of $\lambda_j^0(\cdot)$ that for each $t\ge 0$,  $\E\big[  \wt{\mathfrak{I}}^N_{0,2}(t)-\hat{\mathfrak{I}}^N_{0,2}(t) \big]=0$, and 
\begin{align*}
\E\big[ \big( \wt{\mathfrak{I}}^N_{0,2}(t)-\hat{\mathfrak{I}}^N_{0,2}(t) \big)^2\big]=v_0(t)  \E\big[ |\bar{E}(0) - \bar{E}^N(0) |\big]\to 0 \qasq N \to\infty\,,
\end{align*}
where 
$v_0(t) <\infty$,
and  the convergence follows from Assumption \ref{AS-FCLT} and the dominated convergence theorem. It then remains to show that $ \{\wt{\mathfrak{I}}^N_{0,2}-\hat{\mathfrak{I}}^N_{0,2}: N \in \NN\}$ is tight in $\bD$. We have
\begin{align} \label{eqn-mfI-02-diff}
\wt{\mathfrak{I}}^N_{0,2}(t)-\hat{\mathfrak{I}}^N_{0,2}(t)
&= \text{sign} (\bar{E}(0) - \bar{E}^N(0)) \frac{1}{\sqrt{N}}   \sum_{j=N( \bar{E}(0)\wedge \bar{E}^N(0))+1}^{N( \bar{E}(0)\vee \bar{E}^N(0))} \big( \lambda^{0}_j(t)  - \bar{\lambda}^{0}(t) \big).
\end{align}
We use the moment criterion \eqref{moment3points}, and consider the moment: for $t'\le t \le t''$, 
\begin{equation} \label{eqn-mfI-02-diff-moment1}
\E\left[ \big| \big( \wt{\mathfrak{I}}^N_{0,2}(t)-\hat{\mathfrak{I}}^N_{0,2}(t) \big) -  \big( \wt{\mathfrak{I}}^N_{0,2}(t')-\hat{\mathfrak{I}}^N_{0,2}(t') \big) \big|^2 \times  \big| \big( \wt{\mathfrak{I}}^N_{0,2}(t)-\hat{\mathfrak{I}}^N_{0,2}(t) \big) -  \big( \wt{\mathfrak{I}}^N_{0,2}(t'')-\hat{\mathfrak{I}}^N_{0,2}(t'') \big) \big|^2   \right]\,. 
\end{equation}
Recall $\breve{\lambda}^0_j(t) = \lambda^0_j(t) - \bar\lambda^0(t)$, and we drop the subscript $j$ for the generic variable  $\breve{\lambda}^0(t)$. 
Then by the i.i.d. and mean zero properties of $\breve{\lambda}^0_j(t)$, and by the independence between $\bar{E}^N(0)$ and $\breve{\lambda}^0_j(t)$, 
we obtain that the moment above is equal to
\begin{align}\label{eqn-mfI-02-diff-moment2} 
 & \frac{1}{N^2}\E \left[  \left( \sum_{j=N( \bar{E}(0)\wedge \bar{E}^N(0))+1}^{N( \bar{E}(0)\vee \bar{E}^N(0))} \big( \breve\lambda^{0}_j(t)  - \breve{\lambda}^{0}_j(t') \big)  \right)^2 \left(\sum_{j=N( \bar{E}(0)\wedge \bar{E}^N(0))+1}^{N( \bar{E}(0)\vee \bar{E}^N(0))} \big( \breve\lambda^{0}_j(t)  - \breve{\lambda}^{0}_j(t'') \big) \right)^2 \right]  \non \\
 & =  \frac{1}{N^2} \bigg( \E\big[N|\bar{E}^N(0)- \bar{E}(0) | \big] \E \big[\big( \breve\lambda^{0}(t)  - \breve{\lambda}^{0}(t') \big)^2  \big( \breve\lambda^{0}(t)  - \breve{\lambda}^{0}(t'') \big)^2 \big] \non \\
 & \quad + \E\big[N|\bar{E}^N(0) -\bar{E}(0) |  (N|\bar{E}^N(0)- \bar{E}(0) |-1)\big] \E \big[\big( \breve\lambda^{0}(t)  - \breve{\lambda}^{0}(t') \big)^2 \big] \E\big[\big( \breve\lambda^{0}(t)  - \breve{\lambda}^{0}(t'') \big)^2 \big] \non \\
 & \quad + 2 \E\big[N|\bar{E}^N(0)- \bar{E}(0) |  (N|\bar{E}^N(0)- \bar{E}(0) |-1)\big]  \Big( \E \big[\big( \breve\lambda^{0}(t)  - \breve{\lambda}^{0}(t') \big) \big( \breve\lambda^{0}(t)  - \breve{\lambda}^{0}(t'') \big) \big] \Big)^2 \bigg) \non \\
 & \le \frac{1}{N^2} \bigg( \E\big[N|\bar{E}^N(0)- \bar{E}(0) | \big] \E \big[\big( \breve\lambda^{0}(t)  - \breve{\lambda}^{0}(t') \big)^2  \big( \breve\lambda^{0}(t)  - \breve{\lambda}^{0}(t'') \big)^2 \big] \non \\
 & \quad + 3\E\big[N|\bar{E}^N(0) - \bar{E}(0) |  (N|\bar{E}^N(0) - \bar{E}(0) |-1)\big] \E \big[\big( \breve\lambda^{0}(t)  - \breve{\lambda}^{0}(t') \big)^2 \big] \E\big[\big( \breve\lambda^{0}(t)  - \breve{\lambda}^{0}(t'') \big)^2 \big] \bigg) \non \\
 & \le \frac{1}{N}\E\big[|\bar{E}^N(0)- \bar{E}(0) | \big]  (\psi(t'')-\psi(t'))^\beta+  3 \E\big[|\bar{E}^N(0)- \bar{E}(0) |^2\big] (\phi(t'')-\phi(t'))^{2\alpha},
\end{align}
where we have used Cauchy-Schwarz inequality in the first inequality, and conditions (a) and (b) in Assumption \ref{AS-lambda} (i)  in the second inequality. By Assumption \ref{AS-FCLT},  $\E\big[|\bar{E}^N(0)- \bar{E}(0) | \big] \to 0$ and $\E\big[|\bar{E}^N(0)- \bar{E}(0) |^2\big]\to 0$ as $N\to\infty$.  Thus we conclude that the tightness of $\{\wt{\mathfrak{I}}^N_{0,1}-\hat{\mathfrak{I}}^N_{0,1}: N \in \NN\}$
follows from  Theorem \ref{13.5}. This completes the proof. 
\end{proof}

\subsection{Moment estimates associated with the processes $ \hat{\mathfrak{I}}^N_1$, $ \hat{\mathfrak{I}}^N_2$ and $\hat\Upsilon^N$} \label{sec-moment}

We first consider $ \hat{\mathfrak{I}}^N_1$. 
Let us introduce the canonical probability space associated to the random function $\lambda$. 
If we equip $\bD$ with its Borel $\sigma$--field, and the probability measure which is the law of 
the random function $\lambda$, then  the identity mapping $\lambda\mapsto\lambda$  from $\bD$ into itself defines our random function $\lambda$. So that below $\lambda$ will be both the same object as defined above, and the  
generic element in $\bD$. As such, it will also be a dummy variable in the next integral.
We introduce a PRM  $\breve{Q}$ on $\R_+\times \bD\times \R_+$, 
which to the point $\tau^N_i$ associates the copy $\lambda_i$ of the random function $\lambda$, so that the mean measure of the PRM is 
\[ ds\times\text{ Law of }\lambda\times du\,.\]
With the notation $\breve{\lambda}:=\lambda-\bar{\lambda}$, $\hat{\mathfrak{I}}^N_1$ can be written
\begin{align} \label{eqn-hatfrakI1-PRM-1}
 \hat{\mathfrak{I}}^N_1(t) =N^{-1/2}\int_0^t\int_\bD\int_0^\infty \breve{\lambda}(t-s){\bf1}_{u\le \Upsilon^N(s^-)} \breve{Q}(ds,d\lambda,du)\,.
 \end{align}
We note that if we replace in the above $\breve{Q}$ by its mean measure, then the resulting integral vanishes.
Consequently we also have
\begin{align}\label{eqn-hatfrakI1-PRM-2}
 \hat{\mathfrak{I}}^N_1(t) =N^{-1/2}\int_0^t\int_\bD\int_0^\infty \breve{\lambda}(t-s){\bf1}_{u\le \Upsilon^N(s^-)} \wt{Q}(ds,d\lambda,du)\, ,
 \end{align}
where $\wt{Q}$ is the compensated PRM of $\breve{Q}$. Hence $\E[ \hat{\mathfrak{I}}^N_1(t)]=0$
and
\begin{equation}\label{moment2}
\E\big[(\hat{\mathfrak{I}}^N_1(t))^2\big]
= \E \left[ \int_0^t \breve{\lambda}^2(t-s) \bar{\Upsilon}^N(s)ds \right]\,. 
\end{equation}

 Let 
\begin{align} \label{eqn-wtmfI-1-def}
 \wt{\mathfrak{I}}^N_1(t):=N^{-1/2}\int_0^t\int_\bD\int_0^\infty \breve{\lambda}(t-s) {\bf1}_{u\le N\bar{\Upsilon}(s)} \wt{Q}(ds,d\lambda,du) .
 \end{align}
In the next proposition, we prove the moment bound for the increments of  $ \wt{\mathfrak{I}}^N_1$,
 which will be used in the proof of its convergence. 
Note that by \eqref{eqn-int-Phi-bound} and \eqref{eqn-barPhi-conv}, we have  $\bar\Upsilon(t) \le \lambda^*$ for any $t\ge 0$.

\begin{prop} \label{prop-wt-mfI-1-inc-moment}
There exist $C>0$ and $\beta>0$
such that for any $r<t<v$,  
\begin{equation} \label{eqn-wt-mfI-1-inc-moment}
\E \left[\big| \wt{\mathfrak{I}}^N_1(t) -  \wt{\mathfrak{I}}^N_1(r) \big|^2 \big| \wt{\mathfrak{I}}^N_1(t) -  \wt{\mathfrak{I}}^N_1(v) \big|^2 \right]\le C(v-r)^{1+\beta}\,. 
\end{equation}
\end{prop}

\begin{proof}

We deduce from \eqref{mixed2moment} the following identity
\begin{equation}\label{estim2+2}
\begin{split}
\E &\left[\big| \wt{\mathfrak{I}}^N_1(t) -  \wt{\mathfrak{I}}^N_1(r) \big|^2 \big| \wt{\mathfrak{I}}^N_1(t) -  \wt{\mathfrak{I}}^N_1(v) \big|^2 \right] =
\frac{1}{N}\E\int_r^t\breve{\lambda}^2(t-s)\left[\breve{\lambda}(v-s)-\breve{\lambda}(t-s)\right]^2\bar\Upsilon(s)ds\\
&\quad+\frac{1}{N}\E\int_0^r\left[\breve{\lambda}(t-s)-\breve{\lambda}(r-s)\right]^2\left[\breve{\lambda}(v-s)-\breve{\lambda}(t-s)\right]^2\bar{\Upsilon}(s)ds\\
&\quad+2\bigg(\E\int_r^t\breve{\lambda}(t-s)\left[\breve{\lambda}(v-s)-\breve{\lambda}(t-s)\right]\bar\Upsilon(s)ds\\
&\quad\quad\quad
+\E\int_0^r\left[\breve{\lambda}(t-s)-\breve{\lambda}(r-s)\right]\left[\breve{\lambda}(v-s)-\breve{\lambda}(t-s)\right]\bar{\Upsilon}(s)ds\bigg)^2\\
&\quad+ \left(\E\int_r^t\breve{\lambda}^2(t-s)\bar\Upsilon(s)ds+\E\int_0^r\left[\breve{\lambda}(t-s)-\breve{\lambda}(r-s)\right]^2\bar\Upsilon(s)ds\right)\\
&\quad\quad\quad\times\left(\E\int_t^v\breve{\lambda}^2(v-s)\bar\Upsilon(s)ds+\E\int_0^t\left[\breve{\lambda}(v-s)-\breve{\lambda}(t-s)\right]^2\bar{\Upsilon}(s)ds\right)\,.
\end{split}
\end{equation}
We now bound each of the four terms on the right of \eqref{estim2+2}. It follows from Lemma \ref{lem-barlambda-inc-bound} and \eqref{eqn-lambda-inc} and \eqref{hypF} in Assumption \ref{AS-lambda} that for some constant $C$, the first term on the right of \eqref{estim2+2} is bounded by  
\[ C\frac{(\lambda^\ast)^3}{N}(t-r)((v-t)^{2\alpha}+(v-t)^{\rho})\lesssim C(v-r)^{1+\rho\wedge(2\alpha)}\,.\]
We next consider the third term. The absolute value of the first summand in the square  is clearly bounded by 
$(\lambda^\ast)^3(t-r)$. The second summand is bounded by $\lambda^\ast r (t-r)^\alpha (v-t)^\alpha$, plus
\[ C\sum_{j=1}^{k-1}\int_0^r (F_j(t-s)-F_j(r-s)+F_j(v-s)-F_j(t-s))ds\le 2Ck[ (t-r)+(v-t)],\]
where we have used twice \eqref{estim-int-F}. Finally the third term is bounded by a constant times
\[(v-r)^{4\alpha}+(v-r)^2=2(v-r)^{1+(4\alpha-1)\wedge1}\,.\]
By similar arguments, the first factor in the last term is bounded by a constant times
\[ t-r+(t-r)^{2\alpha}+\sum_{j=1}^{k-1}\int_0^r [F_j(t-s)-F_j(r-s)]ds\le t-r+(t-r)^{2\alpha}+k(t-r)\,.\]
Since the second factor can be estimated in the same way, with however $t-r$ replaced by $v-t$, we conclude that the last term is again bounded by a constant times
$(v-r)^2+(v-r)^{4\alpha}$.

It remains to consider the second term, which is a bit more delicate. We disregard the factor $1/N$.  
Define $A_i(s)=\{r-s<\xi^i\le t-s\}$, $B_j(s)=\{t-s<\xi^j\le v-s\}$, and note that $\P(A_i(s))=F_i(t-s)-F_i(r-s)$,
$\P(B_j(s))=F_j(v-s)-F_j(t-s)$. We disregard the factor $1/N$.  The integrand in the integral from $s=0$ to $s=r$ is bounded from above by a constant times
\begin{align*}
&\left[(t-r)^{2\alpha}+\sum_{i=1}^{k-1}(\bone_{A_i(s)}-\P(A_i(s)))^2\right]\times
\left[(v-t)^{2\alpha}+\sum_{i=1}^{k-1}(\bone_{B_j(s)}-\P(B_j(s)))^2\right]\\
&\le (v-r)^{4\alpha}+(v-r)^{2\alpha}\sum_{i=1}^{k-1}(\bone_{A_i(s)}-\P(A_i(s)))^2\\
&\quad+
(t-r)^{2\alpha}\sum_{i=1}^{k-1}(\bone_{B_j(s)}-\P(B_j(s)))^2\\
&\quad+\sum_{i,j=1}^{k-1}(\bone_{A_i(s)}-\P(A_i(s)))^2(\bone_{B_j(s)}-\P(B_j(s)))^2
\end{align*}
The first term on the right hand side of the last inequality is bounded by $(v-r)^{4\alpha}$. 
The expectation of the second term is bounded by $(t-r)^{2\alpha}\P(A_i(s))$, whose integral from $s=0$ to $s=r$ is bounded by 
$(t-r)(v-r)^{2\alpha}$, and the next term is treated exactly in the same way. We finally treat the last term.
We note that
\[ (\bone_{A_i(s)}-\P(A_i(s)))^2(\bone_{B_j(s)}-\P(B_j(s)))^2\]
is the square of a real number which is between $-1$ and $+1$, hence it is bounded by its absolute value.
Consequently
\begin{align*}
\E&\left\{(\bone_{A_i(s)}-\P(A_i(s)))^2(\bone_{B_j(s)}-\P(B_j(s)))^2\right\}\\
&\le \E\left(\bone_{A_i(s)\cap B_j(s)}+\bone_{A_i(s)}\P(B_j(s))+\bone_{B_j(s)}\P(A_i(s))+\P(A_i(s))\P(B_j(s))\right)\\
&=\P(A_i(s)\cap B_j(s))+3\P(A_i(s))\P(B_j(s))\,.
\end{align*}
It remains to bound each of those two last terms. Note that $\P(A_i(s)\cap B_j(s))\not=0$ only if $i<j\le k-1$, and in that case
\begin{align*}
\P(A_i(s)\cap B_j(s))&\le \P(A_i(s))\times\P(B_j(s)|A_i(s))\\
&\le\P(A_i(s))\times \sup_{u}\P(\xi_{i+1}-\xi_i\le v-r|\xi_i=u)\\
&\le C'(v-r)^\rho\P(A_i(s)),
\end{align*}
thanks to \eqref{increments} in Assumption \ref{AS-lambda}.
Since by \eqref{estim-int-F}, $\int_0^r\P(A_i(s))ds\le t-r$, the integral from $s=0$ to $s=r$ of the first term has the right bound. The second term is bounded by the same
expression, using this time \eqref{hypF} instead of \eqref{increments}.
\end{proof}

Concerning the process $\hat{\mathfrak{I}}^N_2(t)$, we can represent it as
\begin{equation}\label{hatJ2}
\hat{\mathfrak{I}}^N_2(t):=\frac{1}{\sqrt{N}}\int_0^t\int_0^\infty \bar\lambda(t-s)\bone_{u\le\Upsilon^N(s^-)}\overline{Q}(ds,du)
\,.
\end{equation}
Define
\begin{equation}\label{wtJ2}
\wt{\mathfrak{I}}^N_2(t):=\frac{1}{\sqrt{N}}\int_0^t\int_0^\infty \bar\lambda(t-s)\bone_{u\le N\bar\Upsilon(s)}\overline{Q}(ds,du)
\,.
\end{equation}
In the next proposition, we prove the moment bound for the increments of $\wt{\mathfrak{I}}^N_2$.

\begin{prop} \label{prop-wt-mfI-2-inc-moment}
There exist $C>0$ and $\beta>0$
such that for any $r<t<v$,  
\begin{equation} \label{eqn-wt-mfI-2-inc-moment}
\E \left[\big| \wt{\mathfrak{I}}^N_2(t) -  \wt{\mathfrak{I}}^N_2(r) \big|^2 \big| \wt{\mathfrak{I}}^N_2(t) -  \wt{\mathfrak{I}}^N_2(v) \big|^2 \right]\le C(v-r)^{1+\beta}\,. 
\end{equation}
\end{prop}

\begin{proof}
As in the proof of Proposition \ref{prop-wt-mfI-2-inc-moment}, we exploit \eqref{mixed2moment}.
We obtain 
\begin{align} \label{eqn-wt-mfI2-inc-22}
& \E \left[\big| \wt{\mathfrak{I}}^N_2(t) -  \wt{\mathfrak{I}}^N_2(r) \big|^2 \big| \wt{\mathfrak{I}}^N_2(t) -  \wt{\mathfrak{I}}^N_2(v) \big|^2 \right]  \non\\
&= 
\frac{1}{N}\int_r^t\bar{\lambda}^2(t-s)\left[\bar{\lambda}(v-s)-\bar{\lambda}(t-s)\right]^2\bar\Upsilon(s)ds \non\\
&\quad+\frac{1}{N}\int_0^r\left[\bar{\lambda}(t-s)-\bar{\lambda}(r-s)\right]^2\left[\bar{\lambda}(v-s)-\bar{\lambda}(t-s)\right]^2\bar{\Upsilon}(s)ds \non \\
&\quad+2\bigg(\int_r^t\bar{\lambda}(t-s)\left[\bar{\lambda}(v-s)-\bar{\lambda}(t-s)\right]\bar\Upsilon(s)ds \non \\
&\quad\quad\quad
+\int_0^r\left[\bar{\lambda}(t-s)-\bar{\lambda}(r-s)\right]\left[\bar{\lambda}(v-s)-\bar{\lambda}(t-s)\right]\bar{\Upsilon}(s)ds\bigg)^2 \non \\
&\quad+ \left(\int_r^t\bar{\lambda}^2(t-s)\bar\Upsilon(s)ds+\int_0^r\left[\bar{\lambda}(t-s)-\bar{\lambda}(r-s)\right]^2\bar\Upsilon(s)ds\right) \non \\
&\quad\quad\quad\times\left(\int_t^v\bar{\lambda}^2(v-s)\bar\Upsilon(s)ds+\int_0^t\left[\bar{\lambda}(v-s)-\bar{\lambda}(t-s)\right]^2\bar{\Upsilon}(s)ds\right)\,.
\end{align}
For the first term on the right hand side, by Lemma \ref{lem-barlambda-inc-bound} as well as the assumptions in   \eqref{eqn-lambda-inc} and \eqref{hypF},   we obtain  the upper bound: for some constant $C>0$,
\begin{align*}
&C\frac{ (\lambda^*)^3}{N}(t-r)\left[ (v-t)^{2\alpha} + (v-t)^{2\rho} \right] \lesssim C (v-r)^{1+ 2(\alpha \wedge \rho)}. 
\end{align*}
For the second term, we obtain the upper bound 
\begin{align*}
 C T \frac{\lambda^*}{N}  &\left[ (t-r)^{4\alpha} + (\lambda^*)^4 (t-r)^\rho \sum_{j=1}^{k-1}\int_0^r(F_j(t-s)-F_j(r-s))ds\right]\\
& \lesssim C (v-r)^{ 4\alpha \wedge (1+\rho)},
\end{align*}
where we have used \eqref{hypF} and \eqref{estim-int-F}.
For the third term, we obtain
\begin{align*}
&  C \Big(   t-r +   (t-r)^{2\alpha} + t-r  \Big)^2 
 \lesssim C  (v-r)^{ 4\alpha \wedge 2} 
\end{align*}
Finally, we get the following bound for the last term, 
\begin{align*}
& \Big((\lambda^*)^3 (t-r) + \lambda^* T C\left[ (t-r)^{2\alpha} + t-r  \right]  \Big)\times \Big( (\lambda^*)^3(v-t) + \lambda^* T C \left[ (v-t)^{2\alpha} + v-t \right]  \Big) \\
& \le \Big( (\lambda^*)^3(v-r) + \lambda^* T C\left[ (v-r)^{2\alpha} + v-r \right]  \Big)^2 \\
& \lesssim C  (v-r)^{4\alpha\wedge2}. 
\end{align*}
Therefore we have obtain the desired upper bound, with $\beta=(4\alpha-1)\wedge\rho\wedge 1\wedge2\alpha$. 
\end{proof}

\begin{lemma} \label{lem-hatPhi-2-boundedness}
For any $T>0$, the following results hold
\begin{equation}\label{eqn-hatSI-2-boundedness-sup-inside}
\sup_N\E\bigg[\sup_{0\le t\le T}\hat{S}^N(t)^2\bigg]<\infty,\qandq
      \sup_N\E\bigg[\sup_{0\le t\le T}\hat{\mathfrak{I}}^N(t)^2\bigg]<\infty\,,
      \end{equation}
and consequently,
\begin{align} \label{eqn-hatPhi-2-boundedness-sup-inside}
\sup_N \E\bigg[ \sup_{t\in [0,T]}\hat{\Upsilon}^N(t)^2\bigg] <\infty.
\end{align}
\end{lemma}

\begin{proof}
We first show that 
for any $T>0$, 
\begin{equation} \label{eqn-hatSI-2-boundedness-sup}
\sup_N\sup_{0\le t\le T}\E\big[\hat{S}^N(t)^2\big]<\infty,\qandq
      \sup_N\sup_{0\le t\le T}\E\big[\hat{\mathfrak{I}}^N(t)^2\big]<\infty\,,
      \end{equation}
       which combined with \eqref{eqn-hatPhi-N-rep} implies that 
\begin{align} \label{eqn-hatPhi-2-boundedness-sup}
\sup_N \sup_{t\in [0,T]} \E\big[\hat{\Upsilon}^N(t)^2\big] <\infty.
\end{align}
We shall use \eqref{eqn-hatPhi-N-rep}  and the two integral representations in \eqref{eqn-hatfrakI-rep} and \eqref{eqn-hatSn-rep}.
We first obtain the following estimates. It is clear that $\sup_N \sup_{t\in [0,T]} \E\big[|\hat{M}^N_A(t)|^2\big] \le \lambda^* T$  and from Assumption \ref{AS-FCLT}, there exists a constant $C>0$ such that for all $N$, 
$$\sup_N\sup_{t \in [0,T]} \E\big[\big(\hat{I}^N(0) \bar{\lambda}^{0,I}(t)\big)^2\big] \le (\lambda^*)^2 C, \qandq\sup_N\sup_{t \in [0,T]} \E\big[\big(\hat{E}^N(0) \bar{\lambda}^0(t)\big)^2\big] \le (\lambda^*)^2 C.$$ 
It is also clear that 
\begin{align*}
&\sup_N\sup_{t \in [0,T]} \E\big[\big(\hat{\mathfrak{I}}^N_{0,1}(t) \big)^2\big]  = \sup_N\sup_{t \in [0,T]}  \E[\bar{I}^N(0)] v^{0,I}(t) \le \sup_{t \in [0,T]} v^{0,I}(t) <\infty, \\
&\sup_N\sup_{t \in [0,T]} \E\big[\big(\hat{\mathfrak{I}}^N_{0,2}(t) \big)^2\big]  = \sup_N\sup_{t \in [0,T]}  \E[\bar{E}^N(0)] v^{0}(t) \le \sup_{t \in [0,T]} v^{0}(t) <\infty. 
\end{align*}
By \eqref{moment2} and  \eqref{eqn-int-Phi-bound},  we have 
\begin{align*} 
\sup_{t\in [0,T]} \E\big[|\hat{\mathfrak{I}}^N_1(t) |^2\big]  & =\sup_{t\in [0,T]} 
 \E \left[ \int_0^t \breve{\lambda}^2(t-s) \bar{\Upsilon}^N(s)ds \right]
\le  \lambda^*  \int_0^T v(s) d s < \infty,
\end{align*}
and 
again by \eqref{eqn-int-Phi-bound}, we obtain
\begin{align*}
\sup_{t\in [0,T]} \E\big[|\hat{\mathfrak{I}}^N_2(t) |^2\big] 
& =\sup_{t\in [0,T]} \E\Big[ \int_0^t  \bar{\lambda}(t-s)^2  \bar\Upsilon^N(s)ds \Big] \le  (\lambda^*)^3 T\,.
\end{align*}

Combining  \eqref{eqn-hatfrakI-rep}, \eqref{eqn-hatPhi-N-rep} and \eqref{eqn-hatSn-rep} with the simple bounds $\bar{S}(t) \le 1$ and $\bar{\mathfrak{I}}^N(t) \le \lambda^* (\bar{I}^N(0) + \bar{A}^N(t)) \le 2 \lambda^*$, and Gronwall's inequality, we obtain the claims in \eqref{eqn-hatSI-2-boundedness-sup}. 

\smallskip

We next prove \eqref{eqn-hatSI-2-boundedness-sup-inside}. 
By Doob's maximal inequality, we have  $\sup_N  \E\big[\sup_{t\in [0,T]}|\hat{M}^N_A(t)|^2\big] \le 4\sup_N  \E\big[|\hat{M}^N_A(T)|^2\big] \le 4 \lambda^* T $, and then by \eqref{eqn-hatSn-rep} and \eqref{eqn-hatPhi-2-boundedness-sup}, and by applying Gronwall's inequality, we obtain the moment property for $\hat{S}^N$ in \eqref{eqn-hatSI-2-boundedness-sup-inside} holds. 

For $\hat{\mathfrak{I}}^N(t)$, we first consider the two processes $\hat{\mathfrak{I}}^N_{0,1}(t)$ and $\hat{\mathfrak{I}}^N_{0,2}(t)$,  which can be treated in the same way, so we focus on $\hat{\mathfrak{I}}^N_{0,2}(t)$ as in the proof of Lemma \ref{lem:tightJ0}. 
Recall $ \tilde{\mathfrak{I}}^N_{0,2}(t) $ in \eqref{eqn-hatfrakI02-def}.  Under Assumption \ref{AS-lambda} (ii), 
we deduce from a computation similar to the one leading to \eqref{eqn-mfI-02-diff-moment2} and
Theorem \ref{th:momentsup} that  
\begin{align} \label{eqn-mfI02-moment-sup}
\sup_N \E\bigg[\sup_{t\in [0,T]}\tilde{\mathfrak{I}}_{0,2}^N(t)^2\bigg] <\infty.
\end{align}
We next consider $ \hat{\mathfrak{I}}^N_{0,2}(t)- \tilde{\mathfrak{I}}^N_{0,2}(t)$ which is given in \eqref{eqn-mfI-02-diff}. By 
\eqref{eqn-mfI-02-diff-moment1} and \eqref{eqn-mfI-02-diff-moment2}, applying again Theorem \ref{th:momentsup}, we obtain that
the moment property in \eqref{eqn-mfI02-moment-sup} holds for  $ \hat{\mathfrak{I}}^N_{0,2}(t)- \tilde{\mathfrak{I}}^N_{0,2}(t)$. 
Combining these two moment estimates, we deduce that the moment property in \eqref{eqn-mfI02-moment-sup} holds for $\hat{\mathfrak{I}}^N_{0,2}(t)$. 

For the processes  $\hat{\mathfrak{I}}^N_{1}(t)$, we write it as $\hat{\mathfrak{I}}^N_{1}(t) = \wt{\mathfrak{I}}^N_{1}(t) + (\hat{\mathfrak{I}}^N_{1}(t)- \wt{\mathfrak{I}}^N_{1}(t))$ and treat the decomposed terms separately.  By Proposition \ref{prop-wt-mfI-1-inc-moment}, 
 applying Theorem \ref{th:momentsup}, we obtain that
the moment property in \eqref{eqn-mfI02-moment-sup} holds for  $ \wt{\mathfrak{I}}^N_{1}(t)$. 
Now for the difference $\hat{\mathfrak{I}}^N_{1}(t)- \wt{\mathfrak{I}}^N_{1}(t)$, we have
\begin{align} \label{eqn-mfI-diff-1}
\hat{\mathfrak{I}}^N_1(t) - \wt{\mathfrak{I}}^N_1(t)
 &= \frac{1}{\sqrt{N}}\int_0^t  \int_{\bD} \int_{N(\bar\Upsilon^N(s) \wedge \bar\Upsilon(s))}^{N (\bar\Upsilon^N(s) \vee  \bar\Upsilon(s))} 
\breve{\lambda}(t-s)  \text{sign}(\bar\Upsilon^N(s) - \bar\Upsilon(s)) \breve{Q}(ds,d\lambda,du) \,.
\end{align} 
Observe that
\begin{align*}
\big|\hat{\mathfrak{I}}^N_1(t) - \wt{\mathfrak{I}}^N_1(t)\big|
 &\le \frac{1}{\sqrt{N}} ( 2\lambda^*) \int_0^t  \int_{\bD} \int_{N(\bar\Upsilon^N(s) \wedge \bar\Upsilon(s))}^{N (\bar\Upsilon^N(s) \vee  \bar\Upsilon(s))}  \breve{Q}(ds,d\lambda,du) \,,
\end{align*}
is increasing in $t$. Thus,
\begin{align*}
\E\bigg[\sup_{t\in [0,T]}\big|\hat{\mathfrak{I}}^N_1(t) - \wt{\mathfrak{I}}^N_1(t)\big|^2 \bigg]
& \le  \E\bigg[ \bigg(  \frac{1}{\sqrt{N}}( 2 \lambda^*) \int_0^T  \int_{\bD} \int_{N(\bar\Upsilon^N(s) \wedge \bar\Upsilon(s))}^{N (\bar\Upsilon^N(s) \vee  \bar\Upsilon(s))}  \breve{Q}(ds,d\lambda,du) \bigg)^2 \bigg] \\
& \le 2  \E\bigg[ \bigg(  \frac{1}{\sqrt{N}}( 2 \lambda^*) \int_0^T  \int_{\bD} \int_{N(\bar\Upsilon^N(s) \wedge \bar\Upsilon(s))}^{N (\bar\Upsilon^N(s) \vee  \bar\Upsilon(s))}  \widetilde{Q}(ds,d\lambda,du) \bigg)^2 \bigg] \\
& \quad + 2 \E \bigg[ \bigg( \sqrt{N} (2 \lambda^*) \int_0^T   \big|\bar{\Upsilon}^N(s) - \bar{\Upsilon}(s) \big| ds \bigg)^2\bigg] \\
&= 8(\lambda^*)^2 \bigg( \int_0^T \E \big[ \big|\bar\Upsilon^N(s) -  \bar\Upsilon(s) \big| \big] ds  + \E \bigg[ \bigg(\int_0^T |\hat\Upsilon^N(s)| \bigg)^2 \bigg]  \bigg).
\end{align*}
Here the first integral converges to zero as $N\to\infty$, and the second term is bounded by \eqref{eqn-hatPhi-2-boundedness-sup}.  Thus combining these, we have shown that the moment property in \eqref{eqn-mfI02-moment-sup} holds for $\hat{\mathfrak{I}}^N_{1}(t)$.

Similarly, 
we write $\hat{\mathfrak{I}}^N_{2}(t) = \wt{\mathfrak{I}}^N_{2}(t) + (\hat{\mathfrak{I}}^N_{2}(t)- \wt{\mathfrak{I}}^N_{2}(t)) $ and treat the decomposed terms separately. 
We obtain the moment property in \eqref{eqn-mfI02-moment-sup} for $\wt{\mathfrak{I}}^N_{2}(t)$ 
by Proposition \ref{prop-wt-mfI-2-inc-moment} and  Theorem \ref{th:momentsup}.
For the difference $\hat{\mathfrak{I}}^N_{2}(t)- \wt{\mathfrak{I}}^N_{2}(t)$, we have 
\begin{align} \label{eqn-mfI-diff-2}
\hat{\mathfrak{I}}^N_2(t) - \wt{\mathfrak{I}}^N_2(t)
 &= \frac{1}{\sqrt{N}}\int_0^t \int_{N(\bar\Upsilon^N(s) \wedge \bar\Upsilon(s))}^{N (\bar\Upsilon^N(s) \vee  \bar\Upsilon(s))} 
 \bar\lambda(t-s)  \text{sign}(\bar\Upsilon^N(s) - \bar\Upsilon(s)) Q(ds,du)  \non\\
 & \quad - \sqrt{N}\int_0^t  \bar\lambda(t-s) \big(\bar{\Upsilon}^N(s) - \bar{\Upsilon}(s) \big) ds\,.
\end{align}
Observe that 
\begin{align*}
\big|\hat{\mathfrak{I}}^N_2(t) - \wt{\mathfrak{I}}^N_2(t) \big|
 &\le  \frac{1}{\sqrt{N}} \lambda^* \int_0^t \int_{N(\bar\Upsilon^N(s) \wedge \bar\Upsilon(s))}^{N (\bar\Upsilon^N(s) \vee  \bar\Upsilon(s))}  Q(ds,du)  +  \sqrt{N} \lambda^* \int_0^t  \big|\bar{\Upsilon}^N(s) - \bar{\Upsilon}(s) \big| ds\,.
\end{align*}
Both terms on the right hand side are increasing in $t$. Thus, 
\begin{align*}
\E\bigg[\sup_{t\in [0,T]}\big|\hat{\mathfrak{I}}^N_2(t) - \wt{\mathfrak{I}}^N_2(t)\big|^2 \bigg]
& \le 2 \E \bigg[ \bigg( \frac{1}{\sqrt{N}} \lambda^* \int_0^T \int_{N(\bar\Upsilon^N(s) \wedge \bar\Upsilon(s))}^{N (\bar\Upsilon^N(s) \vee  \bar\Upsilon(s))}  Q(ds,du) \bigg)^2\bigg]   \\
& \qquad + 2 
\E \bigg[ \bigg( \sqrt{N} \lambda^* \int_0^T   \big|\bar{\Upsilon}^N(s) - \bar{\Upsilon}(s) \big| ds \bigg)^2\bigg]  \\
& = 2 (\lambda^*)^2 \bigg( \int_0^T \E \big[ \big|\bar\Upsilon^N(s) -  \bar\Upsilon(s) \big| \big] ds  + \E \bigg[ \bigg(\int_0^T |\hat\Upsilon^N(s)| \bigg)^2 \bigg]  \bigg).
\end{align*}
Thus, similar to above, we obtain  the moment property in \eqref{eqn-mfI02-moment-sup} holds for $\hat{\mathfrak{I}}^N_{2}(t)$.  
Finally, combining \eqref{eqn-hatfrakI-rep} with the above estimates yields that $\hat{\mathfrak{I}}^N$ 
satisfies \eqref{eqn-hatSI-2-boundedness-sup-inside}. 
\end{proof}

\subsection{Joint convergence of $\hat{\mathfrak{I}}^N_1$ and $\hat{\mathfrak{I}}^N_2$} \label{sec-mfI-1-proof}

In this subsection, we will
 show the following result.
 \begin{lemma} \label{lem-mfI12-conv}
 Under Assumptions \ref{AS-lambda} (ii) and \ref{AS-FCLT}, 
\begin{align*}
(\hat{\mathfrak{I}}^N_1,\hat{\mathfrak{I}}^N_2)\Rightarrow
 (\hat{\mathfrak{I}}_1,\hat{\mathfrak{I}}_2) \qinq \bD^2 \qasq N \to \infty,
\end{align*}
where $(\hat{\mathfrak{I}}_1,\hat{\mathfrak{I}}_2)$  is a centered Gaussian process with covariance functions: for $t,t'\ge 0$, 
$\Cov( \hat{\mathfrak{I}}_1(t),  \hat{\mathfrak{I}}_2(t'))=0$ and 
 \begin{align*}
\Cov( \hat{\mathfrak{I}}_1(t),  \hat{\mathfrak{I}}_1(t'))
&= \int_0^{t\wedge t'} \Cov(\lambda(t-s), \lambda(t'-s)) \bar{S}(s)\bar{\mathfrak{I}}(s)  ds,\\
\Cov( \hat{\mathfrak{I}}_2(t),  \hat{\mathfrak{I}}_2(t')) &=  \int_0^{t\wedge t'}  \bar{\lambda}(t-s) \bar{\lambda}(t'-s)  \bar{S}(s)\bar{\mathfrak{I}}(s)  ds. 
\end{align*}
 \end{lemma}

 We shall first show that 
 $(\wt{\mathfrak{I}}^N_1,\wt{\mathfrak{I}}^N_2)\Rightarrow
 (\hat{\mathfrak{I}}_1,\hat{\mathfrak{I}}_2)$ in $\bD^2$, as $N\to\infty$. Given the results in 
 Propositions \ref{prop-wt-mfI-1-inc-moment} 
 and \ref{prop-wt-mfI-2-inc-moment}, Theorem \ref{13.5} tells us that it remains to prove that the finite dimensional distributions of  $(\wt{\mathfrak{I}}^N_1,\wt{\mathfrak{I}}^N_2)$ converge to those $(\hat{\mathfrak{I}}_1,\hat{\mathfrak{I}}_2)$, which we establish in the first Lemma which follows. It will then remain to prove that  $\hat{\mathfrak{I}}^N_1-\wt{\mathfrak{I}}^N_1\to0$ and $\hat{\mathfrak{I}}^N_2-\wt{\mathfrak{I}}^N_2\to0$ in $\bD$ in probability, as $N\to\infty$ which will be done in the next two Lemmas.
 
 \begin{lemma} \label{lem-hatmfI1-2}
For any $k\ge1$, $0<t_1<t_2<\cdots<t_k$,  as $N\to\infty$,
\begin{equation} \label{eqn-hatfrakI1-conv}
 \bigg((\wt{\mathfrak{I}}^N_1(t_1),\wt{\mathfrak{I}}^N_2(t_1)),\ldots,(\wt{\mathfrak{I}}^N_1(t_k),\wt{\mathfrak{I}}^N_2(t_k))\bigg)\RA   
 \bigg((\hat{\mathfrak{I}}_1(t_1),\hat{\mathfrak{I}}_2(t_1)),\ldots,(\hat{\mathfrak{I}}_1(t_k),\hat{\mathfrak{I}}_2(t_k)) \bigg)
 \end{equation} 
 in $\R^{2k}$,
where $(\hat{\mathfrak{I}}_1,\hat{\mathfrak{I}}_2)$ is given in 
Lemma \ref{lem-mfI12-conv}.  
\end{lemma}

\begin{proof}
Recall \eqref{eqn-wtmfI-1-def} and \eqref{wtJ2}.
Note that the process $ \wt{\mathfrak{I}}^N_2(t) $ 
can be equivalently
written as follows using $\wt{Q}$: 
\begin{align}  \label{eqn-wtJ2-wtQ}
 \wt{\mathfrak{I}}^N_2(t) & = \frac{1}{\sqrt{N}} 
 \int_0^t\int_\bD \int_0^\infty   \bar{\lambda}(t-s) \bone_{u \le N \bar\Upsilon(s)} \wt{Q}(ds,d \lambda, du)\,. 
\end{align}
In order to simplify the notations, we start with the case $k=2$.
For any $\theta_1,\theta_2, \theta'_1, \theta'_2\in\R$, $t,t'>0$, we compute the limit as $N\to\infty$ of
\[ \E\left[\exp\left(i\theta_1\wt{\mathfrak{I}}^N_1(t)+i\theta_2\wt{\mathfrak{I}}^N_2(t)+i\theta'_1\wt{\mathfrak{I}}^N_1(t')+i\theta'_2\wt{\mathfrak{I}}^N_2(t')\right)\right]\,.\]
We apply Lemma \ref{lem:expmoment} to the particular case $E=\R_+\times \bD\times\R_+$, $Q=\breve{Q}$ and 
\begin{align*}
 f(s,\lambda,u)=iN^{-1/2}\Big\{[\theta_1\breve\lambda(t-s) +\theta_2\bar\lambda(t-s)] {\bf1}_{s\le t} + [\theta'_1 \breve\lambda(t'-s) +\theta'_2\bar\lambda(t'-s)] {\bf1}_{s\le t'} \Big\}{\bf1}_{u\le N\bar{\Upsilon}(s)}\,,
 \end{align*}
 from which we easily deduce that
\begin{align} \label{eqn-wtfrakI1-fdd-conv}
& \lim_{N\to \infty} \E\left[\exp\left(i\theta_1\wt{\mathfrak{I}}^N_1(t)+i\theta_2\wt{\mathfrak{I}}^N_2(t)+i\theta'_1\wt{\mathfrak{I}}^N_1(t')+i\theta'_2\wt{\mathfrak{I}}^N_1(t')\right)\right] \non \\
&=\exp\bigg(-\frac{(\theta_1)^2}{2}\int_0^{t}\E\big[ \breve\lambda^2(t-s)   \big]\bar{\Upsilon}(s)ds   
-\frac{(\theta_2)^2}{2}\int_0^{t}\bar\lambda^2(t-s)\bar{\Upsilon}(s)ds \non \\
&\qquad-\frac{(\theta_1')^2}{2}\int_0^{t'}\E\big[ \breve\lambda^2(t'-s)   \big]\bar{\Upsilon}(s)ds   
-\frac{(\theta'_2)^2}{2}\int_0^{t'} \bar\lambda^2(t'-s) \bar{\Upsilon}(s)ds \non \\
&\qquad 
-\theta_1\theta'_1\int_0^{t\wedge t'}\E\big[ \breve\lambda(t-s)  \breve\lambda(t'-s) \big]\bar{\Upsilon}(s)ds
-\theta_2\theta'_2\int_0^{t\wedge t'} \bar\lambda(t-s)  \bar\lambda(t'-s) \bar{\Upsilon}(s)ds
\bigg)\, .
\end{align}
We have proved that the two dimensional distributions of $(\wt{\mathfrak{I}}^N_1,\wt{\mathfrak{I}}^N_2)$ converge to  those of a centered Gaussian process, whose covariances are the ones of  
$(\hat{\mathfrak{I}}_1,\hat{\mathfrak{I}}_2)$ as given in Theorem \ref{thm-FCLT}. Note that it is immediate that the $k$ dimensional distributions also converge to those of a centered Gaussian process, and the covariances are fully determined by the above computation, hence the result.
\end{proof}

\begin{lemma} \label{lem-frakI1-diff-conv}
Under Assumptions \ref{AS-lambda} (ii) and \ref{AS-FCLT}, 
\begin{equation*}
 \hat{\mathfrak{I}}^N_1 - \wt{\mathfrak{I}}^N_1 \to  0 
\end{equation*}
in probability in $\bD$ as $N\to \infty$.
\end{lemma}

\begin{proof}
Recall the expression of $ \hat{\mathfrak{I}}^N_1(t) - \wt{\mathfrak{I}}^N_1(t)$ in \eqref{eqn-mfI-diff-1}. 
It is straightforward that 
$
\E\big[   \hat{\mathfrak{I}}^N_1(t) - \wt{\mathfrak{I}}^N_1(t)\big]= 0,
$
and
\begin{align*}
\E\big[ \big( \hat{\mathfrak{I}}^N_1(t) - \wt{\mathfrak{I}}^N_1(t)\big)^2\big] 
&=  \E\left[ \int_0^t \breve{\lambda}(t-s)^2 \big|\bar{\Upsilon}^N(s) - \bar{\Upsilon}(s) \big| ds \right] \\
& \le (\lambda^*)^2 \int_0^t \E\big[ \big|\bar{\Upsilon}^N(s) - \bar{\Upsilon}(s)\big|\big] ds \to 0 \qasq N \to\infty.
\end{align*}
Here the convergence follows from
\begin{align} \label{eqn-exp-barPhi-conv}
\E [|\bar\Upsilon^N(s)- \bar\Upsilon(s) |] \to 0 \qasq N \to \infty\,,
\end{align}
which holds 
by \eqref{eqn-barPhi-conv} and the dominated convergence theorem. 
It then suffices to show the tightness of $\{ \hat{\mathfrak{I}}^N_1 - \wt{\mathfrak{I}}^N_1: N \in \NN\}$. By the expression in \eqref{eqn-mfI-diff-1}, 
it suffices to show the tightness of the processes $\{ \Xi^N: N \in \NN\}$ defined by
\begin{align*}
\Xi^N(t)  &:= \frac{1}{\sqrt{N}}\int_0^t \int_\bD\int_{N(\bar\Upsilon^N(s) \wedge \bar\Upsilon(s))}^{N (\bar\Upsilon^N(s) \vee  \bar\Upsilon(s))}  |\breve\lambda(t-s)| \breve{Q}(ds,d\lambda,du) .
\end{align*}
By Lemma \ref{Lem-20}, it suffices to show that
\begin{align} \label{eqn-Xi-tight-mfI1}
\limsup_{N\to\infty} \frac{1}{\delta}  \P \left(\sup_{v\in [0,\delta]}  |\Xi^N(t+v)  - \Xi^N(t)  |> \ep \right) \to 0 \qasq \delta \to 0. 
\end{align}
We have
\begin{align*}
& |\Xi^N(t+v)  - \Xi^N(t)  | \\
& \le  \frac{1}{\sqrt{N}}\int_t^{t+v} \int_\bD\int_{N(\bar\Upsilon^N(s) \wedge \bar\Upsilon(s))}^{N (\bar\Upsilon^N(s) \vee  \bar\Upsilon(s))}   | \breve\lambda(t+v-s) | \breve{Q}(ds, d\lambda,du)\\
&\quad+ \frac{1}{\sqrt{N}}\int_0^{t}  \int_\bD\int_{N(\bar\Upsilon^N(s) \wedge \bar\Upsilon(s))}^{N (\bar\Upsilon^N(s) \vee  \bar\Upsilon(s))}  | \breve\lambda(t+v-s)  -\breve\lambda(t-s) | \breve{Q}(ds, d\lambda,du)\\
& \le   \frac{\lambda^\ast}{\sqrt{N}}\int_t^{t+v} \int_{N(\bar\Upsilon^N(s) \wedge \bar\Upsilon(s))}^{N (\bar\Upsilon^N(s) \vee  \bar\Upsilon(s))} Q(ds,du)\\
&\quad + \frac{1}{\sqrt{N}}\int_0^{t} \int_\bD\int_{N(\bar\Upsilon^N(s) \wedge \bar\Upsilon(s))}^{N (\bar\Upsilon^N(s) \vee  \bar\Upsilon(s))}  \bigg(  2 v^\alpha + \lambda^*  \sum_{j=1}^{k-1} \bone_{t-s < \xi_j \le t+v-s} \\
& \qquad \qquad +  \lambda^* \sum_{j=1}^{k-1} (F_j(t+v-s) - F_j(t-s)) \bigg) \breve{Q}(ds,d\lambda,du),
\end{align*}
where the second inequality follows from  Lemma \ref{lem-barlambda-inc-bound}.
It is clear that the above upper bound is increasing in $v$. 
Thus, we obtain that for any $\ep>0$, 
\begin{align}  \label{eqn-Xi-bound-p1-mfI}
& \P \left(\sup_{v\in [0,\delta]}  |\Xi^N(t+v)  - \Xi^N(t)  |> \ep \right) \non \\
&\le\P\left( \frac{\lambda^\ast}{\sqrt{N}}\int_t^{t+\delta} \int_{N(\bar\Upsilon^N(s) \wedge \bar\Upsilon(s))}^{N (\bar\Upsilon^N(s) \vee  \bar\Upsilon(s))} Q(ds,du)> \ep/4\right)\non\\
&\quad+ \P \left(     \frac{2 \delta^\alpha}{\sqrt{N}} \int_0^{t} \int_{N(\bar\Upsilon^N(s) \wedge \bar\Upsilon(s))}^{N (\bar\Upsilon^N(s) \vee  \bar\Upsilon(s))}    {Q}(ds,du) > \ep/4 \right)\non  \\
& \quad +   \P \left(     \frac{1}{\sqrt{N}}\int_0^{t}  \int_\bD \int_{N(\bar\Upsilon^N(s) \wedge \bar\Upsilon(s))}^{N (\bar\Upsilon^N(s) \vee  \bar\Upsilon(s))} \bigg( \lambda^* \sum_{j=1}^{k-1} \bone_{t-s < \xi_j \le t+\delta-s} \bigg) \breve{Q}(ds,d\lambda,du) > \ep/4 \right)\non  \\
& \quad +   \P \left(     \frac{1}{\sqrt{N}}\int_0^{t} \int_{N(\bar\Upsilon^N(s) \wedge \bar\Upsilon(s))}^{N (\bar\Upsilon^N(s) \vee  \bar\Upsilon(s))}  \bigg( \lambda^* \sum_{j=1}^{k-1} (F_j(t+\delta-s) - F_j(t-s)) \bigg) {Q}(ds,du) > \ep/4 \right). 
\end{align}
The first term is bounded by $\frac{16}{\ep^2}$ times
\begin{align}\label{simpleupperbound}
&\E\left[\left(\frac{\lambda^\ast}{\sqrt{N}}\int_t^{t+\delta} \int_{N(\bar\Upsilon^N(s) \wedge \bar\Upsilon(s))}^{N (\bar\Upsilon^N(s) \vee  \bar\Upsilon(s))} Q(ds,du)\right)^2\right]\non\\
&\le2\left\{\E\int_t^{t+\delta}|\bar{\Upsilon}^N(s)-\bar{\Upsilon}(s)|ds+
\left(\E\int_t^{t+\delta}|\hat{\Upsilon}^N(s)|ds\right)^2\right\}\non\\
&\le 2\left\{\delta\sup_{s\le T}\E|\bar{\Upsilon}^N(s)-\bar{\Upsilon}(s)|+\delta^2\sup_{s\le T}\E\left(|\hat{\Upsilon}^N(s)|^2\right)\right\},
\end{align}
where the first inequality follows from $Q(ds,du)=\overline{Q}(ds,du)+ds\times du$. We note that the first term on the right of \eqref{simpleupperbound} 
tends to $0$ as $N\to\infty$, while the $\limsup_N$ of the second term multiplied by $\delta^{-1}$ tends to $0$, as $\delta\to0$, by the moment property of $\hat\Upsilon^N$ in Lemma \ref{lem-hatPhi-2-boundedness}, which is exactly what we want. 

The second term is bounded by $ \frac{16}{\ep^2}  $ times
\begin{align*}
& \E \left[ \left( \frac{2}{\sqrt{N}} \delta^\alpha \int_0^{t} \int_{N(\bar\Upsilon^N(s) \wedge \bar\Upsilon(s))}^{N (\bar\Upsilon^N(s) \vee  \bar\Upsilon(s))}  {Q}(ds,du) \right)^2\right] \\
& \le  8\, \delta^{2\alpha}\left\{ \int_0^{t}\E |\bar\Upsilon^N(s)- \bar\Upsilon(s) |ds +    T\, \E \bigg[\sup_{s \in [0,T]} \big|\hat\Upsilon^N(s)\big|^2\bigg]\right\}. 
\end{align*}
By \eqref{eqn-FLLN-conv-SEIR}, the first term converges to zero as $N\to\infty$, while, thanks to the fact that 
$2\alpha>1$ and Lemma \ref{lem-hatPhi-2-boundedness}, $\limsup_N$ of the second term multiplied by $\delta^{-1}$ tends to $0$, as $\delta\to0$, which again is exactly what we want. 

The third term on the right hand side of \eqref{eqn-Xi-bound-p1-mfI} is bounded by $\frac{16}{\ep^2} $ times
\begin{align*}
 &   \E\left[  \left(     \frac{1}{\sqrt{N}}\int_0^{t}\int_\bD \int_{N(\bar\Upsilon^N(s) \wedge \bar\Upsilon(s))}^{N (\bar\Upsilon^N(s) \vee  \bar\Upsilon(s))}  \bigg( \lambda^* \sum_{j=1}^{k-1} \bone_{t-s < \xi_j \le t+\delta-s}  \bigg) \breve{Q}(ds,d\lambda,du) \right)^2 \right]\non  \\
 & \le  2  \E\left[  \left(     \frac{1}{\sqrt{N}}\int_0^{t}\int_\bD \int_{N(\bar\Upsilon^N(s) \wedge \bar\Upsilon(s))}^{N (\bar\Upsilon^N(s) \vee  \bar\Upsilon(s))}  \bigg( \lambda^* \sum_{j=1}^{k-1} \bone_{t-s < \xi_j \le t+\delta-s}  \bigg) \wt{Q}(ds,d\lambda,du) \right)^2 \right]\non  \\
 & \quad +  2 (\lambda^*)^2  
 \E\left[  \left(     \int_0^{t}\bigg( \sum_{j=1}^{k-1} (F_j(t+\delta-s) - F_j(t-s)) \bigg)  |\hat{\Upsilon}^N(s)|   ds \right)^2 \right].
\end{align*}
Here the first term is equal to twice
\begin{align*}
&    \int_0^{t} \E\left[ \bigg( \lambda^* \sum_{j=1}^{k-1} \bone_{t-s < \xi_j \le t+\delta-s}  \bigg) ^2   |\bar\Upsilon^N(s)- \bar\Upsilon(s) | \right] ds \non\\  
 & \le  (\lambda^*)^2  k^2  \int_0^{t}   \E [|\bar\Upsilon^N(s)- \bar\Upsilon(s) |]ds,
\end{align*}
which 
converges to zero as $N \to \infty$ by \eqref{eqn-FLLN-conv-SEIR}. The second term satisfies, thanks to
 \eqref{estim-int-F} and Lemma \ref{lem-hatPhi-2-boundedness}, 
\begin{align} \label{eqn-Fj-Upsilon-conv}
\frac{1}{\delta} \E&\left[  \left(    \int_0^{t}\bigg( \sum_{j=1}^{k-1} (F_j(t+\delta-s) - F_j(t-s)) \bigg)  |\hat{\Upsilon}^N(s)|   ds \right)^2 \right]\non\\
&\le \frac{k}{\delta}\sum_{j=1}^{k-1}\left(\int_0^t\big(F_j(t+\delta-s) - F_j(t-s)\big)ds\right)^2\  \E \bigg[\sup_{s \in [0,T]} \big|\hat\Upsilon^N(s)\big|^2\bigg] \non\\
&\le C\delta\,.
\end{align}

The fourth and last term on the right hand side of \eqref{eqn-Xi-bound-p1-mfI} is bounded by $\frac{16}{\ep^2} $
times
\begin{align*}
 &   \E\left[  \left(     \frac{1}{\sqrt{N}}\int_0^{t} \int_{N(\bar\Upsilon^N(s) \wedge \bar\Upsilon(s))}^{N (\bar\Upsilon^N(s) \vee  \bar\Upsilon(s))} \bigg( \lambda^* \sum_{j=1}^{k-1} (F_j(t+\delta-s) - F_j(t-s)) \bigg) {Q}(ds,du) \right)^2 \right]\non  \\
 & \le  2 (\lambda^*)^2  \E\left[  \left(     \frac{1}{\sqrt{N}}\int_0^{t} \int_{N(\bar\Upsilon^N(s) \wedge \bar\Upsilon(s))}^{N (\bar\Upsilon^N(s) \vee  \bar\Upsilon(s))}  \bigg( \sum_{j=1}^{k-1} (F_j(t+\delta-s) - F_j(t-s)) \bigg) \overline{Q}(ds,du) \right)^2 \right]\non  \\
& \quad + 2 (\lambda^*)^2  
 \E\left[  \left(    \int_0^{t}\bigg( \sum_{j=1}^{k-1} (F_j(t+\delta-s) - F_j(t-s)) \bigg)  |\hat{\Upsilon}^N(s)|   ds \right)^2 \right]\non  \\
 & \le   2 (\lambda^*)^2    \int_0^{t}\bigg( \sum_{j=1}^{k-1} (F_j(t+\delta-s) - F_j(t-s)) \bigg) ^2  \E [|\bar\Upsilon^N(s)- \bar\Upsilon(s) |]ds \non\\
 & \quad +  2 (\lambda^*)^2  
 \E\left[  \left(    \int_0^{t}\bigg( \sum_{j=1}^{k-1} (F_j(t+\delta-s) - F_j(t-s)) \bigg)  |\hat{\Upsilon}^N(s)|   ds \right)^2 \right]. 
\end{align*}
Here the first term converges to zero as $N \to \infty$ by \eqref{eqn-exp-barPhi-conv}. The second term also satisfies \eqref{eqn-Fj-Upsilon-conv}. 

It is then clear that \eqref{eqn-Xi-tight-mfI1} holds for $\Xi^N$. This completes the proof. 
\end{proof}

\begin{lemma} \label{lem-frakI2-diff-conv}
Under Assumption \ref{AS-lambda} (ii) and \ref{AS-FCLT}, 
\begin{equation*}
 \hat{\mathfrak{I}}^N_2 - \wt{\mathfrak{I}}^N_2 \to  0 
\end{equation*}
in probability in $\bD$ as $N\to \infty. $
\end{lemma}

\begin{proof}
Recall the expression of $\hat{\mathfrak{I}}^N_2 - \wt{\mathfrak{I}}^N_2$ in \eqref{eqn-mfI-diff-2}. 
It is clear that 
\begin{align*}
\E\big[   \hat{\mathfrak{I}}^N_2(t) - \wt{\mathfrak{I}}^N_2(t)\big]= 0\,,
\end{align*}
and
\begin{align*}
\E\big[ \big( \hat{\mathfrak{I}}^N_2(t) - \wt{\mathfrak{I}}^N_2(t)\big)^2\big] &= \int_0^t  \bar\lambda(t-s)^2  \E\big[\big|\bar{\Upsilon}^N(s) - \bar{\Upsilon}(s) \big|\big] ds \to 0 \qasq N \to\infty\,,
\end{align*}
where the convergence follows from the bounded convergence theorem and \eqref{eqn-FLLN-conv-SEIR}. 
It then suffices to show tightness of the sequence $ \{\hat{\mathfrak{I}}_2^N-  \wt{\mathfrak{I}}_{2}^N; N \in \NN\}$. 
By the expression in  \eqref{eqn-mfI-diff-2},  tightness of the processes $ \{\hat{\mathfrak{I}}_2^N-  \wt{\mathfrak{I}}_{2}^N: N \in \NN\}$  
can be deduced from the tightness of the following two processes 
\begin{align*}
\Xi^N_1(t)  &:= \frac{1}{\sqrt{N}}\int_0^t \int_{N(\bar\Upsilon^N(s) \wedge \bar\Upsilon(s))}^{N (\bar\Upsilon^N(s) \vee  \bar\Upsilon(s))}   \bar\lambda(t-s) Q(ds,du) ,\\
 \Xi^N_2(t) &:=  \int_0^t  \bar\lambda(t-s) \big|\hat{\Upsilon}^N(s) \big| ds. 
\end{align*}
By Lemma \ref{Lem-20}, it suffices to show that for $\ell=1,2$,
\begin{align} \label{eqn-Xi-tight}
\limsup_{N\to\infty} \frac{1}{\delta}  \P \left(\sup_{v\in [0,\delta]}  |\Xi^N_\ell(t+v)  - \Xi^N_\ell(t)  |> \ep \right) \to 0 \qasq \delta \to 0. 
\end{align}

For the process $\Xi^N_1(t) $,  
we have
\begin{align*}
 |\Xi^N_1(t+v)  - \Xi^N_1(t)  | 
& \le   \frac{\lambda^\ast}{\sqrt{N}} \int_t^{t+v}  \int_{N(\bar\Upsilon^N(s) \wedge \bar\Upsilon(s))}^{N (\bar\Upsilon^N(s) \vee  \bar\Upsilon(s))}  Q(ds,du)    \\
&\quad+\frac{1}{\sqrt{N}}\int_0^{t} \int_{N(\bar\Upsilon^N(s) \wedge \bar\Upsilon(s))}^{N (\bar\Upsilon^N(s) \vee  \bar\Upsilon(s))}   | \bar\lambda(t+v-s)  - \bar\lambda(t-s) | Q(ds,du). 
\end{align*}
We already know how to treat the first term, see \eqref{simpleupperbound}.
By Lemma \ref{lem-barlambda-inc-bound}, the second term on the right hand side is bounded by
$$
  \frac{1}{\sqrt{N}}\int_0^{t} \int_{N(\bar\Upsilon^N(s) \wedge \bar\Upsilon(s))}^{N (\bar\Upsilon^N(s) \vee  \bar\Upsilon(s))} \Big(  v^\alpha +  \lambda^* \sum_{j=1}^{k-1} (F_j(t+v-s) - F_j(t-s)) \Big) Q(ds,du),
$$
which is nondecreasing in $v$. 
Thus, we obtain that for any $\ep>0$, 
\begin{align}  \label{eqn-Xi-bound-p1}
& \P \left(\sup_{v\in [0,\delta]}  |\Xi^N_1(t+v)  - \Xi^N_1(t)  |> \ep \right) \non \\
& \le\P \left( \frac{\lambda^\ast}{\sqrt{N}}\int_t^{t+\delta} \int_{N(\bar\Upsilon^N(s) \wedge \bar\Upsilon(s))}^{N (\bar\Upsilon^N(s) \vee  \bar\Upsilon(s))}  Q(ds,du)  > \ep/3  \right)\non\\
&\quad+ \P \left(     \frac{\delta^\alpha}{\sqrt{N}}  \int_0^{t} \int_{N(\bar\Upsilon^N(s) \wedge \bar\Upsilon(s))}^{N (\bar\Upsilon^N(s) \vee  \bar\Upsilon(s))}  Q(ds,du) > \ep/3 \right)\non  \\
& \quad +   \P \left(     \frac{1}{\sqrt{N}}\int_0^{t} \int_{N(\bar\Upsilon^N(s) \wedge \bar\Upsilon(s))}^{N (\bar\Upsilon^N(s) \vee  \bar\Upsilon(s))} \bigg( \lambda^* \sum_{j=1}^{k-1} (F_j(t+\delta-s) - F_j(t-s)) \bigg) Q(ds,du) > \ep/3 \right) 
\end{align}
The first term is bounded as in \eqref{simpleupperbound}. The second term is bounded by $\frac{9}{\ep^2}$ times 
\begin{align*}
&    \E \left[ \left( \frac{\delta^\alpha}{\sqrt{N}} \int_0^{t} \int_{N(\bar\Upsilon^N(s) \wedge \bar\Upsilon(s))}^{N (\bar\Upsilon^N(s) \vee  \bar\Upsilon(s))}  Q(ds,du) \right)^2\right] \\
& \le2  \E \left[ \left( \frac{\delta^\alpha}{\sqrt{N}}  \int_0^{t} \int_{N(\bar\Upsilon^N(s) \wedge \bar\Upsilon(s))}^{N (\bar\Upsilon^N(s) \vee  \bar\Upsilon(s))} \overline{Q}(ds,du) \right)^2\right]   + 2   \E \left[ \left(\delta^\alpha \int_0^{t}  |\hat{\Upsilon}^N(s)| ds   \right)^2\right]\\
& \le   2  \delta^{2\alpha} \int_0^{t}\E |\bar\Upsilon^N(s)- \bar\Upsilon(s) |ds +   2  \delta^{2\alpha} T 
 \E \bigg[\sup_{s \in [0,T]} \big|\hat\Upsilon^N(s)\big|^2\bigg].  
\end{align*}
This upper bound satisfies the proper bound \eqref{eqn-Xi-tight}, by the same argument as already used in the proof of Lemma \ref{lem-frakI1-diff-conv}.

The third term on the right hand side of \eqref{eqn-Xi-bound-p1} is bounded by $\frac{9}{\ep^2} $ times
\begin{align*}
 &   \E\left[  \left(     \frac{1}{\sqrt{N}}\int_0^{t} \int_{N(\bar\Upsilon^N(s) \wedge \bar\Upsilon(s))}^{N (\bar\Upsilon^N(s) \vee  \bar\Upsilon(s))} \bigg( \lambda^* \sum_{j=1}^{k-1} (F_j(t+\delta-s) - F_j(t-s)) \bigg) Q(ds,du) \right)^2 \right]\non  \\
 & \le  2 (\lambda^*)^2  \E\left[  \left(     \frac{1}{\sqrt{N}}\int_0^{t} \int_{N(\bar\Upsilon^N(s) \wedge \bar\Upsilon(s))}^{N (\bar\Upsilon^N(s) \vee  \bar\Upsilon(s))} \bigg( \sum_{j=1}^{k-1} (F_j(t+\delta-s) - F_j(t-s)) \bigg) \overline{Q}(ds,du) \right)^2 \right]\non  \\
& \quad + 2 (\lambda^*)^2  
 \E\left[  \left(  \int_0^{t}\bigg( \sum_{j=1}^{k-1} (F_j(t+\delta-s) - F_j(t-s)) \bigg)  |\hat{\Upsilon}^N(s)|   ds \right)^2 \right]\non  \\
 & \le   2(\lambda^*)^2    \int_0^{t}\bigg( \sum_{j=1}^{k-1} (F_j(t+\delta-s) - F_j(t-s)) \bigg) ^2  \E [|\bar\Upsilon^N(s)- \bar\Upsilon(s) |]ds \non\\
 & \quad +  2 (\lambda^*)^2  
 \E\left[  \left(    \int_0^{t}\bigg( \sum_{j=1}^{k-1} (F_j(t+\delta-s) - F_j(t-s)) \bigg)  |\hat{\Upsilon}^N(s)|   ds \right)^2 \right]. 
\end{align*}
Here the first term converges to zero as $N \to \infty$ by \eqref{eqn-FLLN-conv-SEIR}. The second term satisfies \eqref{eqn-Fj-Upsilon-conv}.

\smallskip

Next for the process $\Xi^N_2(t)$, we have 
\begin{align*} 
 \big|  \Xi^N_2(t+v)  -  \Xi^N_2(t)  \big|  
& = \left| \int_0^{t+v}  \bar\lambda(t+v-s) \big|\hat{\Upsilon}^N(s) \big| ds - \int_0^t  \bar\lambda(t-s)  \big|\hat{\Upsilon}^N(s) \big| ds\right|  \non\\
& \le   \int_t^{t+v}   \bar\lambda(t+v-s)   \big|\hat{\Upsilon}^N(s) \big| ds+ \int_0^{t}  \big|  \bar\lambda(t+v-s)  -   \bar\lambda(t-s) \big|   \big|\hat{\Upsilon}^N(s) \big| ds .
\end{align*}
The first term is bounded from above by
\[ \lambda^\ast   \int_t^{t+v}    \big|\hat{\Upsilon}^N(s) \big| ds,\]
while the second term on the right hand side can be bounded by 
$$
  \int_0^{t}    \Big(  v^\alpha +  \lambda^* \sum_{j=1}^{k-1} (F_j(t+v-s) - F_j(t-s)) \Big)ds  \bigg(\sup_{s \in [0,T]} \big|\hat\Upsilon^N(s)\big|\bigg). 
$$
Those two upper bounds are  nondecreasing in $v$, and by already used arguments, we easily establish that 
$\Xi^N_2$ satisfies \eqref{eqn-Xi-tight}.  
This completes the proof of the lemma. 
\end{proof}

\subsection{Completing the proof of the convergence $(\hat{S}^N, \hat{\mathfrak{I}}^N) \to (\hat{S}, \hat{\mathfrak{I}})$ in $\bD^2$} \label{sec-thm21-proof}

We are now ready to complete the proof of the convergence of $(\hat{S}^N, \hat{\mathfrak{I}}^N) \to (\hat{S}, \hat{\mathfrak{I}})$, stated in Theorem \ref{thm-FCLT}.
 
\begin{proof}[Proof of the convergence $(\hat{S}^N, \hat{\mathfrak{I}}^N) \to (\hat{S}, \hat{\mathfrak{I}})$ ]
We first prove the joint convergence 
\begin{equation} \label{eqn-hat-joint-conv}
 \big(\hat{E}^N(0), \hat{I}^N(0), \hat{M}^N_A, \hat{\mathfrak{I}}^N_{0,1},  \hat{\mathfrak{I}}^N_{0,2}, \hat{\mathfrak{I}}^N_1, \hat{\mathfrak{I}}^N_2\big)\RA  
   \big(\hat{E}(0), \hat{I}(0), \hat{M}_A, \hat{\mathfrak{I}}_{0,1},  \hat{\mathfrak{I}}_{0,2},  \hat{\mathfrak{I}}_1, \hat{\mathfrak{I}}_2\big)
\end{equation}
in $\RR^2 \times \bD^5$ as $N\to\infty$. 
By the independence of the variables associated with the initially and newly infected individuals, 
it suffices to show the joint convergences 
\begin{equation*}
 \big(\hat{E}^N(0), \hat{I}^N(0), \hat{\mathfrak{I}}^N_{0,1},  \hat{\mathfrak{I}}^N_{0,2}\big)\RA  
   \big(\hat{E}(0), \hat{I}(0), \hat{\mathfrak{I}}_{0,1},  \hat{\mathfrak{I}}_{0,2}\big) \qinq \RR^2 \times \bD^2 \qasq N \to \infty\,,
\end{equation*}
and 
\begin{equation} \label{eqn-hatMA-frakI-joint-conv}
 \big(\hat{M}^N_A, \hat{\mathfrak{I}}^N_1, \hat{\mathfrak{I}}^N_2\big)\RA  
   \big(\hat{M}_A,  \hat{\mathfrak{I}}_1, \hat{\mathfrak{I}}_2\big) \qinq \bD^3 \qasq N \to \infty\,.
\end{equation}
 It is also clear that the two sets of limits $   \big(\hat{E}(0), \hat{I}(0), \hat{\mathfrak{I}}_{0,1},  \hat{\mathfrak{I}}_{0,2}\big)$ and 
$ \big(\hat{M}_A,  \hat{\mathfrak{I}}_1, \hat{\mathfrak{I}}_2\big)$ are independent.  
The convergence of $\big(\hat{E}^N(0), \hat{I}^N(0), \hat{\mathfrak{I}}^N_{0,1},  \hat{\mathfrak{I}}^N_{0,2}\big)$ is straightforward. We focus on the convergence of $ \big(\hat{M}^N_A, \hat{\mathfrak{I}}^N_1, \hat{\mathfrak{I}}^N_2\big)$. 
Recall the compensated PRM $\wt{Q}(ds,d\lambda,du)$ on $\RR_+ \times \bD\times \RR_+$. 
Define an auxiliary process 
\begin{align} \label{eqn-breveAn-def}
\wt{M}^N_A (t) &:= \frac{1}{\sqrt{N}} \int_0^t \int_\bD\int_0^\infty \bone_{u \le N \bar{\Upsilon}(s)} \wt{Q}(ds,d\lambda,du)\,.
\end{align}
Recall the process $ \wt{\mathfrak{I}}^N_1(t)$ defined in \eqref{eqn-wtmfI-1-def}, where $\breve{Q}$ can be replaced by $\wt{Q}$.  Also, recall the process $ \wt{\mathfrak{I}}^N_2(t) $ in \eqref{eqn-wtJ2-wtQ} using $\wt{Q}$. 
A minor extension of the computation done in Lemma \ref{lem-hatmfI1-2} yields the joint convergence of the finite dimensional distributions of the processes $ \big(\wt{M}^N_A, \wt{\mathfrak{I}}^N_1, \wt{\mathfrak{I}}^N_2\big)$. Note that $\wt{M}^N_A$ being a martingale, its tightness is easily established.
Hence we deduce the joint convergence
\begin{equation} \label{eqn-wt-MA-frakI-12-conv}
 \big(\wt{M}^N_A, \wt{\mathfrak{I}}^N_1, \wt{\mathfrak{I}}^N_2\big)\RA  
   \big(\hat{M}_A,  \hat{\mathfrak{I}}_1, \hat{\mathfrak{I}}_2\big) \qinq  \bD^3, \qasq N \to \infty\,.
\end{equation}
Concerning the pair $(\hat{M}_A,\hat{\mathfrak{I}}_2)$, we also observe that $\hat{M}_A$ is a non--standard Brownian motion,
and $\hat{\mathfrak{I}}_2(t)=\int_0^t\bar\lambda(t-s)\hat{M}_A(ds)$. Note that $\hat{W}_S$ in Definition \ref{def-W} has the same distribution as $ - \hat{M}_A$.

A simplified version of the proofs in Lemmas \ref{lem-frakI1-diff-conv} and \ref{lem-frakI2-diff-conv} yields that
$\wt{M}^N_A- \hat{M}^N_A\to 0$ in probability in $\bD$, as $N\to \infty$. Hence
\[ (\wt{M}^N_A- \hat{M}^N_A, \wt{\mathfrak{I}}^N_1-\hat{\mathfrak{I}}_1,\wt{\mathfrak{I}}^N_2-\hat{\mathfrak{I}}_2)\to 0\]
in probability in $\bD^3$, as $N\to\infty$.
Combined with \eqref{eqn-wt-MA-frakI-12-conv}, this establishes \eqref{eqn-hatMA-frakI-joint-conv}.

Observe that the equations 
 \eqref{eqn-hatSn-rep} and \eqref{eqn-hatfrakI-rep} coupled with \eqref{eqn-hatPhi-N-rep} define uniquely the processes $ \big(\hat{S}^N,\hat{\mathfrak{I}}^N\big)$ as the solution of a two-dimensional integral equation  driven by   $ \big(\hat{E}^N(0), \hat{I}^N(0), \hat{M}^N_A,  \hat{\mathfrak{I}}^N_{0,1}, \hat{\mathfrak{I}}^N_{0,2}, \\
  \hat{\mathfrak{I}}^N_1, \hat{\mathfrak{I}}^N_2\big)$ and the fixed functions $\bar{\lambda}^0(t)$, $\bar{\lambda}(t)$, $\bar{S}(t)$ and $\bar{\mathfrak{I}}(t)$. The mapping which to those data associates the solution is continuous in the Skorohod $J_1$ topology, see Lemma 8.1 in \cite{PP-2020}. 
Thus, by the joint convergence in \eqref{eqn-hat-joint-conv}, we apply the continuous mapping theorem to conclude \eqref{eqn-hatSfrakI-conv}.

It is clear from \eqref{eqn-hatfrakI-rep} that the limit term 
$\hat{W}_{\mathfrak{I}}(t)$ in the expression of $\hat{\mathfrak{I}}(t)$ in Theorem \ref{thm-FCLT} is equal to $\hat{W}_{\mathfrak{I}}(t) = \hat{\mathfrak{I}}_{0,1}(t)+ \hat{\mathfrak{I}}_{0,2}(t)+ \hat{\mathfrak{I}}_1(t) +\hat{\mathfrak{I}}_2(t)$. 
Given the above results, it is straightforward to derive the covariance functions of $\hat{W}_{\mathfrak{I}}(t)$. Specifically, 
$\Cov(\hat{W}_{\mathfrak{I}}(t), \hat{W}_{\mathfrak{I}}(t')) = \Cov( \hat{\mathfrak{I}}_{0,1}(t),  \hat{\mathfrak{I}}_{0,1}(t')) + \Cov( \hat{\mathfrak{I}}_{0,2}(t),  \hat{\mathfrak{I}}_{0,2}(t'))
+ \Cov( \hat{\mathfrak{I}}_{1}(t),  \hat{\mathfrak{I}}_{1}(t')) + \Cov( \hat{\mathfrak{I}}_{2}(t),  \hat{\mathfrak{I}}_{2}(t'))$. 
In addition, $\Cov(\hat{M}_A(t),\hat{W}_{\mathfrak{I}}(t')) = \Cov(\hat{M}_A(t),\hat{\mathfrak{I}}_2(t')) = \int_0^{t\wedge t'} \bar{\lambda}(t'-s)  \bar{S}(s)\bar{\mathfrak{I}}(s)  ds$.

Finally we show that the limit processes $\hat{\mathfrak{I}}_1$ and $\hat{\mathfrak{I}}_2$ have a continuous version in $\bC$, given their consistent finite dimensional distributions. 
Since its increments of the processes are Gaussian and centered,
it suffices to show the continuity of the covariance functions. It is easy to calculate that 
\begin{align*}
&\E\big[\big| \hat{\mathfrak{I}}_1(t+\delta) - \hat{\mathfrak{I}}_1(t) \big|^2\big]   \\
& =
 \int_t^{t+\delta} \Var(\lambda(t+\delta-s)) \bar{\Upsilon}(s)ds + \int_0^t \Var\big( \big(\lambda(t+\delta-s) - \lambda(t-s) \big)^2\big) \bar{\Upsilon}(s)ds\,, 
\end{align*}
and
\begin{align*}
&\E\big[\big| \hat{\mathfrak{I}}_2(t+\delta) - \hat{\mathfrak{I}}_2(t) \big|^2\big]   \\
& = \int_t^{t+\delta} \bar{\lambda}(t+\delta-s)^2 \bar{\Upsilon}(s)ds + \int_0^t\big(  \bar{\lambda}(t+\delta-s) - \bar{\lambda}(t-s) \big)^2 \bar{\Upsilon}(s)ds\,. 
\end{align*}
Then we can verify the continuity property under the conditions in Assumption \ref{AS-lambda} (iii). 
\end{proof}

\section{Completing the Proof of Theorem \ref{thm-FCLT}} \label{sec-EIR-proof}

 In this section we complete the proof of Theorem \ref{thm-FCLT} by establishing the convergence of  $(\hat{E}^N, \hat{I}^N, \hat{R}^N)$ and their joint convergence with of  $(\hat{S}^N, \hat{\mathfrak{I}}^N)$. The former follows   essentially the same arguments as that of Theorem 3.2 in \cite{PP-2020}, for which we only highlight the differences. However, the joint convergence is nontrivial since the exposed and infectious periods are induced by the random infectivity functions.

We have the following representations for the processes $(\hat{E}^N, \hat{I}^N, \hat{R}^N)$: 
\begin{align}
\hat{E}^N(t) &= \hat{E}^N(0) G_0^c(t) + \hat{E}^N_0(t) + \hat{E}^N_1(t)  + \int_0^t G^c(t-s)  \hat{\Upsilon}^N(s) ds, \label{eqn-hatE-rep-SEIR}\\
\hat{I}^N(t) &= \hat{I}^N(0) F_{0,I}^c(t) + \hat{E}^N(0) \Psi_0(t) +  \hat{I}^N_{0,1}(t)+  \hat{I}^N_{0,2}(t)  + \hat{I}^N_1(t)  + \int_0^t \Psi(t-s)  \hat{\Upsilon}^N(s) ds, \label{eqn-hatI-rep-SEIR}\\
\hat{R}^N(t) &=  \hat{I}^N(0) F_{0,I}(t) + \hat{E}^N(0) \Phi_0(t) +  \hat{R}^N_{0,1}(t)+  \hat{R}^N_{0,2}(t)  + \hat{R}^N_1(t)  + \int_0^t \Phi(t-s)  \hat{\Upsilon}^N(s) ds, \label{eqn-hatR-rep-SEIR}
\end{align}
where 
\begin{align*}
&\hat{E}^N_0(t) := \frac{1}{\sqrt{N}}\sum_{j=1}^{E^N(0)} ( {\bf1}_{\zeta^0_j>t} - G_0^c(t)) , \\
& \hat{I}^N_{0,1}(t) := \frac{1}{\sqrt{N}} \sum_{k=1}^{I^N(0)} ({\bf1}_{\eta^{0,I}_k>t} - F_{0,I}^c(t)), \quad  \hat{I}^N_{0,2}(t) := \frac{1}{\sqrt{N}} \sum_{j=1}^{E^N(0)}( {\bf1}_{\zeta^0_j\le t} {\bf1}_{\zeta^0_j + \eta^0_j>t} -\Psi_0(t) ) ,
\\
&\hat{R}^N_{0,1}(t) := \frac{1}{\sqrt{N}} \sum_{k=1}^{I^N(0)} ( {\bf1}_{\eta^{0,I}_k\le t} - F_{0,I}(t)), \quad \hat{R}^N_{0,2}(t) :=\frac{1}{\sqrt{N}} \sum_{j=1}^{E^N(0)} ( {\bf1}_{\zeta^0_j + \eta^0_j\le t} - \Phi_0(t)), 
\end{align*}
and
\begin{align*}
\hat{E}_1^N(t) &:=\frac{1}{\sqrt{N}}   \sum_{i=1}^{A^N(t)} \bone_{\tau^n_i + \zeta_i >t} - \sqrt{N}  \int_0^t G^c(t-s) \bar{S}^N(s) \bar{\mathfrak{I}}^N(s) ds\,, \\
 \hat{I}^N_1(t) &:=  \frac{1}{\sqrt{N}} \sum_{i=1}^{A^N(t)}{\bf1}_{\tau^N_i + \zeta_i\le t} {\bf1}_{\tau^N_i + \zeta_i+\eta_i>t}  - \sqrt{N}  \int_0^t \Psi(t-s) \bar{S}^N(s) \bar{\mathfrak{I}}^N(s) ds ,\\
  \quad \hat{R}^N_1(t) &:=  \frac{1}{\sqrt{N}} \sum_{i=1}^{A^N(t)}{\bf1}_{\tau^N_i+\zeta_i+\eta_i\le t}- \sqrt{N}  \int_0^t \Phi(t-s) \bar{S}^N(s) \bar{\mathfrak{I}}^N(s) ds. 
 \end{align*}

\begin{lemma} \label{lem-3.11}
Under  Assumptions \ref{AS-lambda}  and  \ref{AS-FCLT},  
\begin{align*}
\big( \hat{E}^N_0, \hat{I}^N_{0,1},  \hat{I}^N_{0,2}, \hat{R}^N_{0,1},  \hat{R}^N_{0,2}\big) \RA 
\big( \hat{E}_0, \hat{I}_{0,1},  \hat{I}_{0,2}, \hat{R}_{0,1},  \hat{R}_{0,2}\big) \qinq \bD^5 \qasq N \to \infty,
\end{align*}
jointly with the convergence $\big( \hat{\mathfrak{I}}^N_{0,1},  \hat{\mathfrak{I}}^N_{0,2} \big) \RA  \big( \hat{\mathfrak{I}}_{0,1},  \hat{\mathfrak{I}}^N_{0,2}\big)$ in \eqref{eqn-hatfrakI-0-conv}.
$( \hat{\mathfrak{I}}_{0,1},\hat{I}_{0,1},\hat{R}_{0,1})$ and 
$(\hat{\mathfrak{I}}_{0,2},\hat{E}_0, \hat{I}_{0,2}, \hat{R}_{0,2})$  
are independent three-dimensional centered Gaussian processes, 
independent of $(\hat{E}(0),\hat{I}(0))$, driven by the common  random source $\lambda^{0,I}(\cdot)$ and $\lambda^{0}(\cdot)$, respectively. 
$( \hat{\mathfrak{I}}_{0,1},\hat{I}_{0,1},\hat{R}_{0,1})$ has covariance functions: 
for  $t, t'\ge 0$, 
\begin{align*}
\Cov(\hat{\mathfrak{I}}_{0,1}(t), \hat{I}_{0,1}(t')) 
&= \bar{I}(0) \big(  \E \big[ \lambda^{0,I}(t) \bone_{\eta^{0,I}>t'} \big]  -  \bar\lambda^{0,I}(t) F^c_{0,I}(t') \big) \,,  \\
\Cov(\hat{\mathfrak{I}}_{0,1}(t), \hat{R}_{0,1}(t')) 
&= \bar{I}(0) \big(  \E \big[ \lambda^{0,I}(t) \bone_{\eta^{0,I}\le t'} \big]  -  \bar\lambda^{0,I}(t) F_{0,I}(t') \big) \,,\\
\Cov(\hat{I}_{0,1}(t), \hat{I}_{0,1}(t')) &= \bar{I}(0) (F_{0,I}^c(t\vee t') - F_{0,I}^c(t) F_{0,I}^c(t')), \\
\Cov(\hat{R}_{0,1}(t), \hat{R}_{0,1}(t')) &= \bar{I}(0) (F_{0,I}(t\wedge t') - F_{0,I}(t) F_{0,I}(t')), \\
\Cov(\hat{I}_{0,1}(t), \hat{R}_{0,1}(t')) &=\bar{I}(0) ( (F_{0,I}(t') - F_{0,I}(t)) \bone(t'\ge t) - F^c_{0,I}(t) F_{0,I}(t') ). 
\end{align*}
$(\hat{\mathfrak{I}}_{0,2},\hat{E}_0, \hat{I}_{0,2}, \hat{R}_{0,2})$   has covariance functions: 
for  $t, t'\ge 0$, 
\begin{align*}
\Cov(\hat{\mathfrak{I}}_{0,2}(t), \hat{E}_0(t')) 
&=  \bar{E}(0) \big(\E \big[ \lambda^{0}(t) \bone_{\zeta^{0}> t'} \big] -\bar\lambda^{0}(t) G^c_{0}(t')   \big) \,, \\
\Cov(\hat{\mathfrak{I}}_{0,2}(t),  \hat{I}_{0,2}(t')) 
&=  \bar{E}(0) \big( \E \big[   \lambda^{0}(t) \bone_{\zeta^{0}\le t'<\zeta^{0} + \eta^{0}} \big]  - \bar\lambda^{0}(t)  \Psi_{0}(t') \big)  \,, \\
\Cov(\hat{\mathfrak{I}}_{0,2}(t), \hat{R}_{0,2}(t')) 
&=   \bar{E}(0) \big( \E \big[  \lambda^{0}(t) \bone_{\zeta^{0} + \eta^{0}\le t'}\big]  -\bar\lambda^{0}(t)\Phi_{0}(t') \big)   \,,
\end{align*}
and
\begin{align*}
\Cov(\hat{E}_0(t), \hat{E}_0(t')) &= \bar{E}(0) (G_0^c(t\vee t') - G^c_0(t) G^c_0(t')),\\
\Cov(\hat{I}_{0,2}(t), \hat{I}_{0,2}(t')) &=  \bar{E}(0) \bigg(  \int_0^{t\wedge t'} F_0^c(t\vee t'-s|s)  G_0(ds) - \Psi_0(t) \Psi_0(t') \bigg), \\
\Cov(\hat{R}_{0,2}(t), \hat{R}_{0,2}(t')) &= \bar{E}(0) \left( \Phi_0(t\wedge t') -   \Phi_0(t) \Phi_0(t')  \right), \\
\Cov( \hat{E}_{0}(t),  \hat{I}_{0,2}(t')) &= \bar{E}(0)\bone(t'\ge t) \bigg( \int_t^{t'}  F^c_0(t'-s|s)G_0(ds) - G_0^c(t) \Psi_0(t') \bigg),  \non\\
\Cov( \hat{E}_{0}(t),  \hat{R}_{0,2}(t')) &= \bar{E}(0) \bone(t'\ge t) \bigg(\int_t^{t'}F_{0}(t'-s|s) G_0(ds) - G_0^c(t)  \Phi_0(t')\bigg),    \non\\
\Cov( \hat{I}_{0,2}(t),  \hat{R}_{0,2}(t')) &= \bar{E}(0) \bone(t'\ge t) \bigg( \int_0^{t} (F_0(t'-s|s) - F_0(t-s|s)) G_0(ds) - \Psi_0(t) \Phi_0(t') \bigg).
\end{align*}

\end{lemma}

\begin{proof}
By the independence of the sequences $\{\lambda^0_j\}_{j\ge 1}$ and $\{\lambda^{0,I}_k\}_{k\ge 1}$,
it suffices to prove the joint convergence of $\big( \hat{\mathfrak{I}}^N_{0,1}, \hat{I}^N_{0,1}, \hat{R}^N_{0,1} \big)$ and $\big( \hat{\mathfrak{I}}^N_{0,2}, \hat{E}^N_0, \hat{I}^N_{0,2}, \hat{R}^N_{0,2} \big)$ separately. 

Recall the processes $\wt{\mathfrak{I}}^N_{0,1}$ and $\wt{\mathfrak{I}}^N_{0,2}$ defined in \eqref{eqn-hatfrakI01-def} and \eqref{eqn-hatfrakI02-def}, respectively. Similarly we define $\big( \wt{E}^N_0, \wt{I}^N_{0,1},  \wt{I}^N_{0,2}, \wt{R}^N_{0,1},  \wt{R}^N_{0,2}\big)$ by replacing $E^N(0)$ and $I^N(0)$ by 
$N\bar{E}(0)$ and $N\bar{I}(0)$, respectively. 
By the FCLT for random elements in $\bD$ (see Theorem 2 in \cite{hahn1978central}, applied to the processes $\wt{\mathfrak{I}}^N_{0,1}$ and $\wt{\mathfrak{I}}^N_{0,2}$ under Assumption \ref{AS-lambda}(i) (a) and (b)) and the FCLT for empirical processes (see Theorem 14.3 in \cite{billingsley1999convergence}, applied to the processes $\big( \wt{E}^N_0, \wt{I}^N_{0,1},  \wt{I}^N_{0,2}, \wt{R}^N_{0,1},  \wt{R}^N_{0,2}\big)$), and by the definitions in \eqref{eqn-lambda-eta-0} and \eqref{eqn-lambda-eta-0I}, we obtain the joint convergences
$$\big( \wt{\mathfrak{I}}^N_{0,1}, \wt{I}^N_{0,1}, \wt{R}^N_{0,1} \big) \RA 
\big( \hat{\mathfrak{I}}_{0,1}, \hat{I}_{0,1}, \hat{R}_{0,1} \big) \qinq \bD^3  \qasq N \to \infty,$$
and
$$\big( \wt{\mathfrak{I}}^N_{0,2}, \wt{E}^N_0, \wt{I}^N_{0,2}, \wt{R}^N_{0,2} \big) \RA 
\big( \hat{\mathfrak{I}}_{0,2},  \hat{E}_0, \hat{I}_{0,2}, \hat{R}_{0,2} \big) \qinq \bD^4  \qasq N \to \infty.$$
It then suffices to show that 
$$
\big( \wt{\mathfrak{I}}^N_{0,1} -\hat{\mathfrak{I}}^N_{0,1}, \wt{\mathfrak{I}}^N_{0,2} -\hat{\mathfrak{I}}^N_{0,2},  \wt{I}^N_{0,1}- \hat{I}^N_{0,1}, \wt{R}^N_{0,1}- \hat{R}^N_{0,1},  \wt{E}^N_{0}- \hat{E}^N_{0},
 \wt{I}^N_{0,2}- \hat{I}^N_{0,2}, \wt{R}^N_{0,2}- \hat{R}^N_{0,2},  \big)  \to 0
$$ in probability in $\bD^7$,
as $N \to \infty$. 
The convergence for $\big( \wt{\mathfrak{I}}^N_{0,1} -\hat{\mathfrak{I}}^N_{0,1}, \wt{\mathfrak{I}}^N_{0,2} -\hat{\mathfrak{I}}^N_{0,2}\big) \to 0$ in probability in $\bD^2$   is shown in \eqref{eqn-hatfrakI-0-diff-conv}. For the other process, the convergence follows similar arguments as already used above.  See also the proofs of Lemmas 6.1 and 8.1 in \cite{PP-2020}. 
This completes the proof.
\end{proof}

\begin{remark}
We remark that 
the limit $(\hat{I}_{0,1},\hat{R}_{0,1})$
  can be written as 
\begin{align} \label{eqn-hatI-01-W01}
\hat{I}_{0,1}(t) =W_{0,1}([t,\infty)),  \quad \hat{R}_{0,1}(t) = W_{0,1}([0,t) )
\end{align} 
where $W_{0,1}$ is a Gaussian white noise on $\RR_+$ satisfying 
\begin{align*}
\E[W_{0,1}([s,t))^2] = \bar{I}(0)   \big( F_{0,I}(t) -F_{0,I}(s) \big) (1-\big( F_{0,I}(t) -F_{0,I}(s) \big)\, 
\end{align*}
for $0\le s \le t$. 
The limits  $(\hat{E}_0, \hat{I}_{0,2}, \hat{R}_{0,2})$ 
can be written as 
\begin{align}\label{eqn-hatI-02-W02}
\hat{E}_0(t) &= W_{0,2}([t,\infty]\times [0,\infty)  ), \,\,\,
\hat{I}_{0,2} (t) = W_{0,2}([0,t]\times [t,\infty) ), \,\,\,
\hat{R}_{0,2}(t) = W_{0,2}([0,t]\times [0,t))
\end{align}
where   $W_{0,2}$ is a Gaussian white noise on $\RR_+^2$, independent of $W_{0,1}$, such that 
\begin{align*}
\E[W_{0,2}([s,t)\times [s',t'))^2] &= \bar{E}(0) \int_{s}^t  \big(F_0(t'-y|y) - F_0(s'-y|y) \big) G_0(dy) \\
& \qquad \times \left( 1-  \int_{s}^t  \big(F_0(t'-y|y) - F_0(s'-y|y) \big) G_0(dy) \right) 
\end{align*}
for $0 \le s \le t$ and $0 \le s' \le t'$. 
\end{remark}

\begin{lemma} \label{lem-3.12}
Under Assumptions \ref{AS-lambda} and  \ref{AS-FCLT},   
\begin{align*}
\big(\hat{E}^N_1,\hat{I}^N_1, \hat{R}^N_1\big) \RA \big( \hat{E}_1, \hat{I}_1, \hat{R}_1\big) \qinq \bD^3 \qasq N \to \infty,
\end{align*}
jointly with the convergence of $ \big(\hat{\mathfrak{I}}^N_1, \hat{\mathfrak{I}}^N_2\big)  \RA \big(\hat{\mathfrak{I}}_1, \hat{\mathfrak{I}}_2\big)$,
where $\big(\hat{\mathfrak{I}}_1, \hat{\mathfrak{I}}_2\big)$ is given in Lemma \ref{lem-mfI12-conv}. 
$(\hat{\mathfrak{I}}_1,\hat{\mathfrak{I}}_2,\hat{E}_1, \hat{I}_1,\hat{R}_1)$  is a four-dimensional centered Gaussian process, independent of $(\hat{E}(0),\hat{I}(0))$, $( \hat{\mathfrak{I}}_{0,1},\hat{I}_{0,1},\hat{R}_{0,1})$ and 
$(\hat{\mathfrak{I}}_{0,2},\hat{E}_0, \hat{I}_{0,2}, \hat{R}_{0,2})$, 
 with the covariance functions: for  $t, t'\ge 0$, 
  \begin{align*}
\Cov(\hat{\mathfrak{I}}_{1}(t),  \hat{E}_1(t')) &= \int_0^{t \wedge t'}  \Big( \E \big[\lambda(t-s) \bone_{\zeta>t'-s}\big] - \bar\lambda(t-s) G^c(t'-s) \Big) \bar{S}(s) \bar{\mathfrak{I}}(s) ds\,, \\
\Cov(\hat{\mathfrak{I}}_{1}(t),   \hat{I}_1(t')) &=  \int_0^{t \wedge t'}  \Big( \E \big[\lambda(t-s) \bone_{\zeta\le t'-s<\zeta+\eta}\big] - \bar\lambda(t-s) \Psi(t'-s) \Big) \bar{S}(s) \bar{\mathfrak{I}}(s) ds\,, \\
\Cov(\hat{\mathfrak{I}}_{1}(t), \hat{R}_1(t')) &= \int_0^{t \wedge t'}  \Big( \E \big[\lambda(t-s) \bone_{\zeta+\eta\le t' -s}\big] - \bar\lambda(t-s) \Phi(t'-s) \Big) \bar{S}(s) \bar{\mathfrak{I}}(s) ds\,,
\end{align*}
  \begin{align*}
\Cov(\hat{\mathfrak{I}}_{2}(t),  \hat{E}_1(t')) &=  \int_0^{t \wedge t'}  \bar\lambda(t-s) G^c(t'-s) \bar{S}(s) \bar{\mathfrak{I}}(s) ds\,, \\
\Cov(\hat{\mathfrak{I}}_{2}(t),   \hat{I}_1(t')) &= \int_0^{t \wedge t'}  \bar\lambda(t-s) \Psi(t'-s)  \bar{S}(s) \bar{\mathfrak{I}}(s) ds\,, \\
\Cov(\hat{\mathfrak{I}}_{2}(t), \hat{R}_1(t')) &=  \int_0^{t \wedge t'}  \bar\lambda(t-s) \Phi(t'-s)  \bar{S}(s) \bar{\mathfrak{I}}(s) ds\,,
\end{align*}
and 
\begin{align}
\Cov(\hat{E}_1(t), \hat{E}_1(t')) & =  \int_0^{t\wedge t'} G^c(t\vee t'-s) \bar{S}(s) \bar{\mathfrak{I}}(s) ds, \non\\
\Cov(\hat{I}_1(t), \hat{I}_1(t')) & =  \int_0^{t\wedge t'}\int_0^{t\wedge t'-s} F^c(t\vee t'-s-u|u)dG(u)  \bar{S}(s) \bar{\mathfrak{I}}(s) ds,  \non\\
\Cov(\hat{R}_1(t), \hat{R}_1(t')) & =  \int_0^{t\wedge t'} \Phi(t\wedge t'-s) \bar{S}(s) \bar{\mathfrak{I}}(s) ds, \non\\
\Cov(\hat{E}_1(t), \hat{I}_1(t')) & = \bone(t'\ge t)  \int_0^{t\wedge t'}  \int_{t-s}^{t'-s} F^c(t'-s-u|u) d G(u)  \bar{S}(s) \bar{\mathfrak{I}}(s) ds,\non\\
\Cov(\hat{E}_1(t), \hat{R}_1(t')) & = \bone(t'\ge t)  \int_0^{t\wedge t'}  \int_{t-s}^{t'-s} F(t'-s-u|u) d G(u) \bar{S}(s) \bar{\mathfrak{I}}(s)) ds,\non\\
\Cov(\hat{I}_1(t), \hat{R}_1(t')) & =  \bone(t'\ge t)  \int_0^{t\wedge t'} \int_0^{t-s}  (F(t'-s-u|u) - F(t-s-u|u)) d G(u) \bar{S}(s) \bar{\mathfrak{I}}(s) ds. \non 
\end{align}
 
\end{lemma}

\begin{proof}
The convergence of $\big(\hat{E}^N_1,\hat{I}^N_1, \hat{R}^N_1\big) $ follows from the same argument as in the proofs of Lemmas 8.3 and 8.4 in \cite{PP-2020}. Let us sketch the argument, and in particular the coupling between $\big(\hat{E}^N_1,\hat{I}^N_1, \hat{R}^N_1\big) $ and $\big(\hat{\mathfrak{I}}^N_1, \hat{\mathfrak{I}}^N_2\big)$, referring the reader to \cite{PP-2020} for the technical details.
We have shown the joint convergence of $\big(\hat{\mathfrak{I}}^N_1, \hat{\mathfrak{I}}^N_2\big)$ in \eqref{eqn-hatMA-frakI-joint-conv} using the processes $\big(\wt{\mathfrak{I}}^N_1, \wt{\mathfrak{I}}^N_2\big)$ which are defined via the PRM $\breve{Q}(ds,d\lambda,du)$. 
Recall the definition of the pair $(\zeta,\eta)$ as a function of $\lambda$ in \eqref{eqn-lambda-eta}. This defines a map $\phi:\bD\mapsto\R_+^2$ such that $(\zeta,\eta)=\phi(\lambda)$. Recall that $\zeta$ is the duration of the exposed period, and $\eta$ the duration of the infectious period. We now define $\Lambda:\R_+\times\bD\times\R_+\mapsto\R_+^4$ by
\[ \Lambda(s,\lambda,u)=(s,\phi^{(s)}(\lambda),u),\ \text{where }\phi^{(s)}(\lambda)=(s+\phi_1(\lambda),s+\phi_1(\lambda)+\phi_2(\lambda))\,.\]
The intuition behind this definition is that if $s$ is the time of infection,  $\phi^{(s)}_1$ stands for the end of the exposed period (transition from exposed to infectious), and $\phi^{(s)}_2$ the end of the infectious period (transition from infectious to recovered). We finally define a PRM $Q_1$ on $\R_+^4$, which is the image of $\breve{Q}$ by the mapping $\Lambda$, i.e., for any Borel subset $A\subset\R_+^4$,
\[   Q_1(A)=\breve{Q}(\Lambda^{-1}(A))\,.\]
 Let $\wt{Q}_1$ denote its associated compensated measure. We have
\begin{align*}
\hat{E}^N_1(t)&=\frac{1}{\sqrt{N}}\int_0^t\int_t^\infty\int_0^\infty\int_0^\infty \bone_{u\le\Upsilon^N(s^-)}\wt{Q}_1(ds,dy,dz,du),\\
\hat{I}^N_1(t)&=\frac{1}{\sqrt{N}}\int_0^t\int_0^t\int_t^\infty\int_0^\infty \bone_{u\le\Upsilon^N(s^-)}\wt{Q}_1(ds,dy,dz,du),\\
\hat{R}^N_1(t)&=\frac{1}{\sqrt{N}}\int_0^t\int_0^t\int_0^t\int_0^\infty \bone_{u\le\Upsilon^N(s^-)}\wt{Q}_1(ds,dy,dz,du)\,.
\end{align*}
Recall that
\begin{align*}
\hat{M}^N_A(t)=\frac{1}{\sqrt{N}}\int_0^t\int_0^\infty\int_0^\infty\int_0^\infty\bone_{u\le\Upsilon^N(s^-)}\wt{Q}_1(ds,dy,dz,du),
\end{align*}
and define the auxiliary process
\begin{align*}
\hat{L}^N_1(t)&=\frac{1}{\sqrt{N}}\int_0^t\int_0^t\int_0^\infty\int_0^\infty \bone_{u\le\Upsilon^N(s^-)}\wt{Q}_1(ds,dy,dz,du)\,.
\end{align*}

We define $\big(\wt{E}^N_1,\wt{I}^N_1, \wt{R}^N_1,\wt{M}^N_A,\wt{L}^N_1\big) $ by replacing $\bone_{u\le\Upsilon^N(s^-)}$ by $\bone_{u\le N\bar{\Upsilon}(s^-)}$ in the above integrals.  It is not hard to show that the three processes $\wt{M}^N_A(t)$, $\wt{L}^N_1(t)$ and $\wt{R}^N_1(t)$ are martingales (with respect to three different filtrations), whose tightness is easy to establish. Moreover, 
\begin{align} \label{eqn-wt-EI-diff-rep}
\wt{E}^N_1(t)=\wt{M}^N_A(t)-\wt{L}^N_1(t), \quad \wt{I}^N_1(t)=\wt{L}^N_1(t)-\wt{R}^N_1(t).
\end{align} Hence the tightness of $\wt{E}^N_1$ and $\wt{I}^N_1$.
In order to establish the joint convergence $\big(\wt{\mathfrak{I}}^N_1, \wt{\mathfrak{I}}^N_2, \wt{E}^N_1,\wt{I}^N_1, \wt{R}^N_1\big) \RA \big(\hat{\mathfrak{I}}_1, \hat{\mathfrak{I}}_2, \hat{E}_1,\hat{I}_1, \hat{R}_1\big)$  in $\bD^5$ as $N\to \infty$, it remains to prove the joint convergence of the finite-dimensional distributions, that is,
for any $k\ge1$, $0<t_1<t_2<\cdots<t_k$,  as $N\to\infty$, $
 \big((\wt{\mathfrak{I}}^N_1(t_1),\wt{\mathfrak{I}}^N_2(t_1), \wt{E}^N_1(t_1),\wt{I}^N_1(t_1), \wt{R}^N_1(t_1)),\ldots,(\wt{\mathfrak{I}}^N_1(t_k),\wt{\mathfrak{I}}^N_2(t_k), \wt{E}^N_1(t_k),\wt{I}^N_1(t_k), \wt{R}^N_1(t_k))\big)$ converges  to $
 \big((\hat{\mathfrak{I}}_1(t_1),\hat{\mathfrak{I}}_2(t_1), \hat{E}_1(t_1),\hat{I}_1(t_1), \hat{R}_1(t_1)),\ldots,(\hat{\mathfrak{I}}_1(t_k),\hat{\mathfrak{I}}_2(t_k), \hat{E}_1(t_k),\hat{I}_1(t_k), \hat{R}_1(t_k)) \big)$ in distribution in $\R^{5k}$. Noting that an integral w.r.t. the compensated measure $\wt{Q}_1$ can be considered in fact as an integral w.r.t. the compensated measure $\wt{Q}$, it is clear that the computation done in the proof of Lemma 
 \ref{lem-hatmfI1-2} can be extended to this situation, yielding the announced result.  
 We just show one piece of this computation, namely, 
 for $\vartheta, \theta, \vartheta', \theta' \in \R$ and $t,t'>0$, 
we get
\begin{align*}
& \lim_{N\to\infty}  \E\left[\exp\left(i\vartheta\wt{\mathfrak{I}}^N_1(t)+i\theta\hat{I}^N_1(t)+i\vartheta'\wt{\mathfrak{I}}^N_1(t')+i\theta'\hat{I}^N_1(t')\right)\right] \\
& = \exp\bigg(-\frac{\vartheta^2}{2}\int_0^{t}\E\big[ \breve\lambda^2(t-s)   \big]\bar{\Upsilon}(s)ds - \frac{\theta^2}{2}\int_0^{t} \Psi(t-s) \bar{\Upsilon}(s)ds \\
& \qquad \qquad  -\frac{(\vartheta')^2}{2}\int_0^{t'}\E\big[ \breve\lambda^2(t'-s)   \big]\bar{\Upsilon}(s)ds - \frac{(\theta')^2}{2}\int_0^{t'} \Psi(t'-s) \bar{\Upsilon}(s)ds  \\
& \qquad \qquad - \vartheta\theta  \int_0^{t}  \Big( \E \big[\lambda(t-s) \bone_{\zeta\le t-s<\zeta+\eta}\big] - \bar\lambda(t-s) \Psi(t-s) \Big) \bar{\Upsilon}(s)ds  \\
&\qquad \qquad - \vartheta\vartheta'\int_0^{t\wedge t'}\E\big[ \breve\lambda(t-s)  \breve\lambda(t'-s) \big]\bar{\Upsilon}(s)ds  \\
& \qquad \qquad -\theta \theta'   \int_0^{t\wedge t'}\int_0^{t\wedge t'-s} F^c(t\vee t'-s-u|u)G(du)  \bar{S}(s) \bar{\mathfrak{I}}(s) ds   \non\\
& \qquad \qquad - \vartheta\theta'  \int_0^{t \wedge t'}  \Big( \E \big[\lambda(t-s) \bone_{\zeta\le t'-s<\zeta+\eta}\big] - \bar\lambda(t-s) \Psi(t'-s) \Big) \bar{\Upsilon}(s)ds  \\
& \qquad \qquad - \vartheta' \theta  \int_0^{t \wedge t'}  \Big( \E \big[\lambda(t'-s) \bone_{\zeta\le t-s<\zeta+\eta}\big] - \bar\lambda(t'-s) \Psi(t-s) \Big)\bar{\Upsilon}(s)ds \\
& \qquad \qquad - \vartheta'\theta'  \int_0^{t'}  \Big( \E \big[\lambda(t'-s) \bone_{\zeta\le t'-s<\zeta+\eta}\big] - \bar\lambda(t'-s) \Psi(t'-s) \Big) \bar{\Upsilon}(s)ds  \bigg)
\end{align*}

Finally, we need to show that  $\big(\hat{E}^N_1-\wt{E}^N_1,\hat{I}^N_1-\wt{I}^N_1, \hat{R}^N_1-\wt{R}^N_1\big) \to0$ in $\bD^3$ in probability. 
Since $\hat{E}^N_1(t)=\hat{M}^N_A(t)-\hat{L}^N_1(t)$ and $\hat{I}^N_1(t)=\hat{L}^N_1(t)-\hat{R}^N_1(t)$, given \eqref{eqn-wt-EI-diff-rep},  it suffices to show that  $\big(\hat{M}^N_A-\wt{M}^N_A,\hat{L}^N_1-\wt{L}^N_1, \hat{R}^N_1-\wt{R}^N_1\big) \to0$ in $\bD^3$ in probability. 
We will show that $\hat{R}^N_1-\wt{R}^N_1\to 0$ and the other two follow from similar arguments. 
It is clear that
\begin{align*}
\E \big[\big|\hat{R}^N_1(t)-\wt{R}^N_1(t)\big|^2\big] = \int_0^t \Phi(t-s) \E\big[\big|\bar\Upsilon^N(s^-) - \bar\Upsilon(s)\big|\big] ds \to 0 \qasq N \to \infty. 
\end{align*}
We have
\begin{align*}
\big|\hat{R}^N_1(t)-\wt{R}^N_1(t) \big|&\le \frac{1}{\sqrt{N}}\int_0^t\int_0^t\int_0^t\int_{N (\bar\Upsilon^N(s^-) \wedge \bar\Upsilon(s))}^{N (\bar\Upsilon^N(s^-) \vee \bar\Upsilon(s))} Q_1(ds,dy,dz,du)  +  \int_0^t \Phi(t-s) |\hat\Upsilon^N(s)|ds\,. 
\end{align*} 
Denote these two terms as $\mathfrak{R}_1^N(t)$ and $\mathfrak{R}_2^N(t)$. We show tightness of these two processes. Observe that both are increasing in $t$. 
 Thus, we only need to verify the following (see the Corollary on page 83 in \cite{billingsley1999convergence}): 
for any $\ep>0$, and $i=1,2$, 
\begin{equation}\label{mfR-n-tight}
\limsup_{N\to\infty} \frac{1}{\delta}\P\big( \big|\mathfrak{R}_i^N(t+\delta) - \mathfrak{R}_i^N(t) \big| \ge \ep \big) \to 0\qasq \delta \to 0.
\end{equation}
By some direct calculations (see also the similar derivations in the proofs of Lemmas 6.3 and 8.3 in \cite{PP-2020} for the models with constant infectivity rate), this condition for both processes $\mathfrak{R}_1^N(t)$ and $\mathfrak{R}_2^N(t)$  reduces to show that 
\begin{equation}\label{mfR-n-tight-p1}
\limsup_{N\to\infty} \frac{1}{\delta}  \E\left[ \left(\int_t^{t+\delta}  
\Phi(t+\delta -s)|\hat\Upsilon^N(s)|  ds \right)^2 \right] \to 0  
\end{equation}
and 
\begin{equation}\label{mfR-n-tight-p2}
\limsup_{N\to\infty} \frac{1}{\delta}  \E\left[ \left(\int_0^t 
(\Phi(t+\delta -s) - \Phi(t-s))|\hat\Upsilon^N(s)|  ds \right)^2 \right] \to 0  
\end{equation}
as $\delta \to 0$.  It is clear the expectation in \eqref{mfR-n-tight-p1} is bounded by $\delta^2 \sup_{0\le s\le T} \E\big[ |\hat\Upsilon^N(s)|^2\big]$, so the claim follows by \eqref{eqn-hatPhi-2-boundedness-sup}. 
For \eqref{mfR-n-tight-p2}, we first observe that
\begin{align}
 & \Phi(t+\delta -s) - \Phi(t-s)  \non\\
 &= \int_0^{t+\delta-s} F(t+\delta-s-u|u)G(du)  -  \int_0^{t-s} F(t-s-u|u)G(du)\non\\
& = \int_{t-s}^{t+\delta-s} F(t+\delta-s-u|u)G(du)  +  \int_0^{t-s} (F(t+\delta-s-u|u)- F(t-s-u|u))G(du). \non
\end{align}
Thus, we have
\begin{align}
&\E\left[ \left(\int_0^t 
(\Phi(t+\delta -s) - \Phi(t-s))  |\hat\Upsilon^N(s)| ds \right)^2 \right] \non\\
& \le 2 \E\bigg[ \sup_{0\le s\le T}|\hat\Upsilon^N(s)|^2\bigg] \Bigg(   \left(\int_0^t 
(G(t+\delta-s) - G(t-s))  ds \right)^2 \non\\
& \qquad \qquad +  \left(\int_0^t 
\int_0^{t-s} (F(t+\delta-s-u|u)- F(t-s-u|u))dG(u) ds \right)^2 \Bigg) \non \\
& \le 2 \delta^2  \E\bigg[ \sup_{0\le s\le T}|\hat\Upsilon^N(s)|^2\bigg], \non 
\end{align}  
where the last inequality follows from applying \eqref{estim-int-F} (noting that the second integral requires using interchange of the order of integration first).  Thus by  \eqref{eqn-hatPhi-2-boundedness-sup-inside}, we obtain \eqref{mfR-n-tight-p2}. This completes the proof.
 \end{proof}

\begin{remark}
We remark that
the limit
$ (\hat{E}_1, \hat{I}_1,\hat{R}_1)$
 can be written as
\begin{align} \label{eqn-hatI-1-WH}
\hat{E}_1(t) &= W_1([0,t]\times[t,\infty)\times [0,\infty)), \non\\
 \hat{I}_1(t) &= W_1([0,t]\times[0,t)\times[t,\infty)), \non\\
\hat{R}_1(t) &= W_1([0,t]\times[0,t)\times [0,t)),   
\end{align}
where  $W_1$ is
a Gaussian white noise on $\RR_+^3$, independent of the pair $(W_{0,1},W_{0,2})$,  such that 
\begin{align*}
& \E\left[ W_1([r,t)\times [a,b)\times [a',b'))^2\right] \non\\
& =   \int_r^t \left( \int_{a-s}^{b-s} (F(b'-y-s|y) - F(a'-y-s|y)) G(dy) \right)  \bar{S}(s) \bar{\mathfrak{I}}(s) ds,
\end{align*}
for $0 \le r \le t$,  $0 \le a \le b$ and $0 \le a' \le b'$. 
\end{remark}

{\bf Completing the proof of Theorem \ref{thm-FCLT}.} 
The representations of $E^N(t)$, $I^N(t)$ and $R^N(t)$ in \eqref{eqn-hatE-rep-SEIR}, \eqref{eqn-hatI-rep-SEIR} and \eqref{eqn-hatR-rep-SEIR},  give a natural integral mapping from $(\hat{E}^N(0), \hat{I}^N(0))$,  $\big( \hat{E}^N_0, \hat{I}^N_{0,1},  \hat{I}^N_{0,2}, \hat{R}^N_{0,1},  \hat{R}^N_{0,2}\big)$, $\big(\hat{E}^N_1,\hat{I}^N_1, \hat{R}^N_1\big)$, $\big(\hat{S}^N, \hat{\mathfrak{J}}^N\big)$  and the fixed functions $G_0(t), F_{0,I}(t), \Psi_0(t)$, 
$\bar{\lambda}^0(t)$, $\bar{\lambda}^{0,I}(t)$,  $\bar{\lambda}(t)$, $\bar{S}(t)$ and $\bar{\mathfrak{I}}(t)$ to the processes 
$ \big(\hat{E}^N, \hat{I}^N,\hat{R}^N\big)$. The mapping is continuous in the Skorohod $J_1$ topology (by a slight modification of Lemma 8.1 in \cite{PP-2020}). We can then apply the continuous mapping theorem to conclude the convergence $\big(\hat{E}^N, \hat{I}^N, \hat{R}^N\big) \RA \big(\hat{E}, \hat{I}, \hat{R}\big)$ in $\bD^3$.  Given their joint convergence with $\big( \hat{\mathfrak{I}}^N_{0,1},  \hat{\mathfrak{I}}^N_{0,2} \big)$ in Lemma \ref{lem-3.11} and  $ (\hat{\mathfrak{I}}^N_1,\hat{\mathfrak{I}}^N_2)$ in Lemma  \ref{lem-3.12}, we can also conclude the joint convergence of the processes $\big(\hat{E}^N, \hat{I}^N, \hat{R}^N\big)$ with $(\hat{S}^N, \hat{\mathfrak{I}}^N)$.

Finally, for the covariance functions of $(\hat{W}_S,\hat{W}_{\mathfrak{I}},\hat{W}_E, \hat{W}_I, \hat{W}_R)$, 
we have observed in the previous section from  \eqref{eqn-hatfrakI-rep} that 
$\hat{W}_{\mathfrak{I}} = \hat{\mathfrak{I}}_{0,1}(t)+ \hat{\mathfrak{I}}_{0,2}(t)+ \hat{\mathfrak{I}}_1(t) +\hat{\mathfrak{I}}_2(t)$, and we also observe 
 from \eqref{eqn-hatE-rep-SEIR}, \eqref{eqn-hatI-rep-SEIR} and \eqref{eqn-hatR-rep-SEIR}
that $\hat{W}_E = \hat{E}_0(t) + \hat{E}_1(t)$, $\hat{W}_I=\hat{I}_{0,1}(t)+  \hat{I}_{0,2}(t)  + \hat{I}_1(t) $, $\hat{W}_R=\hat{R}_{0,1}(t)+  \hat{R}_{0,2}(t)  + \hat{R}_1(t)$. Then the covariance functions can be easily computed from the 
the covariance functions of $( \hat{\mathfrak{I}}_{0,1},\hat{I}_{0,1},\hat{R}_{0,1})$, 
$(\hat{\mathfrak{I}}_{0,2},\hat{E}_0, \hat{I}_{0,2}, \hat{R}_{0,2})$ 
and $(\hat{\mathfrak{I}}_1,\hat{\mathfrak{I}}_2,\hat{E}_1, \hat{I}_1,\hat{R}_1)$  in Lemmas \ref{lem-3.11} and \ref{lem-3.12}. 
This completes the proof of Theorem \ref{thm-FCLT}.

\section{On the generalized SIS and SIRS models with varying infectivity}

In this section, we discuss how the results can be generalized to the SIS and SIRS models. We state the FCLTs for these models without proofs, since they can be done analogously.

\subsection{Generalized SIS models with varying infectivity} \label{sec-SIS}

In the SIS model, individuals become  susceptible immediately after going through the infectious periods. 
Since $S^N(t)=N-I^N(t)$, the epidemic dynamics is determined by the two dimensional processes $\mathfrak{I}^N(t)$ and $I^N(t)$. 
As stated in Remark 2.3 of \cite{FPP2020b}, the FLLN limit $(\bar{\mathfrak{I}}(t), \bar{I}(t))$ is 
determined by the two--dimensional integral equations:
  \begin{align} 
\bar{\mathfrak{I}}(t)&=\bar{I}(0)\bar{\lambda}^0(t)+\int_0^t\bar{\lambda}(t-s)(1-\bar{I}(s))\bar{\mathfrak{I}}(s)ds\,,\label{eqn-barfrakI-SIS}\\
 \bar{I}(t) &= \bar{I}(0) F_{0,I}^c(t)   +\int_0^t F^c(t-s) (1-\bar{I}(s))\bar{\mathfrak{I}}(s)ds\,. \label{eqn-barI-SIS}
 \end{align}
Here the c.d.f.'s $F$  and $F_{0,I}$ denote the distributions of the infectious periods of newly infected individuals and those of initially infectious ones.

\begin{theorem} \label{thm-FCLT-SIS}
In the generalized SIS model, under Assumptions \ref{AS-lambda} and \ref{AS-FCLT} (with $E^N(t)\equiv 0$ and only infectious periods), 
\begin{align*}
\big( \hat{\mathfrak{I}}^N, \hat{I}^N\big) 
\to \big(\hat{\mathfrak{I}}, \hat{I}\big) \qinq \bD^2 \qasq N \to \infty. 
\end{align*}
The limit processes $\hat{\mathfrak{I}}$ and $\hat{I}$ are the unique solution to the following stochastic integral equations: 
\begin{align*}
\hat{\mathfrak{I}}(t) &= \hat{I}(0) \bar{\lambda}^{0,I}(t) +  \hat{W}_{\mathfrak{I}}(t) + \int_0^t \bar{\lambda}(t-s) \hat{\Upsilon}(s) ds, \\
\hat{I}(t) &= \hat{I}(0) F_{0,I}^c(t) +  \hat{W}_{I}(t)   + \int_0^t F^c(t-s)  \hat{\Upsilon}(s) ds, 
\end{align*}
where
\begin{align*}
\hat{\Upsilon}(t)  =  (1-\bar{I}(t))  \hat{\mathfrak{I}}(t) - \bar{\mathfrak{I}}(t) \hat{I}(t) ,
\end{align*}
 and $ \bar{\mathfrak{I}}(t)$ and  $\bar{I}(t)$ are given by the unique solutions to the integral equations  \eqref{eqn-barfrakI-SIS} and \eqref{eqn-barI-SIS}.
 $( \hat{W}_{\mathfrak{I}}, \hat{W}_{I})$
 is a two-dimensional centered Gaussian process, independent of $\hat{I}(0)$, where the covariance functions are as given for the processes  $( \hat{W}_{\mathfrak{I}}, \hat{W}_{I})$ in the SIR model with $\bar{S}=1-\bar{I}$ in the corresponding formulas. 
If $\bar{\lambda}^{0,I}(t)$ and $F_{0,I}$ are continuous, then the limits $\hat{\mathfrak{I}}$ and $\hat{I}$  are continuous.
\end{theorem}

\subsection{Generalized SIRS models with varying infectivity} \label{sec-SIRS} 

In the SIRS model, individuals experience the infectious and recovered/immune periods, and then become susceptible.  
We use $I^N(t)$ and $R^N(t)$ to denote the numbers of infectious and recovered/immune individuals, respectively, corresponding to the processes $E^N(t)$ and $I^N(t)$ in the SEIR model. 
Similarly, we use $\{\lambda^0_j\}_{j \ge 1}$ and $\{\lambda_i\}_{i\ge 1}$ to denote the infectivity processes of the initially and newly infectious individuals, respectively, and also use $\{(\xi^0_j,\eta^0_j)\}_{j\ge 1}$ and $\{(\xi_i, \eta_i)\}_{i \ge 1}$  for the infectious and recovered/immune periods for the initially and newly infected individuals, respectively. 
Denote the  remaining immune time of the initially recovered/immune individuals by $\eta^{0,R}_k$ (changing notation $\eta^{0,I}$ to $\eta^{0,R}$ accordingly). Of course,  $\{\lambda^0_j\}_{j \ge 1}$ and $\{\lambda_i\}_{i\ge 1}$ only take positive values over the intervals $[0, \zeta^0_j)$ and $[0, \zeta_i)$, respectively. The definitions of the variables $(\xi_i, \eta_i)$, $(\xi^0_j,\eta^0_j)$ and  $\eta^{0,R}_k$ in \eqref{eqn-lambda-eta}, \eqref{eqn-lambda-eta-0} and \eqref{eqn-lambda-eta-0I} also need to change accordingly in a natural manner. The c.d.f.'s $G_0, G$ denote the distributions of infectious periods of the initially and newly infectious individuals, and the c.d.f.'s $F_{0,R}$ and $F$ denote the distributions of the recovered/immune periods of the initially and newly recovered individuals. Similarly for the notation $\Psi_0,\Psi, \Phi_0, \Phi$. 

Since $S^N(t) = N- I^N(t) - R^N(t)$, the epidemic dynamics is described by the three processes $(\mathfrak{I}^N, I^N, R^N)$. As stated in Remark 2.3 in \cite{FPP2020b}, the FLLN limit $\big(\bar{\mathfrak{I}}, \bar{I}, \bar{R}\big)$ is determined by the three--dimensional integral equations: 
\begin{align} 
\bar{\mathfrak{I}}(t)&=\bar{I}(0)\bar{\lambda}^0(t)+\int_0^t\bar{\lambda}(t-s)\big(1- \bar{I}(s) - \bar{R}(s)\big)\bar{\mathfrak{I}}(s)ds\,,\label{eqn-barfrakI-SIRS}\\
 \bar{I}(t) &=\bar{I}(0)G_0^c(t)+\int_0^tG^c(t-s)\big(1- \bar{I}(s) - \bar{R}(s)\big) \bar{\mathfrak{I}}(s)ds\,, \label{eqn-barI-SIRS}\\
 \bar{R}(t) &= \bar{R}(0) F_{0,R}^c(t) + \bar{I}(0)\Psi_0(t)  +\int_0^t \Psi(t-s) \big(1- \bar{I}(s) - \bar{R}(s)\big) \bar{\mathfrak{I}}(s)ds\,.\label{eqn-barR-SIRS}
\end{align}
We next state the FCLT for these processes.

\begin{theorem}\label{thm-FCLT-SIRS}
In the generalized SIRS model, under Assumptions \ref{AS-lambda} and \ref{AS-FCLT}, 
\begin{align*}
\big( \hat{\mathfrak{I}}^N, \hat{I}^N,  \hat{R}^N\big) 
\RA \big(\hat{\mathfrak{I}}, \hat{I},  \hat{R}\big) \qinq \bD^3 \qasq N \to \infty. 
\end{align*}
The limit process $(\hat{\mathfrak{I}},\hat{I},\hat{R})$ is the unique solution to the following system of stochastic integral equations: 
\begin{align*}
\hat{\mathfrak{I}}(t) &= \hat{I}(0) \bar{\lambda}^{0}(t) +  \hat{W}_{\mathfrak{I}}(t)  + \int_0^t \bar{\lambda}(t-s) \hat{\Upsilon}(s) ds, \\
\hat{I}(t) &= \hat{I}(0) G_{0}^c(t) +  \hat{W}_{I}(t)+ \int_0^t G^c(t-s)  \hat{\Upsilon}(s) ds,\\
\hat{R}(t) &= \hat{R}(0) F_{0,R}^c(t) + \hat{I}(0) \Psi_0(t) +  \hat{W}_{R}(t) + \int_0^t \Psi(t-s)  \hat{\Upsilon}(s) ds, 
\end{align*}
where
\begin{align*}
\hat{\Upsilon}(t)  =  (1-\bar{I}(t) -\bar{R}(t))  \hat{\mathfrak{I}}(t) - \bar{\mathfrak{I}}(t) (\hat{I}(t) + \hat{R}(t)),
\end{align*}
 and $ \bar{\mathfrak{I}}(t)$, $\bar{I}(t)$ and  $\bar{R}(t)$ are given by the unique solution to the integral equations  \eqref{eqn-barfrakI-SIRS}--\eqref{eqn-barR-SIRS}. 
 $(\hat{W}_{\mathfrak{I}}, \hat{W}_{I}, \hat{W}_{R})$ is a three-dimensional centered Gaussian process, independent of $\hat{I}(0)$ and $\hat{R}(0)$, which  has the covariance functions as the processes  $(\hat{W}_{\mathfrak{I}}, \hat{W}_{E}, \hat{W}_{I})$  in the SEIR model with $\bar{S}=1-\bar{I}-\bar{R}$ in the corresponding formulas. 
  If $\bar\lambda^0(t)$, $G_0$ and $F_{0,R}$ are continuous, the limits $\hat{\mathfrak{I}}$, $\hat{I}$ and $\hat{R}$  have continuous paths. 
\end{theorem}

\section{Appendix}
We recall Theorem 13.5 from \cite{billingsley1999convergence}. In the cited reference, the result states conditions which imply  convergence in $D([0,1];\R)$, while we want to obtain convergence in $\bD:=D([0,+\infty);\R)$.
This implies that in our case condition (13.12) in \cite{billingsley1999convergence} is not required, and the above quoted theorem becomes the 
following.
\begin{theorem}\label{13.5}
Let $\{X^N,\ N\ge1\}$ be a sequence of random processes with trajectories in $\bD$. $X$ being a process with trajectories in $\bD$, suppose that the two following conditions are satisfied.
\begin{equation}\label{finitedimconv}
\text{For any }k\ge1, t_1,t_2,\ldots,t_k\ge0, \ (X^N_{t_1},X^N_{t_2},\ldots,X^N_{t_k})\Rightarrow(X_{t_1},X_{t_2},\ldots,X_{t_k}),
\end{equation}
and there exists a continuous nondecreasing function $F$ from $\R_+$ into itself, $\alpha>1/2$ and $\beta\ge0$ such that for any $0\le r\le s\le t$, $N\ge1$ and $\lambda>0$, 
\begin{equation}\label{3points}
\P[|X^N_s-X^N_r|\wedge|X^N_t-X^N_s|\ge\lambda]\le \lambda^{-4\beta}[F(t)-F(r)]^{2\alpha}\,.
\end{equation}
Then $X^N\Rightarrow X$ in $\bD$.
\end{theorem}
As explained in \cite{billingsley1999convergence}, the following moment condition is stronger than \eqref{3points}
\begin{equation}\label{moment3points}
\E[|X^N_s-X^N_r|^{2\beta}|X^N_t-X^N_s|^{2\beta}]\le[F(t)-F(r)]^{2\alpha}\,.
\end{equation}

\bibliographystyle{plain}
\bibliography{Epidemic-Age}

\end{document}